\pdfoutput=1
\documentclass[11pt,reqno]{amsart}

\usepackage[T1]{fontenc}
\usepackage{amsmath}								
\usepackage{amssymb}
\usepackage{amsthm}
\usepackage{amscd}
\usepackage{amsfonts}
\usepackage{mathtools}
\usepackage{mathptmx}
\usepackage[scaled]{helvet}
\usepackage{microtype,comment}
\usepackage[all]{xy}
\usepackage{euler}
\usepackage{extarrows}
\usepackage[colorlinks, citecolor = blue, linkcolor = blue]{hyperref}
\usepackage[dvipsnames]{xcolor}
\usepackage{blkarray}
\usepackage{tikz}									
\usetikzlibrary{matrix}
\usetikzlibrary{patterns}
\usetikzlibrary{matrix}
\usetikzlibrary{positioning}
\usetikzlibrary{decorations.pathmorphing}
\usetikzlibrary{cd}
\usetikzlibrary{intersections,calc}
\tikzset{dot/.style={circle, fill=black, inner sep=.05cm}}
\usepackage{dsfont}
\usepackage[left=3.5cm,top=3.5cm,right=3.5cm]{geometry}
\usepackage[shortlabels]{enumitem}

\newtheorem{theorem}{Theorem}[section]
\newtheorem{lemma}[theorem]{Lemma}
\newtheorem{proposition}[theorem]{Proposition}

\theoremstyle{definition}
\newtheorem{definition}[theorem]{Definition}

\newtheorem{convention}[theorem]{Convention}
\newtheorem{claim}[theorem]{Claim}

\newcommand{\Z}{\mathds{Z}}

\newcommand{\bP}{\mathds{P}}

\newcommand{\calM}{\mathcal{M}}

\newcommand{\Mbar}{\overline{\calM}}

\newcommand\trop{{\mathtt{trop}}}

\newcommand{\dg}{\mathrm{deg}}

\DeclareMathOperator{\Aut}{\mathrm Aut}

\DeclareMathOperator{\lcm}{lcm}

\DeclareMathOperator{\val}{val}

\newcommand\Mgn[1]{\mathcal M_{#1}}
\newcommand\Mgnbar[1]{\Mbar_{#1}}

\newcommand\Hur[1]{\mathcal H_{#1}}
\newcommand\Adm[1]{ \mathcal{ \overline{H}}_{#1}}

\newcommand{\Ft}{\mathsf{F_t}}
\newcommand{\Fs}{\mathsf{F_s}}

\newcommand{\ft}{\mathsf{f_t}}
\newcommand{\fs}{\mathsf{f_s}}

\newcommand{\gGamma}{\widehat{\Gamma}}
\newcommand{\gT}{\widehat{T}}

\newcommand{\Tev}{{\mathsf{Tev}}}

\newcommand{\mysetminusD}{\hbox{\tikz{\draw[line width=0.6pt,line cap=round] (3pt,0) -- (0,6pt);}}}
\newcommand{\mysetminusT}{\mysetminusD}
\newcommand{\mysetminusS}{\hbox{\tikz{\draw[line width=0.45pt,line cap=round] (2pt,0) -- (0,4pt);}}}
\newcommand{\mysetminusSS}{\hbox{\tikz{\draw[line width=0.4pt,line cap=round] (1.5pt,0) -- (0,3pt);}}}
\newcommand{\mysetminus}{\mathbin{\mathchoice{\mysetminusD}{\mysetminusT}{\mysetminusS}{\mysetminusSS}}}
\renewcommand\setminus\mysetminus
\renewcommand\smallsetminus\mysetminus

\title{Generalizations of Tropical Tevelev Degrees}
\author{Erin Dawson}
\address{Universit\"at T\"ubingen, Fachbereich Mathematik, Auf der Morgenstelle 10, 72076 T\"ubingen, Germany}
\email{\href{mailto:erin.dawson@math.uni-tuebingen.de}{erin.dawson@math.uni-tuebingen.de}}

\begin{document}

\maketitle
\begin{abstract}
  We study tropical Tevelev degrees arising from maps between certain tropical moduli spaces of curves. Building on work of Cavalieri and Dawson, who defined and computed tropical Tevelev degrees in the case of degree $d = g+1$ and $n = g+3$ marked points, we extend the theory by introducing an additional integer parameter $\ell$. In our framework the curve degree and number of marked points vary as $d = g + 1 + \ell$ and $n = g + 3 + 2\ell$, and we analyze the resulting tropical Tevelev degrees for both positive and negative values of $\ell$. This tropicalizes results of Cela, Pandharipande, and Schmitt on algebraic Tevelev degrees. 
  We then further broaden the framework by introducing generalized tropical Tevelev degrees, providing the tropical counterpart to the generalized Tevelev degrees studied by Cela and Lian. These results establish a wider set of computational and structural patterns for intersection calculations on tropical moduli spaces and reveal new behavior beyond the classical setting.
\end{abstract}
\section{Introduction}

Tevelev degrees are enumerative invariants obtained as degrees of
natural morphisms from moduli spaces of maps. In this paper, we study
the case of target $\mathbb{P}^1$, where these invariants admit a
description in terms of Hurwitz spaces of admissible covers. The name was coined by Cela, Pandharipande and Schmitt in \cite{CPS:Tevdeg} when they reformulated the notion previously studied by Tevelev \cite{Tevelev:2020zux}.

More precisely, we denote by $\Adm{g,d,n}$ the admissible cover compactification \cite{HM:ac} of the Hurwitz
space parametrizing isomorphism classes of degree $d$, genus $g$ covers
$C \to \bP^1$ with simple ramification. All ramification points are marked and have distinct images under the map $C \to \bP^1$, and in addition
we mark $n$ unramified points on $C$. The images of the unramified source markings are mutually distinct and disjoint from the branch images.

There is a natural morphism
\begin{equation}
    \label{eq:forgfinitemap}
    \fs \times \ft \colon \Adm{g,d,n} \longrightarrow
    \Mgnbar{g,n} \times \Mgnbar{0,n},
\end{equation}
where $\fs$ forgets the ramification points and all unmarked points of the source curve, and $\ft$ forgets the branch
points of the target.

For integers $d[g,\ell] = g+1+\ell$ and $n[g,\ell] = g+3+2\ell$, the morphism $\fs \times \ft$ is
finite, and both the source and target have dimension $5g + 4\ell$. In this situation, the
Riemann--Hurwitz formula implies that the ramification (equivalently, branch) locus of any
cover $\varphi \colon C \to \bP^1$ consists of exactly $2g + 2d[g,\ell] - 2$ points.

Following \cite[Section~1.2]{CPS:Tevdeg}, the \emph{Tevelev degree} is defined by
\begin{equation}
    \label{def:algTev}
    \Tev_{g,\ell}
    := \frac{\deg(\fs \times \ft)}{(2g + 2d[g,\ell] - 2)!}.
\end{equation}
The factorial in the denominator accounts for the fact that $\deg(\fs \times \ft)$ counts
covers with labeled simple branch points. Dividing by
$(2g + 2d[g,\ell] - 2)!$ removes this labeling and produces an enumerative invariant.


The tropical Tevelev degree is defined analogously in \cite{troptev} as the degree of
the source--target morphism
\[
F_{\mathrm{s}}\times F_{\mathrm{t}}\colon
\mathcal{H}^{\mathrm{trop}}_{g,d,n}
\longrightarrow
\mathcal{M}^{\mathrm{trop}}_{g,n}
\times
\mathcal{M}^{\mathrm{trop}}_{0,n}.
\]
Since $\mathbb{P}^1$ is the only target considered in this paper, we
henceforth suppress it from the notation. Cavalieri and Dawson showed that Tevelev degrees admit a fully combinatorial interpretation: they can be realized as degrees of morphisms between tropical moduli spaces of admissible covers and computed by explicit lattice–counting arguments \cite{troptev}. In particular, for covers of degree $d=g+1$ with $n=g+3$ marked points, the tropical Tevelev degree agrees with its algebraic counterpart and is equal to $2^g$.  Tevelev degrees have also been studied for higher-dimensional target varieties. In particular, Cela and Iribar
L\'opez used logarithmic and tropical methods to study fixed-domain
counts for two-dimensional toric targets \cite{CelaIribarLopez}.

The purpose of this paper is to extend this tropical framework in two directions. First, following the algebraic generalizations studied in \cite{CPS:Tevdeg}, we allow the degree of the covers and number of marked points to vary by introducing an additional integer parameter $\ell$. Concretely, for a fixed genus $g$ we consider covers of degree 
\[d=g+1+\ell\] and with \[n=g+3+2\ell\] marked points. These choices ensure that the relevant morphism between moduli spaces has equal-dimensional source and target, so that a well-defined global degree may still be defined. The resulting invariants are called \textit{tropical Tevelev degrees} and are denoted $\Tev_{g,\ell}^\trop$. Allowing $\ell$ to be positive or negative leads to qualitatively different combinatorial behavior: when $\ell>0$ the curve degree increases while the genus remains unchanged, whereas when $\ell<0$ the degree constraint eliminates entire families of genus configurations that are present in the classical case $\ell=0$.

Second, we introduce \textit{tropical generalized Tevelev degrees}, providing a tropical counterpart to the generalized Tevelev degrees studied algebraically by Cela and Lian in \cite{GenTev}. In this setting, the marked points are equipped with ramification profiles $\mu_1,\dots,\mu_k$. The resulting morphism of tropical moduli spaces again admits a notion of global degree, and we show that these generalized tropical Tevelev degrees recover the expected algebraic generalized Tevelev degrees.

Our main results give explicit, closed-form expressions for these tropical degrees and explain them combinatorially in terms of lattice paths and grids of admissible covers. In particular, we obtain a complete description of how the classical value $2^g$ is preserved or modified as $\ell$ varies and as ramification profiles are introduced.

\subsection{Algebraic generalized Tevelev degrees}\label{sec:gentevdeg}
Cela and Lian considered a natural generalization of the previous problem by equipping marked points with arbitrary ramification profiles \cite{GenTev}. Let $g, \ell$ and $d$ be as previously defined and fix $k\geq0$ vectors of positive integers \[\mu_h=(e_{h,1},\dots,e_{h,r_h})\in\Z_{\geq1}^{r_h}\]
with $r_h\geq1$ for $h=1,\dots,k$. Define 
\begin{equation}\label{eq:3}
    \lvert\mu_h\rvert=\sum_{j=1}^{r_h}e_{h,j}
\end{equation}
and assume 
\begin{equation}\label{eq:4}
    g+3+2\ell=\sum_{h=1}^k\lvert\mu_h\rvert,
\end{equation}
and \begin{equation}\label{eq:5}
    \lvert\mu_h\rvert\leq d[g,\ell]
\end{equation} for all h.

Following Cela and Lian \cite[Definition~2]{GenTev}, let
\[
\Adm{g,\ell,\mu_1,\ldots,\mu_k}
\]
denote the Deligne--Mumford stack of Harris--Mumford admissible
covers with the prescribed marking and ramification data. Its smooth
locus parametrizes degree $d$ covers
\[
\pi\colon C\longrightarrow \mathbb{P}^1
\]
from a smooth curve $C$ of genus $g$, together with marked points
$q_1,\ldots,q_k$ on the target and distinguished marked points
\[
p_{h,j}\in \pi^{-1}(q_h),
\qquad
1\leq h\leq k,\quad 1\leq j\leq r_h,
\]
at which $\pi$ has ramification index $e_{h,j}$. Over the $h$-th marked target point, the complete ramification
profile is
\[
\eta_h=
\bigl(e_{h,1},\ldots,e_{h,r_h},1^{d-|\mu_h|}\bigr),
\]
although only the points corresponding to the entries of $\mu_h$ are
marked. In addition, the
cover has $b$ labeled simple branch points, together with their
corresponding simple ramification points. We refer to
\cite[Definition~2]{GenTev} for the precise stability and admissibility
conditions on boundary objects.

Define $n=\sum_{h=1}^kr_h$, the authors study the map \begin{equation}
    \label{eq:genforgfinitemap}
    \fs\times\ft: \Adm{g,\ell,\mu_1,\dots,\mu_k} \to \Mgnbar{g,n} \times \Mgnbar{0,k},
\end{equation}
where $\fs\times\ft$ acts as before by forgetting ramification and branch points but remembers marked points and is a finite morphism of $3g-3+n+(k-3)$ dimensional spaces. We observe that in this case the cardinality of the unmarked ramification (or equivalently branch) locus of any cover $\varphi:C\to \bP^1$ is $b=2g+2d-2-\sum_{h=1}^k\lvert\mu_h\rvert+n$. In \cite[Section 1.3]{GenTev} the {\it generalized Tevelev degree} is defined as:
\begin{equation}
    \label{def:alggenTev}
    \Tev_{g,\ell,\mu_1,\dots,\mu_k}:= \frac{\deg (\fs\times\ft)}{b!}.
\end{equation}

\subsection{Tropical generalized Tevelev degrees}\label{sec:gentroptev}

We make a definition of tropical generalized Tevelev degrees following the algebraic one from \cite{GenTev}. We consider the tropical version of the morphism from Equation \eqref{eq:genforgfinitemap}: it is a map of weighted cone complexes with integral structures and there exists a notion of local degree and global degree.

\begin{definition}\label{def:gentroptev}
Let $g,\ell$ and $d$ be as previously defined and fix $k\geq3$ vectors of positive integers \[\mu_h=(e_{h,1},\dots,e_{h,r_h})\in\Z_{\geq1}^{r_h}\]
with $r_h\geq1$ for $h=1,\dots,k$. Require that equations \eqref{eq:3},\eqref{eq:4},\eqref{eq:5} all hold.
\begin{definition}
\label{def:tropical-generalized-Tevelev}
Let
\[
\mathcal{H}^{\trop}_{g,\ell,\mu_1,\ldots,\mu_k}
\]
be the moduli space of tropical admissible cover  obtained
using the following conventions:
\begin{enumerate}
    \item The first $k$ ends of the target remain marked and are
    labeled $q_1,\ldots,q_k$.

    \item Over the end $q_h$, the source has $r_h$ distinguished
    marked ends
    \[
    p_{h,1},\ldots,p_{h,r_h}
    \]
    with expansion factors
    \[
    e_{h,1},\ldots,e_{h,r_h},
    \]
    respectively. The remaining $d-|\mu_h|$ ends over $q_h$ have
    expansion factor one and are unmarked.

    \item The remaining $b$ ends of the target have ramification
    profile $(2,1^{d-2})$. These simple branch ends and all of their
    inverse-image ends are unmarked.
\end{enumerate}
The space
$\mathcal{H}^{\trop}_{g,\ell,\mu_1,\ldots,\mu_k}$
is equipped with the integral structures and weights induced from the
tropical admissible-cover space of \cite{CMRadmissible}, with permutation
factors arising from unmarked ends incorporated into the local
Hurwitz numbers as in \cite[Convention~3.2]{troptev}.
\end{definition}

Define $n=\sum_{h=1}^kr_h$, consider the morphism of tropical moduli spaces:
\begin{equation}
    \Fs\times \Ft: \Hur{g,\ell,\mu_1,\dots,\mu_k}^\trop \to \Mgn{g,n}^\trop \times \Mgn{0,k}^\trop.
\end{equation}
We equip this space with the standard tropical admissible-cover weights and integral structures described in Section~\ref{sec:generalized-Hurwitz-space}. We define the {\bf tropical generalized Tevelev degree} to be:
\begin{equation}
    \Tev^\trop_{g,\ell,\mu_1,\dots,\mu_k}:={\dg (\Fs\times \Ft)}
\end{equation}
\end{definition}

\subsection{Statement of results}
Our first result is generalizing the correspondence theorem in \cite[Theorem 1.1]{troptev} showing our tropical constructions recover the algebraic invariants, proven in Section \ref{sec:proofofgencor}.
\begin{theorem}\label{thm:gencorrespondence}
For every admissible choice of $g$, $\ell$, and ramification profiles
$\mu_1,\ldots,\mu_k$, the tropical generalized Tevelev degree agrees
with its algebraic counterpart:
\[\Tev_{g,\ell,\mu_1,\dots,\mu_k}=\Tev_{g,\ell,\mu_1,\dots,\mu_k}^\trop.\]
\end{theorem}

Our next result, proven in Section \ref{sec:ell>0}, shows that increasing the degree and number of marked points does not change the tropical Tevelev degree.
\begin{theorem}\label{thm:ellpositive}
    For any non-negative integers $g$ and $\ell$, 
    \[\Tev_{g,\ell}^\trop = 2^g.\]
\end{theorem}
Section \ref{sec:ell<0} considers the case when $\ell$ is negative. Since $n=g+3+2\ell$ and we require $n\geq3$, then $g \geq -2\ell$. Also, the degree constraint forces the disappearance of certain genus configurations. The resulting deficit from $2^g$ can be described explicitly.
\begin{theorem}\label{thm:ellnegative}
    For any negative integer $\ell$ and $g \geq -2\ell$, 
    \[\Tev^\trop_{g,\ell}= 2^g - \sum_{i=0}^{-\ell-1}\bigg(g-2i+1\bigg)\bigg({g\choose i}-{g \choose i-1}\bigg).\]
\end{theorem}
Both Theorem \ref{thm:ellpositive} and Theorem \ref{thm:ellnegative} are tropical counterparts to algebraic theorems proven by Cela, Pandharipande and Schmitt \cite[Theorem 1 and Theorem 6 when $r=1$]{CPS:Tevdeg}. While the formulas are presented differently, we show in Appendix \ref{appendix} that they are equivalent.

We then extend the theory to arbitrary ramification profiles in Sections \ref{sec:gentroptev} and \ref{sec:higherram}. Let $\mu_1,\dots,\mu_k$ be vectors of positive integers, with $|\mu_h|$ denoting the sum of the entries of $\mu_h$.
\begin{theorem}\label{thm:higherram}
    Let $g\geq 0$, $\ell$ an integer, and $\mu_1,\dots,\mu_k$ be vectors where $\mu_h \in \Z_{\geq1}^{r_h}$ with $r_h\geq1$. We require $k\geq3$ and conditions \eqref{eq:3}, \eqref{eq:4} and \eqref{eq:5} satisfied, then
    \begin{align*}
        \Tev^\trop_{g,\ell,\mu_1,\dots,\mu_k}=&2^g - \sum_{i=0}^{-\ell-1}\bigg(g-2i+1\bigg)\bigg({g\choose i}-{g \choose i-1}\bigg)-\sum_{i=-\ell}^{\lvert\mu_1\rvert-\ell-2}{g\choose i}-\sum_{i=-\ell}^{\lvert\mu_k\rvert-\ell-2}{g\choose i}\\&+\bigg(\lvert\mu_1\rvert+\lvert\mu_k\rvert-2\bigg){g\choose-\ell-1}-\sum_{h=2}^{k-1}\sum_{i=-\ell}^{\lvert\mu_h\rvert-\ell-2}\bigg(\lvert\mu_h\rvert-i-\ell-1\bigg)\bigg({g\choose i}-{g \choose i-1}\bigg).
    \end{align*}

\end{theorem}
The algebraic counterpart to this theorem was proven by Cela and Lian \cite[Theorem 9]{GenTev}. We show that the two different presentations of the formula are equivalent in Appendix \ref{appendix}.

Together, these results establish a broad generalization of tropical Tevelev theory and provide a uniform combinatorial framework for computing intersection numbers on tropical moduli spaces.

\subsection{Acknowledgments}
We are grateful to Renzo Cavalieri, Alessio Cela, Hannah Markwig, and Felix Röhrle for helpful conversations about the project. We thank the anonymous reviewers for their constructive comments and suggestions that improved the manuscript. This work was supported by the Alexander von Humboldt Foundation. The author used OpenAI during the revision of this manuscript to assist with editing the exposition and checking clarity. All suggestions were critically reviewed and verified by the author, who takes full responsibility for the content of the
manuscript.

\section{Background}

\subsection{Tropical Tevelev degrees}

In \cite{troptev}, Cavalieri and Dawson developed a tropical framework for Tevelev
degrees by interpreting them as degrees of a morphism between moduli spaces of tropical
admissible covers and tropical curves. One of the main results of \cite{troptev} is a
correspondence theorem identifying tropical Tevelev degrees with their algebraic
counterparts, together with a purely combinatorial computation showing that the
resulting invariant equals $2^g$ in genus $g$. The tropical approach not only recovers
the classical value, but also provides a transparent combinatorial explanation for it.

The present paper builds directly on the tropical framework of \cite{troptev}. Since the
results rely heavily on the notation, conventions, and proof techniques introduced
there, we begin recalling the tropical definition of
Tevelev degrees and the main results and recalling notation and conventions. This allows us to state and prove
generalizations without requiring prior familiarity with the earlier paper.

\subsubsection{Tropical admissible covers and Tevelev degrees}
We briefly recall the tropical objects needed to define tropical Tevelev degrees.
Background on moduli spaces of tropical curves, tropical admissible covers,
and tropical intersection theory may be found in
\cite{MikhalkinModuli,m:survey,Caporaso,CMRadmissible,macstu:tropgeom}.

\begin{definition}
A \emph{tropical admissible cover} of a rational tropical curve with labeled ends
is a morphism of tropical curves
\[
\phi \colon \Gamma \to T
\]
satisfying the following conditions:
\begin{enumerate}
    \item The target $T$ is a stable tree with labeled ends.

    \item Parameterizing edges by arclength (with $\phi(0_e)=0_{\phi(e)}$), the
    restriction of $\phi$ to an edge $e$ is linear,
    \[
        \phi|_e \colon [0,l_e] \to [0,l_{\phi(e)}],
    \]
    with slope
    \(
        m_e = l_{\phi(e)}/l_e
    \)
    a positive integer. The integer $m_e$ is called the \emph{expansion factor}
    (or \emph{degree}) of the edge $e$.

    \item The morphism $\phi$ is \emph{harmonic}: for any vertex $v \in \Gamma$
    and any two edges $e_1,e_2$ incident to $\phi(v)$ in $T$, we have
    \begin{equation}
        \label{eq:localdegree}
        \sum_{\substack{e \ni v \\ \phi(e)=e_1}} m_e
        =
        \sum_{\substack{e \ni v \\ \phi(e)=e_2}} m_e .
    \end{equation}
    This common value is called the \emph{local degree} of $\phi$ at $v$ and is denoted $d_v$.

    \item The \emph{local Riemann--Hurwitz condition} holds at every vertex $v$ of
    $\Gamma$, namely
    \[
        \val(v) + 2g_v - 2
        =
        d_v \bigl( \val(\phi(v)) - 2 \bigr),
    \]
    where $\val$ denotes valence, $g_v$ is the genus of $v$, and $d_v$ is the local
    degree of $\phi$ at $v$.
\end{enumerate}
\end{definition}

Cavalieri and Dawson define tropical Tevelev degrees by tropicalizing the
algebraic construction of \cite{CPS:Tevdeg}. The tropical forgetful morphism
\[
\Fs \times \Ft
\]
is a map of weighted cone complexes with integral structures, for which a notion
of local degree above a general point of the target is defined. The map is
well-defined by the correspondence theorem recalled below.

For any non-negative integer $g$, consider the Hurwitz data
\begin{equation}
    \label{eq:hd}
    \mathfrak{h}(g)
    =
    \bigl(g,\, d=g+1,\, N=5g+3,\, \eta_1,\ldots,\eta_{5g+3}\bigr),
\end{equation}
where
\[
\eta_i =
\begin{cases}
    (1,\ldots,1) & \text{for } i \le g+3, \\
    (2,1,\ldots,1) & \text{for } g+4 \le i \le 5g+3.
\end{cases}
\]

\begin{definition}
\label{admtrop}
For any non-negative integer $g$, let $n=g+3$ and $d=g+1$. The authors define
$\Hur{g,d,n}^\trop$ to be the moduli space of tropical admissible covers
$\Hur{\mathfrak{h}(g)}^\trop$ with the following conventions:
\begin{itemize}
    \item for $i \le g+3$, the $i$-th end of $T$ is marked
and labeled $i$, and exactly one end of $\Gamma$ mapping
to it is marked and given the same label. All other ends
of $\Gamma$ over this target end are unmarked;
    \item for $i > g+3$, the ends of $T$ and their inverse images in $\Gamma$
    are unmarked.
\end{itemize}
\end{definition}

For $g \ge 0$, consider the morphism of tropical moduli spaces
\[
\Fs \times \Ft \colon
\Hur{g,d,n}^\trop
\longrightarrow
\Mgn{g,n}^\trop \times \Mgn{0,n}^\trop .
\]
Let \[ x=[\varphi:\Gamma\to T] \in (\Fs \times \Ft)^{-1}(p). \] In \cite[Section~3]{troptev}, the \emph{tropical Tevelev degree} is defined as
\[
\Tev_g^\trop := \deg(\Fs \times \Ft),
\]
where the local degree at a point $x$ is given by
\begin{equation}
\label{eq:locdegttev}
\deg_x(\Fs \times \Ft)
=
\frac{|\Aut(\overline{\Gamma})|}{|\Aut(\phi)|}
\cdot
\prod_{v \in V(\Gamma)} H_v
\cdot
\left| \det(M_{\sigma_\Theta}) \right|.
\end{equation}
Here $M_{\sigma_\Theta}$ is the matrix expressing the lengths of the compact
edges of $\Fs(\Gamma)$ and $\Ft(T)$ as linear functions of the edge lengths of
$\Gamma$ and $H_v$ denotes the {\it local Hurwitz number} associated to the vertex $v$ (\cite[Section 3.2.4]{CMRadmissible}).

The authors make the following convention about local Hurwitz numbers.

\begin{convention}\label{conv:lochurnum}
    Consider an admissible cover $\phi:\Gamma\to T$, and suppose a vertex $v\in \Gamma$ is adjacent to a certain number $n$ of unmarked ends mapping to the same end of $T$. Then the automorphism factor $1/n!$ consisting of permuting these ends is incorporated in the local Hurwitz number, rather than in the automorphism factors.  
 \end{convention}

The authors show that tropical Tevelev degrees agree with their algebraic
counterparts and compute them explicitly, yielding the following result.

\begin{theorem}[{\cite[Theorem~1.1 and Theorem~1.2]{troptev}}]
\label{thm:correspondence} \label{thm:ttev}
For any $g \ge 0$, we have
\[
\Tev_g = \Tev_g^\trop = 2^g .
\]
\end{theorem}

\subsubsection{Notation and conventions}
We summarize notation and conventions introduced in \cite{troptev} that will be used
throughout this paper. This subsection is an overview of notation, some of which will be explained in more detail in the following sections. The purpose of this subsection is to allow for efficient recovery of necessary information from \cite{troptev}. 

\begin{itemize}


    \item Tropical covers appearing in their paper correspond to tropical admissible covers with only simple ramification. Points with simple ramification are called \emph{transpositions}. The ends and labels of ramification index $2$ are left off in some figures in this paper to avoid overcrowding.


    \item For a tropical cover $\varphi \colon \Gamma \to T$, for every horizontal end/edge of the base curve, $T$, there is a end/edge in the cover mapping to it with degree greater than one; we call it the {\it active end/edge} and we call the collection of all active ends/edges the {\it active path}. The authors show that this path is unique for their chosen point in $\Mgn{g,n}^{\trop}\times \Mgn{0,n}^\trop$ and that all other edges of $\Gamma$
    have expansion factor one. This can also be thought of as the path connecting all cycles and connecting the last cycle to the vertex in $\Gamma$ to which the marked points $1$ and $2$ stabilize to. 

    \item Orient the active edge from left to right from the last cycle. We define a {\it cut} to be a transposition (point with simple ramification) with an edge of degree 1 separating from the active edge going to the right, causing the degree of the active edge to decrease by one. Similarly, we define a {\it join} to be a transposition with an edge of degree 1 coming from the left joining the active edge, causing the degree of the active edge to increase by one (illustrated in Figure \ref{fig:cutjoin}). 

    \item We always assume that the target tropical curve $T$ is
    chosen generically in a maximal cone of $\Mgn{0,n}^{\trop}$, so that all edge lengths
    are distinct and no combinatorial degeneracies occur.

    \item In figures, ends of degree one are sometimes omitted to
avoid clutter. Marked ends are labeled, and isomorphisms
preserve these labels. Unmarked ends carry no labels.
\end{itemize}

All notation not explicitly defined here follows the conventions of \cite{troptev}.
\begin{figure}[tb]
    \centering
     \resizebox{.8\textwidth}{!}{\tikzset{every picture/.style={line width=0.75pt}} 

\begin{tikzpicture}[x=0.75pt,y=0.75pt,yscale=-1,xscale=1]

\draw [line width=1.5]    (90.5,138) -- (300.5,138) ;
\draw [line width=1.5]    (368,138) -- (579,138.5) ;
\draw  [draw opacity=0] (91.56,71.61) .. controls (93.61,71.54) and (95.67,71.5) .. (97.75,71.5) .. controls (156.17,71.5) and (203.57,101.23) .. (204,138) -- (97.75,138.5) -- cycle ; \draw   (91.56,71.61) .. controls (93.61,71.54) and (95.67,71.5) .. (97.75,71.5) .. controls (156.17,71.5) and (203.57,101.23) .. (204,138) ;  
\draw  [draw opacity=0] (577.74,204.4) .. controls (575.69,204.48) and (573.62,204.53) .. (571.54,204.54) .. controls (513.13,204.79) and (465.59,175.27) .. (465.01,138.5) -- (571.25,137.54) -- cycle ; \draw   (577.74,204.4) .. controls (575.69,204.48) and (573.62,204.53) .. (571.54,204.54) .. controls (513.13,204.79) and (465.59,175.27) .. (465.01,138.5) ;  

\draw (134,123.4) node [anchor=north west][inner sep=0.75pt]  [font=\tiny]  {$d_{0}$};
\draw (407,126.4) node [anchor=north west][inner sep=0.75pt]  [font=\tiny]  {$d_{0}$};
\draw (231,125.4) node [anchor=north west][inner sep=0.75pt]  [font=\tiny]  {$d_{0} +1$};
\draw (515,126.4) node [anchor=north west][inner sep=0.75pt]  [font=\tiny]  {$d_{0} -1$};
\draw (130,64.4) node [anchor=north west][inner sep=0.75pt]  [font=\tiny]  {$1$};
\draw (508,200.9) node [anchor=north west][inner sep=0.75pt]  [font=\tiny]  {$1$};

\end{tikzpicture}}
    \caption{The active edge, drawn thickened, is oriented from left to right, i.e. away from the genus part. An edge joining the active edge is shown on the left, and an edge cutting from the active edge is shown on the right. Adapted from \cite{troptev}.}
    \label{fig:cutjoin}
\end{figure}

\subsubsection{Constructing the Covers}
\label{sec:constructing-covers}




To demonstrate the constructions of covers in \cite{troptev}, we briefly summarize low genus and degree examples. 

To compute $\Tev_1^\trop$ we must compute the global degree of the map
\begin{equation}
    \Fs\times \Ft: \Hur{1,2,4}^\trop \to \Mgn{1,4}^\trop \times \Mgn{0,4}^\trop.
\end{equation}

Consider the point $p = (\overline{\Gamma}, \overline{T})\in \Mgn{1,4}^\trop \times \Mgn{0,4}^\trop$ depicted in Figure \ref{fig:pgen1}. The set $(F_s \times F_t)^{-1}(p)$ consists of covers $\phi \colon \Gamma \to T$ such that $T$ stabilizes to
$\overline{T}$ after forgetting its four unmarked simple
branch ends, each with ramification profile $(2)$, and
$\Gamma$ stabilizes to $\overline{\Gamma}$ after forgetting
all its unmarked ends, including the four ends with
expansion factor $2$.

\begin{figure}[b]
    \centering
     \resizebox{\textwidth}{!}{\tikzset{every picture/.style={line width=0.75pt}} 

\begin{tikzpicture}[x=0.75pt,y=0.75pt,yscale=-1,xscale=1]

\draw  [line width=1.5]  (62,135) .. controls (62,129.34) and (66.92,124.75) .. (73,124.75) .. controls (79.08,124.75) and (84,129.34) .. (84,135) .. controls (84,140.66) and (79.08,145.25) .. (73,145.25) .. controls (66.92,145.25) and (62,140.66) .. (62,135) -- cycle ;
\draw [line width=1.5]    (84,135) -- (288,135) ;
\draw [color={rgb, 255:red, 208; green, 2; blue, 27 }  ,draw opacity=1 ][line width=1.5]    (288,135) -- (325,171) ;
\draw [color={rgb, 255:red, 208; green, 2; blue, 27 }  ,draw opacity=1 ][line width=1.5]    (326,107) -- (288,135) ;
\draw [color={rgb, 255:red, 208; green, 2; blue, 27 }  ,draw opacity=1 ][line width=1.5]    (193,134) -- (192,95) ;
\draw [color={rgb, 255:red, 208; green, 2; blue, 27 }  ,draw opacity=1 ][line width=1.5]    (126,135) -- (125,96) ;
\draw [line width=1.5]    (474,139) -- (905,139) ;
\draw [color={rgb, 255:red, 208; green, 2; blue, 27 }  ,draw opacity=1 ][line width=1.5]    (905,139) -- (942,175) ;
\draw [color={rgb, 255:red, 208; green, 2; blue, 27 }  ,draw opacity=1 ][line width=1.5]    (942,106) -- (905,139) ;
\draw [color={rgb, 255:red, 208; green, 2; blue, 27 }  ,draw opacity=1 ][line width=1.5]    (438,103) -- (451.68,116.31) -- (475,139) ;
\draw [color={rgb, 255:red, 208; green, 2; blue, 27 }  ,draw opacity=1 ][line width=1.5]    (475,139) -- (437,167) ;

\draw (333,171.4) node [anchor=north west][inner sep=0.75pt] [font = \Large]   {$\textcolor[rgb]{0.82,0.01,0.11}{1}$};
\draw (953,173.4) node [anchor=north west][inner sep=0.75pt]  [font = \Large]   {$\textcolor[rgb]{0.82,0.01,0.11}{1}$};
\draw (334,94.4) node [anchor=north west][inner sep=0.75pt] [font = \Large]    {$\textcolor[rgb]{0.82,0.01,0.11}{2}$};
\draw (951,89.4) node [anchor=north west][inner sep=0.75pt] [font = \Large]    {$\textcolor[rgb]{0.82,0.01,0.11}{2}$};
\draw (188,68.4) node [anchor=north west][inner sep=0.75pt] [font = \Large]    {$\textcolor[rgb]{0.82,0.01,0.11}{3}$};
\draw (422,171.4) node [anchor=north west][inner sep=0.75pt]  [font = \Large]   {$\textcolor[rgb]{0.82,0.01,0.11}{3}$};
\draw (120,69.4) node [anchor=north west][inner sep=0.75pt]  [font = \Large]   {$\textcolor[rgb]{0.82,0.01,0.11}{4}$};
\draw (419,93.4) node [anchor=north west][inner sep=0.75pt] [font = \Large]    {$\textcolor[rgb]{0.82,0.01,0.11}{4}$};
\draw (64,99.4) node [anchor=north west][inner sep=0.75pt]  [font = \Large]   {$x_1$};
\draw (96,111.4) node [anchor=north west][inner sep=0.75pt] [font = \Large]    {$x_2$};
\draw (154,112.4) node [anchor=north west][inner sep=0.75pt]  [font = \Large]   {$x_3$};
\draw (232,110.4) node [anchor=north west][inner sep=0.75pt]  [font = \Large]   {$x_4$};
\draw (686,116.4) node [anchor=north west][inner sep=0.75pt]  [font = \Large]   {$L_1$};
\draw (166,27.4) node [anchor=north west][inner sep=0.75pt]  [font=\LARGE]  {$\overline{\Gamma}$};
\draw (673,28.4) node [anchor=north west][inner sep=0.75pt]  [font=\LARGE]  {$\overline{T}$};

\end{tikzpicture}}
    \caption{The point $p$ in $\Mgn{1,4}^\trop \times \Mgn{0,4}^\trop$. We have $x_1<<x_2<<x_3<<x_4<<L_1$. This figure is reproduced from \cite{troptev}.}
    \label{fig:pgen1}
\end{figure}

To start constructing the cover $\Gamma \to T$, we look at $T$. Since $p$ lies in the interior of a maximal cone of (an appropriate refinement of) $\Mgn{1,4}^\trop \times \Mgn{0,4}^\trop$, $T$ must be trivalent. The authors start by constructing the covers without any marked points on them, which they call $\tilde{\Gamma}\to \tilde{T}$. $\tilde{T}$ must be the tropical curve with two simple branch points on either side of a bounded edge. There is a unique Hurwitz cover of this target curve, consisting of two ends with expansion factor one covering the compact edge of $T$ (and forming  a loop between two vertices),  and two infinite edges with degree $2$ coming from each vertex and covering the ends.

We now seek to add the marked points onto our cover $\Gamma \to T$. Due to the chosen point $p$ requiring the length $L_1$ to be much greater than any of the lengths of $\overline{\Gamma}$, there must be a long length in the cover curve $\Gamma$ mapping to the compact edge of $\overline{T}$ that is lost when stabilizing to $\overline{\Gamma}$. The local description of the two different ways to do this is depicted in Figure \ref{fig:deg2long}. 

\begin{figure}
    \centering
    \tikzset{every picture/.style={line width=0.75pt}} 

\begin{tikzpicture}[x=0.5pt,y=0.5pt,yscale=-1,xscale=1]

\draw    (334,100.5) -- (546,98.5) ;
\draw    (332,230.5) -- (544,228.5) ;
\draw    (438,229.5) -- (438,193.5) ;
\draw [color={rgb, 255:red, 208; green, 2; blue, 27 }  ,draw opacity=1 ]   (438,193.5) -- (423,169.5) ;
\draw [color={rgb, 255:red, 208; green, 2; blue, 27 }  ,draw opacity=1 ]   (438,193.5) -- (451,168.5) ;
\draw    (437,119.5) -- (437.94,151.5) ;
\draw [shift={(438,153.5)}, rotate = 268.32] [color={rgb, 255:red, 0; green, 0; blue, 0 }  ][line width=0.75]    (10.93,-3.29) .. controls (6.95,-1.4) and (3.31,-0.3) .. (0,0) .. controls (3.31,0.3) and (6.95,1.4) .. (10.93,3.29)   ;
\draw    (440,99.5) -- (422,62.5) ;
\draw    (440,99.5) -- (453,62.5) ;
\draw [color={rgb, 255:red, 208; green, 2; blue, 27 }  ,draw opacity=1 ]   (422,62.5) -- (399,48.5) ;
\draw    (422,62.5) -- (421,38.5) ;
\draw    (453,62.5) -- (446,38.5) ;
\draw [color={rgb, 255:red, 208; green, 2; blue, 27 }  ,draw opacity=1 ]   (453,62.5) -- (468,42.5) ;
\draw [line width=0.75]    (215,231.5) -- (215,210) ;
\draw [line width=0.75]    (215,210) -- (232,195) ;
\draw [line width=0.75]    (215,210) -- (198,196) ;
\draw [color={rgb, 255:red, 208; green, 2; blue, 27 }  ,draw opacity=1 ][line width=0.75]    (198,196) -- (203,177) ;
\draw [color={rgb, 255:red, 208; green, 2; blue, 27 }  ,draw opacity=1 ][line width=0.75]    (198,196) -- (180,185) ;
\draw [color={rgb, 255:red, 208; green, 2; blue, 27 }  ,draw opacity=1 ][line width=0.75]    (232,195) -- (234,177) ;
\draw [color={rgb, 255:red, 208; green, 2; blue, 27 }  ,draw opacity=1 ][line width=0.75]    (232,195) -- (249,192) ;
\draw [line width=0.75]    (213.45,101) -- (196.53,72.75) ;
\draw [line width=0.75]    (213.45,101) -- (227.76,72.75) ;
\draw [line width=0.75]    (196.53,72.75) -- (173.12,67.1) ;
\draw [line width=0.75]    (196.53,72.75) -- (197.83,54.38) ;
\draw [line width=0.75]    (227.76,72.75) -- (223.85,52.97) ;
\draw [line width=0.75]    (227.76,72.75) -- (245.97,60.03) ;
\draw [line width=0.75]    (173.12,67.1) -- (171.82,44.49) ;
\draw [color={rgb, 255:red, 208; green, 2; blue, 27 }  ,draw opacity=1 ][line width=0.75]    (173.12,67.1) -- (151,68.51) ;
\draw [line width=0.75]    (197.83,54.38) -- (187.43,34.6) ;
\draw [line width=0.75]    (197.83,54.38) -- (205.64,34.6) ;
\draw [line width=0.75]    (223.85,52.97) -- (217.35,31.78) ;
\draw [color={rgb, 255:red, 208; green, 2; blue, 27 }  ,draw opacity=1 ][line width=0.75]    (223.85,52.97) -- (235.56,36.02) ;
\draw [color={rgb, 255:red, 208; green, 2; blue, 27 }  ,draw opacity=1 ][line width=0.75]    (245.97,60.03) -- (251.17,43.08) ;
\draw [color={rgb, 255:red, 208; green, 2; blue, 27 }  ,draw opacity=1 ][line width=0.75]    (245.97,60.03) -- (262.88,60.03) ;
\draw    (107,102.5) -- (319,100.5) ;
\draw    (105,232.5) -- (317,230.5) ;
\draw    (210,121.5) -- (210.94,153.5) ;
\draw [shift={(211,155.5)}, rotate = 268.32] [color={rgb, 255:red, 0; green, 0; blue, 0 }  ][line width=0.75]    (10.93,-3.29) .. controls (6.95,-1.4) and (3.31,-0.3) .. (0,0) .. controls (3.31,0.3) and (6.95,1.4) .. (10.93,3.29)   ;

\draw (381,79.4) node [anchor=north west][inner sep=0.75pt]  [font=\small,color={rgb, 255:red, 1; green, 1; blue, 240 }  ,opacity=1 ]  {$2$};
\draw (488,79.4) node [anchor=north west][inner sep=0.75pt]  [font=\small]  {$\textcolor[rgb]{0,0,0.94}{2}$};
\draw (450,152.4) node [anchor=north west][inner sep=0.75pt]  [font=\footnotesize,color={rgb, 255:red, 208; green, 2; blue, 27 }  ,opacity=1 ]  {$1$};
\draw (412,153.4) node [anchor=north west][inner sep=0.75pt]  [font=\footnotesize,color={rgb, 255:red, 208; green, 2; blue, 27 }  ,opacity=1 ]  {$2$};
\draw (471,28.4) node [anchor=north west][inner sep=0.75pt]  [font=\footnotesize,color={rgb, 255:red, 208; green, 2; blue, 27 }  ,opacity=1 ]  {$1$};
\draw (386,33.4) node [anchor=north west][inner sep=0.75pt]  [font=\footnotesize,color={rgb, 255:red, 208; green, 2; blue, 27 }  ,opacity=1 ]  {$2$};
\draw (413,72.4) node [anchor=north west][inner sep=0.75pt]  [font=\footnotesize]  {$y$};
\draw (168,173.4) node [anchor=north west][inner sep=0.75pt]  [font=\footnotesize,color={rgb, 255:red, 208; green, 2; blue, 27 }  ,opacity=1 ]  {$4$};
\draw (199,159.4) node [anchor=north west][inner sep=0.75pt]  [font=\footnotesize,color={rgb, 255:red, 208; green, 2; blue, 27 }  ,opacity=1 ]  {$3$};
\draw (230.86,17.7) node [anchor=north west][inner sep=0.75pt]  [font=\footnotesize,color={rgb, 255:red, 208; green, 2; blue, 27 }  ,opacity=1 ]  {$3$};
\draw (236,161.4) node [anchor=north west][inner sep=0.75pt]  [font=\footnotesize,color={rgb, 255:red, 208; green, 2; blue, 27 }  ,opacity=1 ]  {$2$};
\draw (254.28,27.59) node [anchor=north west][inner sep=0.75pt]  [font=\footnotesize,color={rgb, 255:red, 208; green, 2; blue, 27 }  ,opacity=1 ]  {$2$};
\draw (253,179.4) node [anchor=north west][inner sep=0.75pt]  [font=\footnotesize,color={rgb, 255:red, 208; green, 2; blue, 27 }  ,opacity=1 ]  {$1$};
\draw (269.89,50.19) node [anchor=north west][inner sep=0.75pt]  [font=\footnotesize,color={rgb, 255:red, 208; green, 2; blue, 27 }  ,opacity=1 ]  {$1$};
\draw (224.51,80.69) node [anchor=north west][inner sep=0.75pt]  [font=\footnotesize]  {$y$};
\draw (175.12,72.5) node [anchor=north west][inner sep=0.75pt]  [font=\footnotesize]  {$w$};
\draw (237.86,65.79) node [anchor=north west][inner sep=0.75pt]  [font=\footnotesize]  {$z$};
\draw (153,81.4) node [anchor=north west][inner sep=0.75pt]  [font=\small,color={rgb, 255:red, 1; green, 1; blue, 240 }  ,opacity=1 ]  {$2$};
\draw (261,80.4) node [anchor=north west][inner sep=0.75pt]  [font=\small,color={rgb, 255:red, 1; green, 1; blue, 240 }  ,opacity=1 ]  {$2$};
\draw (137,59.4) node [anchor=north west][inner sep=0.75pt]  [font=\footnotesize,color={rgb, 255:red, 208; green, 2; blue, 27 }  ,opacity=1 ]  {$4$};

\end{tikzpicture}
    \caption{A local picture of two configurations of marked points on a degree two tropical cover that allow for a long edge to be forgotten in the stabilization of the cover curve. The cover on the left has long length $w$ and the cover on the right has long length $y$. This figure is reproduced from \cite{troptev}.}
    \label{fig:deg2long}
\end{figure}

\begin{figure}
    \centering
    \include{Figures/genus1}
    \caption{The two covers in $(\Fs\times\Ft)^{-1}(p)$. The thickened parts show the subgraphs to which the graphs stabilize via the maps $\Fs, \Ft$. This figure is reproduced from \cite{troptev}.}
    \label{fig:genus1}
\end{figure}

From these local pictures we can get global inverse images of $p$ by placing the marked points on $\tilde{\Gamma}$ to recover $\Gamma$ in two possible ways.
We can attach the tree from the left hand side of Figure \ref{fig:deg2long} at distance $x_2$ from the loop; we can then choose the length $y = x_3, z = x_4 , w = L_1-x_4$.

Alternatively, we place mark $4$ at distance $x_2$ from the loop, followed by mark $3$ after a distance of $x_3$,  and finally  the marks $2$ and $1$ on a tripod (as in Figure \ref{fig:deg2long}) that attaches after distance $x_4$. We  choose $y = L-2x_4$.
Both types of inverse images of $p$ are shown in Figure \ref{fig:genus1}.

We now compute the local degree of $(\Fs\times\Ft)$ at these inverse images, which gives us the multiplicities with which we need to count the covers.

We follow Equation \eqref{eq:locdegttev}, for $i = 1,2$ the multiplicity of the inverse image $\phi_i:\Gamma_i\to T_i$ is the product of three factors: an automorphism factor, a product of local Hurwitz numbers and a dilation factor corresponding to the determinant of  the matrix representing the map $\Fs\times\Ft$.

The automorphism factor is the same for $i = 1,2$.
For each cover, we have $\Aut(\phi_i) = \mu_2\times \mu_2$: one factor corresponds to switching simultaneously the two unlabeled left ends of branching type $2$ and their inverse images; the second factor consists of switching the two degree $1$ edges of $\Gamma_i$ forming the loop. However, this is also a nontrivial automorphism of $\overline{\Gamma_i}$. Altogether, we have
\begin{equation}
    \frac{|\Aut(\overline{\Gamma_i})|}{|\Aut(\phi_i)|} = \frac{1}{2}.
\end{equation}
 
The local Hurwitz numbers factors are also computed identically for $i = 1,2$:  every vertex in the cover is either trivalent with degree one edges in all directions or two edges of degree 2 in different directions and 2 edges of degree 1 in the same direction. Both of these types of vertices have local Hurwitz number equal to 1\footnote{this follows from the definition of Hurwitz numbers and from Convention \ref{conv:lochurnum}; see \cite{CMHurw} for an elementary treatment of Hurwitz numbers, or \cite[Lemma 4.2]{CJMtrophurw} for a tropical way to obtain this result.},
therefore the product of all local Hurwitz numbers is 1. 

 To calculate the dilation factors, we set up the following matrices representing the $x_i$'s and $L_j$'s in terms of the $y_k$'s:
\[\begin{array}{cc}
  M_1=\begin{blockarray}{cccccc}
   &y_1 & y_2 & y_3 & y_4 & y_5\\
    \begin{block}{c[ccccc]}
      x_1 & 2 & 0 & 0 & 0 & 0\\
      x_2 & 0 & 1 & 0 & 0 & 0\\
      x_3 & 0 & 0 & 1 & 0 & 0\\
      x_4 & 0 & 0 & 0 & 0 & 1\\
      L_1 & 0 & 0 & 0 & 1 & 1\\
    \end{block}
  \end{blockarray} &
 M_2=\begin{blockarray}{cccccc}
   &y'_1 & y'_2 & y'_3 & y'_4 & y'_5\\
    \begin{block}{c[ccccc]}
      x_1 & 2 & 0 & 0 & 0 & 0\\
      x_2 & 0 & 1 & 0 & 0 & 0\\
      x_3 & 0 & 0 & 1 & 0 & 0\\
      x_4 & 0 & 0 & 0 & 1 & 0\\
      L_1 & 0 & 0 & 0 & 1 & 1\\
    \end{block}
  \end{blockarray}
  \end{array}
\]
We see that $|\det M_i |=2$ for $i = 1,2$.
All together the multiplicity of each cover in $(\Fs\times\Ft)^{-1}(p)$ is $1\cdot \frac{1}{2}\cdot 2=1$. Since we have two inverse images each with multiplicity one,  we obtain $\Tev^\trop_1 = 2$.

To compute $\Tev_2^\trop$, we look for the global degree of the map
\begin{equation}
    \Fs\times \Ft: \Hur{2,3,5}^\trop \to \Mgn{2,5}^\trop \times \Mgn{0,5}^\trop.
\end{equation}

We consider the point $p=(\overline{\Gamma},\overline{T})\in \Mgn{2,5}^\trop \times \Mgn{0,5}^\trop$ as illustrated in Figure \ref{fig:genim}. Unlike the genus 1 case, there are multiple options for how $\tilde{T}$ can be shaped. For two of these $\tilde{T}$-trees it is possible to place fragments of the marked tree $\overline{T}$ to obtain the base $T$ of the cover. 

When adding the 5 marked points, we use the same techniques as in the genus 1 example to get two long edges of $\Gamma$ that map to  compact edges of $\overline{T}$ and are lost when stabilizing to $\overline{\Gamma}$. All of the preimages are shown in Figure \ref{fig:l=0g=2}.

\begin{figure}[h]
    \centering
    \include{Figures/l=0g=2}
    \caption{The four covers in $(\Fs\times\Ft)^{-1}(p)$ when $g=2$. The thickened parts show the subgraphs to which the graphs stabilize via the maps $\Fs, \Ft$. Note that the covers are organized in such a way that the covers in the same row have the genus formed in the same way while if two covers are in the same column the marked points are organized in the same way on $T$. This figure is reproduced from \cite{troptev}.}
    \label{fig:l=0g=2}
\end{figure}

The multiplicities of these four covers can be computed by computing the local degree of $(\Fs\times\Ft)$ at all 4 inverse images as in the previous genus 1 example and shown to all be equal to 1, and we conclude $\Tev_2^\trop=4$. 

\subsubsection{Proof techniques in \emph{Tropical Tevelev Degrees}}\label{sec:prooftroptev}

The proofs in this paper use the same techniques and build off of the proof of Theorem 1.2 in \cite{troptev}. We therefore recall the main techniques and sketch their proof in this section. 

The authors' proof that the tropical Tevelev degree equals $2^{g}$ gives a fully combinatorial computation of the degree of the tropical morphism
\begin{equation}
    \Fs\times \Ft: \Hur{g,d,n}^\trop \to \Mgn{g,n}^\trop \times \Mgn{0,n}^\trop
\end{equation}
when $d = g+1$ and $n = g+3$. After refining the cone complexes so that the map is strict, they select a target point $p = (\overline{\Gamma}, \overline{T})$ in the interior of a top dimensional cone, shown in Figure \ref{fig:genim}. Here, $\overline{\Gamma}$ is a chain of $g$ loops followed by a caterpillar tree, and $\overline{T}$ is a trivalent caterpillar tree, both with generic edge lengths. 

\begin{figure}[tb]
    \centering
     \resizebox{\textwidth}{!}{\tikzset{every picture/.style={line width=0.75pt}} 

\begin{tikzpicture}[x=0.75pt,y=0.75pt,yscale=-1,xscale=1]

\draw   (26,66) .. controls (26,57.72) and (36.52,51) .. (49.5,51) .. controls (62.48,51) and (73,57.72) .. (73,66) .. controls (73,74.28) and (62.48,81) .. (49.5,81) .. controls (36.52,81) and (26,74.28) .. (26,66) -- cycle ;
\draw    (73,66) -- (108,66) ;
\draw  [draw opacity=0] (108,66) .. controls (108.94,61.43) and (117.27,57.85) .. (127.41,57.85) .. controls (137.48,57.84) and (145.78,61.36) .. (146.81,65.89) -- (127.41,66.85) -- cycle ; \draw   (108,66) .. controls (108.94,61.43) and (117.27,57.85) .. (127.41,57.85) .. controls (137.48,57.84) and (145.78,61.36) .. (146.81,65.89) ;  
\draw  [draw opacity=0] (146.81,65.89) .. controls (146.36,84.42) and (138.25,99.16) .. (128.12,99.28) .. controls (117.54,99.42) and (108.77,83.59) .. (108.53,63.94) .. controls (108.53,63.65) and (108.52,63.37) .. (108.52,63.09) -- (127.67,63.7) -- cycle ; \draw   (146.81,65.89) .. controls (146.36,84.42) and (138.25,99.16) .. (128.12,99.28) .. controls (117.54,99.42) and (108.77,83.59) .. (108.53,63.94) .. controls (108.53,63.65) and (108.52,63.37) .. (108.52,63.09) ;  
\draw    (146.81,65.89) -- (178,66) ;
\draw    (224,66) -- (289,66) ;
\draw [color={rgb, 255:red, 208; green, 2; blue, 27 }  ,draw opacity=1 ]   (256.67,66.67) -- (256.67,45.67) ;
\draw [color={rgb, 255:red, 208; green, 2; blue, 27 }  ,draw opacity=1 ]   (289,66) -- (302,46) ;
\draw [color={rgb, 255:red, 208; green, 2; blue, 27 }  ,draw opacity=1 ]   (289,66) -- (302,85) ;
\draw    (369,67) -- (466,67) ;
\draw [color={rgb, 255:red, 208; green, 2; blue, 27 }  ,draw opacity=1 ]   (369,67) -- (351,48) ;
\draw [color={rgb, 255:red, 208; green, 2; blue, 27 }  ,draw opacity=1 ]   (369,67) -- (350,87) ;
\draw [color={rgb, 255:red, 208; green, 2; blue, 27 }  ,draw opacity=1 ]   (417.5,67) -- (417.5,44) ;
\draw    (527,67) -- (612,67) ;
\draw [color={rgb, 255:red, 208; green, 2; blue, 27 }  ,draw opacity=1 ]   (569.5,67) -- (569,44) ;
\draw [color={rgb, 255:red, 208; green, 2; blue, 27 }  ,draw opacity=1 ]   (612,67) -- (630,47) ;
\draw [color={rgb, 255:red, 208; green, 2; blue, 27 }  ,draw opacity=1 ]   (612,67) -- (629,87) ;

\draw (41.67,34.73) node [anchor=north west][inner sep=0.75pt]  [font=\small]  {$x_{1}$};
\draw (85,47.4) node [anchor=north west][inner sep=0.75pt]  [font=\small]  {$x_{2}$};
\draw (121,40.73) node [anchor=north west][inner sep=0.75pt]  [font=\small]  {$x_{3}$};
\draw (121,101.07) node [anchor=north west][inner sep=0.75pt]  [font=\small]  {$x_{4}$};
\draw (304.09,88.34) node [anchor=north west][inner sep=0.75pt]  [font=\footnotesize,color={rgb, 255:red, 208; green, 2; blue, 27 }  ,opacity=1 ,rotate=-358.45]  {$ \begin{array}{l}
1\\
\end{array}$};
\draw (305.63,30.16) node [anchor=north west][inner sep=0.75pt]  [font=\footnotesize,color={rgb, 255:red, 208; green, 2; blue, 27 }  ,opacity=1 ,rotate=-359.63]  {$ \begin{array}{l}
2\\
\end{array}$};
\draw (251.68,27.69) node [anchor=north west][inner sep=0.75pt]  [font=\footnotesize,color={rgb, 255:red, 208; green, 2; blue, 27 }  ,opacity=1 ,rotate=-0.17]  {$3$};
\draw (631,90.4) node [anchor=north west][inner sep=0.75pt]  [font=\footnotesize,color={rgb, 255:red, 208; green, 2; blue, 27 }  ,opacity=1 ]  {$1$};
\draw (633.33,29.73) node [anchor=north west][inner sep=0.75pt]  [font=\footnotesize,color={rgb, 255:red, 208; green, 2; blue, 27 }  ,opacity=1 ]  {$2$};
\draw (563,28.73) node [anchor=north west][inner sep=0.75pt]  [font=\footnotesize,color={rgb, 255:red, 208; green, 2; blue, 27 }  ,opacity=1 ]  {$3$};
\draw (337.28,86.79) node [anchor=north west][inner sep=0.75pt]  [font=\footnotesize,color={rgb, 255:red, 208; green, 2; blue, 27 }  ,opacity=1 ,rotate=-359.4]  {$n$};
\draw (335.08,29.85) node [anchor=north west][inner sep=0.75pt]  [font=\footnotesize,color={rgb, 255:red, 208; green, 2; blue, 27 }  ,opacity=1 ,rotate=-0.85]  {$n-1$};
\draw (402.67,30.05) node [anchor=north west][inner sep=0.75pt]  [font=\footnotesize,color={rgb, 255:red, 208; green, 2; blue, 27 }  ,opacity=1 ,rotate=-0.08]  {$n-2$};
\draw (384.33,50.73) node [anchor=north west][inner sep=0.75pt]  [font=\footnotesize]  {$L_{1}$};
\draw (431,51.4) node [anchor=north west][inner sep=0.75pt]  [font=\footnotesize]  {$L_{2}$};
\draw (577,50.4) node [anchor=north west][inner sep=0.75pt]  [font=\footnotesize]  {$L_{g+2\ell }$};
\draw (185,58.4) node [anchor=north west][inner sep=0.75pt]  [font=\huge]  {$\dotsc $};
\draw (480,59.4) node [anchor=north west][inner sep=0.75pt]  [font=\huge]  {$\dotsc $};
\draw (258.67,49.07) node [anchor=north west][inner sep=0.75pt]  [font=\footnotesize]  {$x_{4g+2\ell }$};
\draw (179,13.4) node [anchor=north west][inner sep=0.75pt]    {$\overline{\Gamma }$};
\draw (486,13.4) node [anchor=north west][inner sep=0.75pt]    {$\overline{T}$};

\end{tikzpicture}}
    \caption{The graphs $\overline{\Gamma}, \overline{T}$ defining  the chosen point $p$ of $\Mgn{g,n}^\trop\times \Mgn{0,n}^\trop$. This figure is reproduced from \cite{troptev}.}
    \label{fig:genim}
\end{figure}

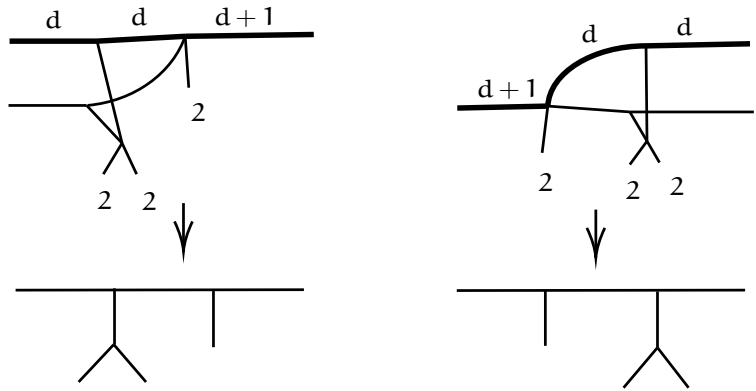
\begin{figure}[tb]
    \centering
     \resizebox{.7\textwidth}{!}{\tikzset{every picture/.style={line width=0.75pt}} 

\begin{tikzpicture}[x=0.75pt,y=0.75pt,yscale=-1,xscale=1]

\draw    (26.81,130.05) -- (132.34,129.89) ;
\draw    (188.28,130.2) -- (293.82,130.04) ;
\draw    (63.02,129.4) -- (63.02,149.97) ;
\draw    (62.75,149.97) -- (49.88,163.69) ;
\draw    (62.75,149.97) -- (74.55,163.69) ;
\draw    (99.23,130.38) -- (99.23,150.95) ;
\draw    (261.74,130.53) -- (261.74,151.1) ;
\draw    (261.74,151.1) -- (249.33,165.8) ;
\draw    (261.74,151.1) -- (273.12,165.8) ;
\draw    (220.7,129.88) -- (220.7,150.45) ;
\draw [line width=1.5]    (24.6,38.85) -- (56.63,38.85) ;
\draw [line width=0.75]    (24.22,62.4) -- (52.82,62.4) ;
\draw [line width=0.75]    (56.63,38.85) -- (65.78,76.21) ;
\draw [line width=0.75]    (52.82,62.4) -- (65.78,76.21) ;
\draw [line width=1.5]    (88.89,37.07) -- (135.95,36.41) ;
\draw  [draw opacity=0][line width=0.75]  (89.08,36.53) .. controls (84.44,50.48) and (70.34,60.81) .. (53.2,62.38) -- (48,26.02) -- cycle ; \draw  [line width=0.75]  (89.08,36.53) .. controls (84.44,50.48) and (70.34,60.81) .. (53.2,62.38) ;  
\draw [line width=1.5]    (56.63,38.85) -- (88.89,37.07) ;
\draw [line width=0.75]    (65.78,76.21) -- (58.92,87.58) ;
\draw [line width=0.75]    (65.78,76.21) -- (71.12,87.58) ;
\draw [line width=0.75]    (88.89,37.07) -- (90.19,55.9) ;
\draw [line width=0.75]    (88.26,98.18) -- (88.26,113.82) ;
\draw [shift={(88.26,115.82)}, rotate = 270] [color={rgb, 255:red, 0; green, 0; blue, 0 }  ][line width=0.75]    (10.93,-3.29) .. controls (6.95,-1.4) and (3.31,-0.3) .. (0,0) .. controls (3.31,0.3) and (6.95,1.4) .. (10.93,3.29)   ;
\draw [line width=1.5]    (188.46,63.35) -- (221.61,62.64) ;
\draw [line width=0.75]    (221.61,62.64) -- (219.52,79.76) ;
\draw  [draw opacity=0][line width=1.5]  (221.61,62.65) .. controls (222.2,50.4) and (238.07,40.58) .. (257.55,40.58) .. controls (257.55,40.58) and (257.55,40.58) .. (257.55,40.58) -- (257.55,63.36) -- cycle ; \draw  [line width=1.5]  (221.61,62.65) .. controls (222.2,50.4) and (238.07,40.58) .. (257.55,40.58) .. controls (257.55,40.58) and (257.55,40.58) .. (257.55,40.58) ;  
\draw [line width=0.75]    (221.61,62.64) -- (251.63,64.78) ;
\draw [line width=0.75]    (257.55,40.58) -- (257.91,75.48) ;
\draw [line width=0.75]    (251.63,64.78) -- (257.91,75.48) ;
\draw [line width=0.75]    (251.63,64.78) -- (297,64.78) ;
\draw [line width=1.5]    (257.55,40.58) -- (295.6,39.81) ;
\draw [line width=0.75]    (257.91,75.48) -- (251.63,84.04) ;
\draw [line width=0.75]    (257.91,75.48) -- (262.49,83.32) ;
\draw [line width=0.75]    (239.18,100.55) -- (239.18,116.18) ;
\draw [shift={(239.18,118.18)}, rotate = 270] [color={rgb, 255:red, 0; green, 0; blue, 0 }  ][line width=0.75]    (10.93,-3.29) .. controls (6.95,-1.4) and (3.31,-0.3) .. (0,0) .. controls (3.31,0.3) and (6.95,1.4) .. (10.93,3.29)   ;

\draw (35.52,25.92) node [anchor=north west][inner sep=0.75pt]  [font=\tiny,color={rgb, 255:red, 0; green, 0; blue, 0 }  ,opacity=1 ]  {$d_0$};
\draw (97.48,24.85) node [anchor=north west][inner sep=0.75pt]  [font=\tiny,color={rgb, 255:red, 0; green, 0; blue, 0 }  ,opacity=1 ]  {$d_0+1$};
\draw (66.93,24.85) node [anchor=north west][inner sep=0.75pt]  [font=\tiny,color={rgb, 255:red, 0; green, 0; blue, 0 }  ,opacity=1 ]  {$d_0$};
\draw (190,51.27) node [anchor=north west][inner sep=0.75pt]  [font=\tiny,color={rgb, 255:red, 0; green, 0; blue, 0 }  ,opacity=1 ]  {$d_0+1$};
\draw (231.66,31.13) node [anchor=north west][inner sep=0.75pt]  [font=\tiny,color={rgb, 255:red, 0; green, 0; blue, 0 }  ,opacity=1 ]  {$d_0$};
\draw (266,28.8) node [anchor=north west][inner sep=0.75pt]  [font=\tiny,color={rgb, 255:red, 0; green, 0; blue, 0 }  ,opacity=1 ]  {$d_0$};
\draw (55.09,92.19) node [anchor=north west][inner sep=0.75pt]  [font=\tiny]  {$2$};
\draw (71.35,92.76) node [anchor=north west][inner sep=0.75pt]  [font=\tiny]  {$2$};
\draw (89.29,60.3) node [anchor=north west][inner sep=0.75pt]  [font=\tiny]  {$2$};
\draw (216.51,85.79) node [anchor=north west][inner sep=0.75pt]  [font=\tiny]  {$2$};
\draw (248.47,88.64) node [anchor=north west][inner sep=0.75pt]  [font=\tiny]  {$2$};
\draw (264.58,87.23) node [anchor=north west][inner sep=0.75pt]  [font=\tiny]  {$2$};

\end{tikzpicture}}
    \caption{Two possible ways to add genus $U$, on left, and $D$, on right. We omit from the picture ends of degree $1$ to avoid clutter. The active path is thickened. Note that adding $U$ increases degree of active edge from right to left, while adding $D$ decreases the degree. This figure is reproduced from \cite{troptev}.}
    \label{fig:2genusoptions}
\end{figure}

\begin{figure}[tb]
    \centering
     \resizebox{.9\textwidth}{!}{\input{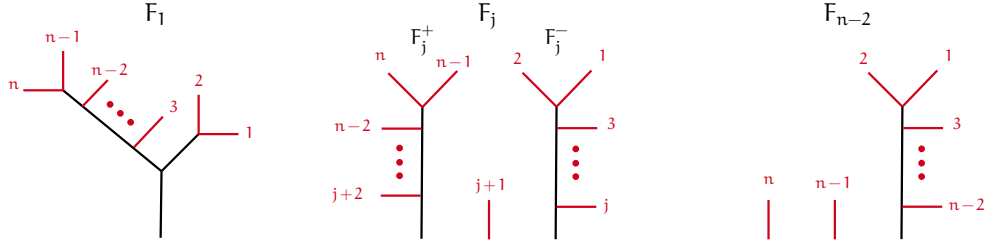}}
    \caption{Marked fragments that attach to the horizontal edge of $\tilde{T}$ to obtain the base graph $T$. We denote by $F_j^-$ the connected component that contains the marks with the lowest indices, and $F_j^+$  the one containing the highest labels. This figure is reproduced from \cite{troptev}.}
    \label{fig:markfrag}
\end{figure}
Due to the structure of $\overline{\Gamma}$, each tropical cover $\Gamma \to T$ in this fiber decomposes into two parts:
\begin{enumerate}
  \item A \emph{genus part}, consisting of the unique degree 2, genus one cover, followed by a chain of $g-1$ genus fragments, each appearing in one of two configurations as was shown in the genus 2 example (labeled \textit{U} or \textit{D} as shown in Figure \ref{fig:2genusoptions}). 
  \item A \emph{marked-tree part}, a trivalent tree carrying the $g+3$ marked ends, attached to the genus part along a single active edge. The authors show that there are $n-2$ ways to organize the marked points, which they called marked fragments shown in Figure \ref{fig:markfrag}. One possible marked fragment is $F_0$, which is a tree splitting into two edges with one edge containing marked points $1$ and $2$ and the rest of the marked points in numerical order toward the end of the other edge. The rest of the marked fragments are of type $F_j$ and consist of two trees with a single marked point between them. The tree on the right hand side of $F_j$ we denote $F_j^-$ and it contains marked points $1$ through $j$. The single marked point is $j+1$ and the rest of the marked points are on the right tree of $F_j$. Fragment type is shown to exist for $2\leq j \leq n-2$. The authors rule out all other configurations of marked points. \end{enumerate}

Here, by a ``part'' of the cover we mean the restriction of the cover to a connected subgraph of \(\Gamma\), together with its image in \(T\). We regard the attaching active edge as belonging to both the genus part and the marked-tree part. Thus, their source subgraphs intersect precisely along the active edge, and their union is \(\Gamma \to T\).

The authors start by constructing the covers without any marked points on them, which they call $\tilde{\Gamma}\to \tilde{T}$. They start this process by constructing the genus part of the cover starting with a loop of degree 2 followd by a chain of $U$ and $D$ genus fragments. The genus part ends with an edge of degree equal to $d-2i$ where $i$ is the number of $D$ genus fragments added to the chain of genus. This edge is the active edge. The marked-tree part is attached to the genus part along this active edge. 

The marked-tree part is a genus zero part of the cover, and thus is a sequence of cuts and joins that can not reconnect to the cover or else a loop would be formed. The marked-tree part starts with an active edge of degree equal to $d-2i$ where $0\leq i\leq \lfloor\frac{d-1}{2}\rfloor$. The genus part attached to this active edge on the left along its active edge. All possible gluings of genus parts with marked-tree parts form a finite collection of covers $\tilde{\Gamma}\to \tilde{T}$. To complete $\tilde{\Gamma}\to \tilde{T}$ to $\Gamma\to T$, marked points are added to the cover. 

We recall how the marked-tree fragments are attached to the covers. Suppose that the genus part ends with an active edge of degree $d-2i$. For a compatible fragment $F_{j}$, with $2\leq j\leq n-2$, extend the active path by $d-i-j$ cuts, followed by $i$ consecutive joins, and then $j-2$ further cuts. Let $q^{+}$ be a point separating the first cuts from the joins, let $q^{0}$ separate the joins from the remaining cuts, and let $q^{-}$ lie on the final edge of the active path. The component $F_{j}^{+}$, containing the markings with the largest labels, is attached at $q^{+}$; the end labeled $j+1$ is attached at $q^{0}$; and $F_{j}^{-}$, containing the markings with the smallest labels, is attached at $q^{-}$. When there are no joins, $q^{+}$ and $q^{0}$ are chosen in this order on the same active edge. The exceptional fragment $F_{0}$ is instead attached in its entirety to the final edge of the active path.

It remains to choose the marked preimages on the covering curve. Order the connected components of the inverse image of the marked fragment according to the order in which they meet the active path, beginning at the genus part and moving outwards. Place the mark \(n\) on the first inverse image of \(F_{j}^{+}\), and then place the marks \(n-1,n-2,\ldots,3\), in descending order, on the successive inverse-image components corresponding to their target branches. The final inverse image of \(F_{j}^{-}\) meets an active edge of degree two and therefore has two branches; the marks \(1\) and \(2\) are placed on distinct branches. This process is demonstrated in Figure \ref{fig:attachmarkedpts}. For \(F_{0}\), the same descending-order procedure is used, with the mark \(4\) on one of the two final copies and the marks \(3,2,1\) on the other. After forgetting the unmarked ends and stabilizing, the resulting source and target are precisely \(\overline{\Gamma}\) and \(\overline{T}\).

\begin{figure}[tb]
    \centering
     \resizebox{1\textwidth}{!}{\input{Figures/attachingmarkedpts}}
    \caption{The image represents the marked-tree part of the cover $\Gamma$ for the marked fragment $F_j$. The inverse images of the points $q^\pm,q^0$ on the active path are denoted by the same names to not clutter the picture. The $\widehat{\Gamma}$ denotes where the genus part of the cover is connected.}
    \label{fig:attachmarkedpts}
\end{figure}

For each constructed cover, the \emph{multiplicity} is determined from the local degree of the map $F_s \times F_t$. The local degree is the product of a local Hurwitz numbers factor, a dilation factor, and an automorphism factor. The local Hurwitz factor is proven to be $1$. The matrix associated to the dilation factor splits into a block corresponding to the genus part (whose determinant equals $2^{g}$) and another for the marked-tree part (determinant $1$). Each cover has an automorphism group of order $2^{g}$,
switching pairs of simple branch points attached to the same vertex and their preimages, so the ratio of lattice index to automorphism order yields multiplicity one.

\begin{figure}[tb]
    \centering
     \resizebox{\textwidth}{!}{\input{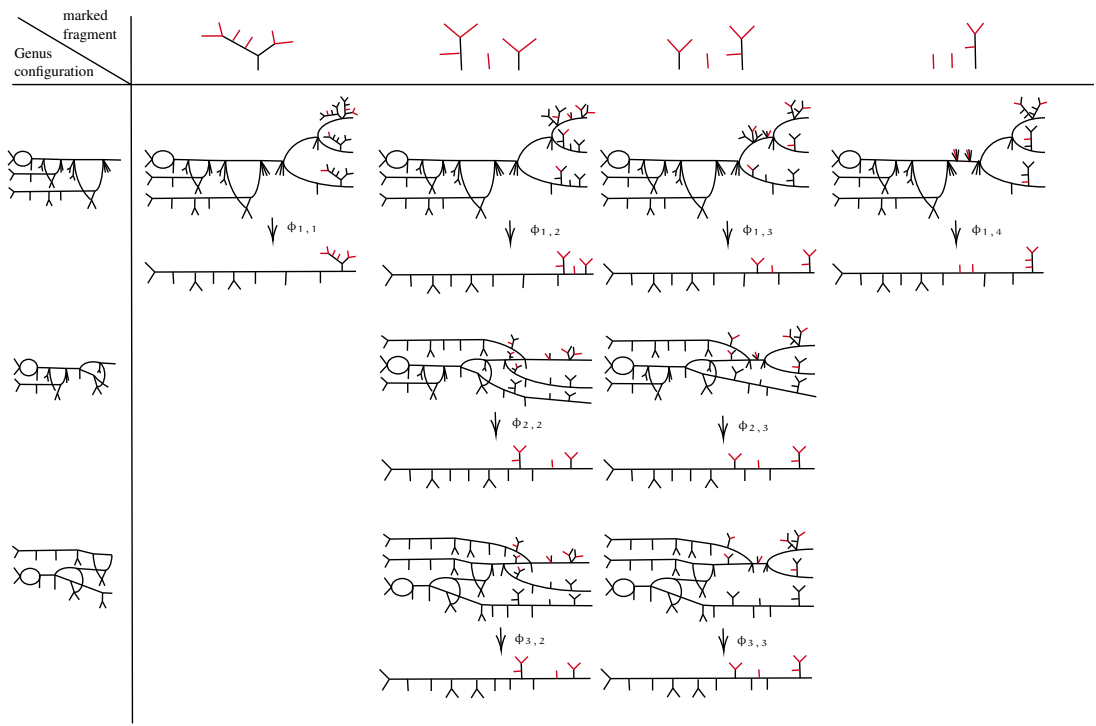}}
    \caption{The grid of solutions constructed for counting covers when $g=3, n=6$. Adapted from \cite{troptev}.}
    \label{fig:g=3grid}
\end{figure}

\begin{figure}
    \centering
    \include{Figures/genusdeggraph}
    \caption{A graph showing for each $g$ which degrees are possible for the active edge by adding $U$ or $D$, green or blue arrow, respectively. Adapted from \cite{troptev}.}
    \label{fig:genusdeggraph}
\end{figure}

\subsubsection*{Constructing a grid.}\label{sec:thegrid}
Cavalieri and Dawson then complete the enumeration of tropical covers by organizing all constructed solutions into (but not filling) a rectangular grid. The grid is organized so that the rows correspond to how the genus is formed in the genus part and the columns are indexed by marked-tree fragments, as demonstrated in Figure \ref{fig:g=3grid}. The rows can also be considered as organized by the degree of the active edge leaving the last loop of the genus configuration, this does not change the grid but some rows correspond to the same active edge degree. Each genus part corresponds to a word in the letters $U$ and $D$, and since adding $U$ and $D$ correspond to increasing, respectively decreasing, the active edge degree this can be represented by a lattice path shown in Figure \ref{fig:genusdeggraph}. Paths of length $g-1$ beginning at height $2$, due to the first genus loop always being the degree $2$ loop formed by $4$ transpositions, and never dropping below degree $1$ correspond bijectively to valid genus parts. Paths end at height $d-2i$ for some $i$ due to the fact that to create a genus part ending with active edge degree $d-2i$, genus fragment $D$ must have been added $i$ times, causing $i$ number of downward steps in the lattice path. The authors show that the number of such paths ending at height at least $d - 2i$ (for $d = g + 1$) is $A_{d,\ge d - 2i} = \binom{d - 1}{i}$.  

For a fixed genus type whose active edge has degree \(d-2i\), the compatible marked-tree fragments are those satisfying
\[
i+1\leq j\leq n-2-i=d-i.
\]
We rename the fragment denoted by \(F_1\) in \cite{troptev} as \(F_0\),
while retaining the labels \(F_2,\ldots,F_d\) for the remaining
fragments. Thus, \(F_0\) is compatible only with the genus types for which \(i=0\), and hence the column corresponding to \(F_0\) contains
\[
A_{d,\geq d}=1
\]
cover.

For a fixed \(j\in\{2,\ldots,n-2\}\), the compatibility inequalities are equivalent to
\[
0\leq i\leq \min\{j-1,d-j\}.
\]
Consequently, the covers in the column corresponding to \(F_j\) are those whose genus part has active edge degree at least
\[
d-2\min\{j-1,d-j\}.
\]
The total number of covers in this column is therefore
\[
A_{d,\geq d-2\min\{j-1,d-j\}}
=
\binom{d-1}{\min\{j-1,d-j\}}
=
\binom{d-1}{j-1},
\]
where the final equality follows from the symmetry of the binomial coefficients. Summing over all columns, we obtain
\[
\begin{aligned}
\left|
(\mathsf{F}_{\mathsf{s}}\times
\mathsf{F}_{\mathsf{t}})^{-1}(p)
\right|
&=
A_{d,\geq d}
+
\sum_{j=2}^{d}
A_{d,\geq d-2\min\{j-1,d-j\}} \\
&=
\sum_{m=0}^{d-1}\binom{d-1}{m} \\
&=
2^{d-1}
=
2^g.
\end{aligned}
\]

Since each cover contributes multiplicity one, the fiber over $p$ consists of at least $2^{g}$ distinct covers.

 The authors finally exclude any further solutions by showing no other marked fragments work, there is no other way to form a genus part of the cover with independent cycle lengths, and all joins on $\Gamma$ must occur in a row. This concludes their proof of $\Tev^\trop_g = 2^g$ when $d=g+1$ and $n=g+3$.

\section{Generalizations of tropical Tevelev degrees}

This paper considers $3$ generalizations of the previously defined tropical Tevelev degrees: varying $\ell$ to be a positive and negative integers as done in \cite{CPS:Tevdeg} with algebraic Tevelv degrees, and equipping marked points with arbitrary ramification profiles, introducing \textit{tropical generalized Tevelev degrees}, corresponding to algebraic generalized Tevelev degrees \cite{GenTev}.

\subsection{Varying $\ell$}

For an integer $\ell$, we can generalize the conditions on degree and the number of marked points to let $d=g+1+\ell$ and $n=g+3+2\ell$. These conditions still ensure that the dimension of the moduli space of tropical admissible covers $\Hur{g,d,n}^{\trop}$ equals the sum of the dimensions of $\Mgn{g,n}^\trop$ and $\Mgn{0,n}^\trop$. The degree of $(\Fs\times\Ft)$ is the \textit{tropical Tevelev degree} $\Tev_{g,\ell}^\trop$.

Computing tropical Tevelev degrees $\Tev_{g,\ell}^\trop$ is the same combinatorial inverse problem as solved in \cite{troptev}. We start in Section \ref{sec:ell>0} by considering the case of positive $\ell$, where the degree and number of marked points increases for a given genus. This case uses the techniques previously developed. Section \ref{sec:ell<0} considers the case of negative $\ell$. Navigating this case involves new techniques to count the number of covers that are no longer attainable compared to the $\ell=0$ base case due to the decrease in the degree. We require that $n\geq 3$ and thus $g+3+2\ell\geq 3$, requiring $g \geq -2\ell$.  

Together with Theorem \ref{thm:gencorrespondence}, we have provided tropical proofs of Tevelev degree computations done in \cite{CPS:Tevdeg}. These proofs first provide a combinatorial understanding of why increasing to a positive $\ell$ does not change the count of $2^g$. This is because we build the grid of solutions using the same genus sections of the covers and connect them to the marked point sections of the covers from the genus $g+2\ell$ covers from the $\ell=0$ case. The second proof provides a combinatorial understanding of the defect from $2^g$. We are subtracting away portions of the grid where ${g\choose i}-{g \choose i-1}$ corresponds to how many rows we are removing and $g-2i+1$ is how many columns wide that group of rows is.

\subsection{The generalized tropical Hurwitz space and its local degree}\label{sec:generalized-Hurwitz-space}

We now describe in greater detail the generalized cone complex,
integral structures, and weights underlying
Definition~\ref{def:tropical-generalized-Tevelev}.

For each $h=1,\ldots,k$, define the complete ramification profile
\[
\eta_h
=
\bigl(
e_{h,1},\ldots,e_{h,r_h},
1^{d-|\mu_h|}
\bigr).
\]
Thus $\mu_h$ records the distinguished marked preimages over the
marked target end $q_h$, while the additional
$d-|\mu_h|$ parts equal to one correspond to unmarked unramified
preimages. For each of the $b$ additional simple branch points, the
complete ramification profile is
\[
\eta_{k+t}=(2,1^{d-2}),
\qquad t=1,\ldots,b.
\]
These partitions determine the Hurwitz datum
\[
\mathbf h(g,\ell,\boldsymbol\mu)
=
\bigl(
g,d,k+b,\eta_1,\ldots,\eta_{k+b}
\bigr),
\]
where
\[
\boldsymbol\mu=(\mu_1,\ldots,\mu_k).
\]

The tropical Hurwitz space
\[
\mathcal H^{\trop}_{g,\ell,\mu_1,\ldots,\mu_k}
\]
is the partially labeled moduli space of tropical admissible covers
associated to this datum. The target of such a cover has $k+b$ ends.
The first $k$ target ends are labeled
$q_1,\ldots,q_k$. Over $q_h$, the source has distinguished labeled
ends
\[
p_{h,1},\ldots,p_{h,r_h}
\]
with expansion factors
\[
e_{h,1},\ldots,e_{h,r_h},
\]
and the remaining $d-|\mu_h|$ degree-one ends are unmarked.
The remaining $b$ target ends have simple ramification profile
$(2,1^{d-2})$ and, together with all of their inverse-image ends,
are unmarked.

The space
$\mathcal{H}^{\trop}_{g,\ell,\mu_1,\ldots,\mu_k}$
is equipped with the integral structures and weights induced from the
tropical admissible-cover space of \cite{CMRadmissible}, with permutation
factors arising from unmarked ends incorporated into the local Hurwitz numbers as in \cite[Convention~3.2]{troptev}.

A combinatorial type
\[
\Theta=[\varphi\colon\Gamma\longrightarrow T]
\]
of tropical admissible covers determines a rational polyhedral cone
$\sigma_\Theta$. Its coordinates may be taken to be the lengths
of the compact edges of the target $T$. If an edge $e'\subset\Gamma$
maps to an edge $e\subset T$ with expansion factor $m_{e'}$, then
their lengths satisfy
\[
\ell(e')=\frac{\ell(e)}{m_{e'}}.
\]

The integral structure is the one inherited from the tropical admissible cover space of \cite{CMRadmissible}. More explicitly, for a
compact edge $e$ of $T$, set
\[
M_e
=
\operatorname{lcm}
\bigl\{
m_{e'}:e'\subset\Gamma,\ \varphi(e')=e
\bigr\}.
\]
The primitive integral variation of the target length $\ell(e)$ is
determined by the condition that all source lengths
$\ell(e)/m_{e'}$ vary integrally. 

For a tropical admissible cover
\[
\varphi\colon\Gamma\longrightarrow T,
\]
the source morphism is
$
\Fs(\varphi)
$
where all unmarked ends are forgotten and the resulting tropical
curve is stabilized. Similarly,
$
\Ft(\varphi)
$
is obtained by forgetting the $b$ simple branch ends and stabilizing.
Together these define
\[
\Fs\times \Ft\colon
\mathcal H^{\trop}_{g,\ell,\mu_1,\ldots,\mu_k}
\longrightarrow
\mathcal M^{\trop}_{g,n}
\times
\mathcal M^{\trop}_{0,k}.
\]

The local degree of $\Fs\times\Ft$ differs slightly than in the case in \cite{troptev} (Equation \eqref{eq:locdegforlater}) due to the fact that we can no longer assume for each edge $e\in T$, there is at most one edge in $\phi^{-1}(e)$ with expansion factor $>1$. Therefore we include the factor of $\prod_{e\in CE(T)}\frac{\prod_{\phi(e') = e}m_{e'}}{M_e}$, that was previously equal to $1$. The local degree, coming from the general admissible cover weight in \cite[Definition 22]{CMRadmissible}, is given by:
 \begin{equation}\label{eq:locdegforlater}
     \deg_x (\Fs\times \Ft) =  \frac{|\Aut(\overline{\Gamma})|}{|\Aut(\phi)|}\cdot \prod_{v\in V(\Gamma)} H_v \cdot \left| \det(M_{\sigma_\Theta}) \right|\cdot\prod_{e\in CE(T)}\frac{\prod_{\phi(e') = e}m_{e'}}{M_e},
 \end{equation}
where $M_{\sigma_\Theta}$ is the matrix whose rows express the lengths of the compact edges of $\Fs(\Gamma)$ and $\Ft(T)$ as linear functions of the lengths of the edges of $\Gamma$.

For a general point
\[
p\in
\mathcal M^{\trop}_{g,n}
\times
\mathcal M^{\trop}_{0,k},
\]
the global degree is
\[
\deg(\Fs\times \Ft)
=
\sum_{x\in(\Fs\times \Ft)^{-1}(p)}
\deg_x(\Fs\times \Ft).
\]
By Definition~\ref{def:tropical-generalized-Tevelev}, this is the
tropical generalized Tevelev degree
\[
\Tev^{\trop}_{g,\ell,\mu_1,\ldots,\mu_k}.
\]
\subsection{Correspondence with algebraic generalized Tevelev degrees}\label{sec:proofofgencor}
In this section we prove Theorem \ref{thm:gencorrespondence} by adapting the correspondence argument of \cite[Theorem 1.1]{troptev} to the generalized Hurwitz data introduced in the previous section. The main difference in the generalized setting is that the admissible cover space may have several branches in its normalization over a fixed boundary point. We explain how these branches give the additional factors appearing in Equation~\eqref{eq:locdegforlater}.

Let $\boldsymbol\mu=(\mu_1,\ldots,\mu_k)$ and 
\[
\nu\colon
\Hur{g,\ell,\boldsymbol\mu}^\nu
\longrightarrow
\Hur{g,\ell,\boldsymbol\mu}
\]
denote the normalization of the admissible cover stack introduced in Section \ref{sec:gentevdeg}. Passing to the normalization does not change the degree of $\fs\times \ft$, since the normalization is an isomorphism over the locus of smooth covers.

The tropicalization results of \cite{CMRadmissible} give the factorization
 \begin{equation}\label{eq:hurwtropdiag}
     \xymatrix{ 
\Hur{g,\ell,\boldsymbol\mu}^{\nu,an}\ar[r]^{p} \ar@/^2pc/[rr]^{\trop} & \Sigma\Hur{g,\ell,\boldsymbol\mu}^{\nu,an}\ar[r]^{\trop_\Sigma} & \Hur{g,\ell,\boldsymbol\mu}^\trop},
 \end{equation}

where $p$ is the retraction to the skeleton. Moreover, source and target stabilization are compatible with tropicalization. We therefore have the following diagram
 \begin{equation}\label{eq:diagram}
\xymatrix{
\Hur{g,\ell,\boldsymbol\mu}^{\nu,an} \ar[rr]^{(f_s\times f_t)^{an}} \ar[d]_p  & & \Mgn{g,n}^{an} \times \Mgn{0,k}^{an} \ar[d]\\
\Sigma\Hur{g,\ell,\boldsymbol\mu}^{\nu,an} \ar[rr]^{f_s\times f_t\vert_{\Sigma\Hur{}}} \ar[d]_{\trop_\Sigma} & & \Sigma\Mgn{g,n}^{an} \times \Sigma\Mgn{0,k}^{an} \ar@{}[d]|*[@]{=}\\
\Hur{g,\ell,\boldsymbol\mu}^{\trop} \ar[rr]^{\Fs\times\Ft}& & \Mgn{g,n}^\trop \times \Mgn{0,k}^\trop.
}
\end{equation}

We first recall how the analytic and combinatorial local degrees are compared. By the algebraic definition of generalized Tevelev degrees, $\deg(f_s\times f_t)^{an}=b!\Tev_{g,\ell,\boldsymbol{\mu}}$. The top and middle horizontal maps of diagram \eqref{eq:diagram} have the same degree because the middle map is induced by restricting \((f_s\times f_t)^{an}\) to the skeleton, and the degree of the analytic map is computed on this skeleton by the local analytic degrees of the corresponding monomial maps, as in the proof of \cite[Theorem~1.1]{troptev}.

After refining the cone structures, let
\[
y=(\overline{\Gamma},\overline{T})
\in
\Sigma\Mgn{g,n}^{an}\times\Sigma\Mgn{0,k}^{an}
\]
be a general point contained in a maximal cone
$\sigma_{\overline{\Theta}}$, and let
\[
x\in\Sigma\Hur{g,\ell,\boldsymbol\mu}^{\nu,an}
\]
be a point in its inverse image. Then $x$ lies in the interior of a maximal cone $\sigma_\Theta$. As in the proof of \cite{troptev}, the local analytic degree of $f_s\times f_t\vert_{\Sigma\Hur{}}$ at $x$ is given by 
\[
\frac{|\Aut(\overline\Theta)|}{|\Aut(\Theta)|} [\Lambda_{\overline \Theta}: Im(\Lambda_\Theta)].
\]

We now examine where the generalized setting differs from
\cite{troptev}. Let $e$ be a compact edge of the target and suppose
the edges $e'_1,\ldots,e'_r$ above $e$ have expansion factors
$m_1,\ldots,m_r$. If $\xi$ is the smoothing parameter of the target
node and $\widetilde\xi_i$ is the smoothing parameter of the source
node corresponding to $e'_i$, then locally the admissible cover space
has equations
\[
\xi
=
\widetilde\xi_1^{m_1}
=
\cdots
=
\widetilde\xi_r^{m_r}.
\]
Set
\[
M_e=\operatorname{lcm}(m_1,\ldots,m_r).
\]
The normalization of these equations is described in
\cite[Section~4.2]{CMRadmissible}. It has
\[
\frac{m_1\cdots m_r}{M_e}
\]
branches, and on each branch there is a local parameter $s$ such that
\[
\xi=s^{M_e},
\qquad
\widetilde\xi_i=\zeta_i s^{M_e/m_i}
\]
for appropriate roots of unity $\zeta_i$.

In the setting of \cite{troptev}, the two consequences of this
normalization were trivial for the covers under consideration. In the
generalized setting, they need not be trivial. First, the different
branches of the normalization give the factor
\[
\frac{\prod_{\varphi(e')=e}m_{e'}}{M_e}
\]
in the weight of the corresponding tropical cone. The product of
these normalization factors, together with the local Hurwitz numbers
and the appropriate automorphism factors, gives the weight
$w(\Theta)$ of the corresponding tropical admissible cover cone. The
local Hurwitz numbers account for the choices of covers of the
components with the prescribed local ramification data, while the
product-over-LCM factors account for the different branches produced
by normalization. With these weights, the map
\[
\trop_\Sigma\colon
\Sigma\Hur{g,\ell,\boldsymbol\mu}^{\nu,an}
\longrightarrow
\Hur{g,\ell,\boldsymbol\mu}^{\trop}
\]
has degree one, as in \cite{CMRadmissible} and the proof of
\cite[Theorem~1.1]{troptev}.

Second, the normalization determines the integral structure on each cone of the skeleton. By \cite[Sections~4.2.2--4.2.3] {CMRadmissible}, the length of a tropical edge is the valuation of the smoothing parameter of the corresponding node. Applying the valuation associated to a point of the skeleton to the
equalities
\[
\xi=s^{M_e},
\qquad
\widetilde\xi_i=\zeta_i s^{M_e/m_i},
\]
we obtain
\[
\operatorname{length}(e)
=
M_e\operatorname{val}(s),
\qquad
\operatorname{length}(e'_i)
=
\frac{M_e}{m_i}\operatorname{val}(s).
\]
Here $\operatorname{val}(\zeta_i)=0$, since $\zeta_i$ is a root of
unity. Since $s$ is a local equation for the boundary divisor on the normalized branch, $\operatorname{val}(s)=1$ determines the primitive integral generator of the corresponding ray. In edge-length coordinates, this generator maps to \[ \left( M_e, \frac{M_e}{m_1}, \ldots, \frac{M_e}{m_r} \right). \] This is precisely the smallest positive integral vector satisfying \[ \operatorname{length}(e) = m_i\cdot\operatorname{length}(e'_i) \qquad\text{for all }i. \] After stabilization, the lengths of the surviving edges are integral linear combinations of these coordinates. Thus this vector determines the corresponding entries of $M_{\sigma_\Theta}$ and hence the lattice index of the middle horizontal map.

Thus the two effects of normalization are incorporated in different
parts of Diagram~\eqref{eq:diagram}: the normalized smoothing
parameters determine the integral lattices used to compute the local
analytic degree of the middle horizontal map, while the number of
branches of the normalization contributes to the weights that make
$\trop_\Sigma$ a map of degree one. The local Hurwitz numbers and
automorphism factors are incorporated into these weights as in
\cite{CMRadmissible}.

We have therefore shown that the analytic degree of the middle horizontal map agrees with its combinatorial degree. Since $\trop_\Sigma$ has degree one, the combinatorial degree of
$\Fs\times\Ft$ agrees with the combinatorial degree of
$\fs\times\ft\vert_{\Sigma\Hur{}}$. Together with the equality of the
degrees of the top and middle horizontal maps, this identifies the
tropical degree with the corresponding analytic degree. 
It follows that
\[
\Tev^\trop_{g,\ell,\boldsymbol\mu}
=
\Tev_{g,\ell,\boldsymbol\mu},
\]
which proves the result.

\section{Tropical Tevelev degrees $\Tev^\trop_{g,\ell}$ for $\ell>0$}
\label{sec:ell>0}

When $\ell$ becomes positive, we fix $g$ and compare to the $\ell=0$ case computed in \cite{troptev}. The degree of the cover becomes $\ell$ larger and there are $2\ell$ additional marked points for a given genus. When constructing the covers $\Gamma \to T$ for positive $\ell$, we end up with the same genus part of the cover $\gGamma \to \gT$ as the $\ell = 0$ case. From here we add on the marked fragments from the $g+2\ell$ covers of the $\ell=0$ case. Piecing together these different pieces from the $\ell=0$ case results in the same total count of $2^g$. We now go over some examples in low genera and low $\ell$ to demonstrate these ideas and then prove the total count.

\subsection{Examples in low genera and low $\ell$}
\subsubsection{$\ell=1$:}
Starting with $g=1$, in order to compute $\Tev^\trop_{1,1}$, we compute the global degree of the map \begin{equation*}
   (\Fs\times\Ft):\Hur{1,3,6}^\trop \to \Mgn{1,6}^\trop \times \Mgn{0,6}^\trop.
\end{equation*}
Consider the point $p = (\overline{\Gamma}, \overline{T})\in \Mgn{1,6}^\trop \times \Mgn{0,6}^\trop$ depicted in Figure \ref{fig:l=1g=1image}. The set $(\Fs\times \Ft)^{-1}(p)$ consists of covers $\phi: \Gamma\to T$ such that $T$ stabilizes to $\overline{T}$ when forgetting  the six ends with branching data $(2)$,  and $\Gamma$ stabilizes to $\overline\Gamma$ when forgetting the six ends with expansion factor $2$ as well as all the other unmarked ends.
\begin{figure}[tb]
    \centering
    \tikzset{every picture/.style={line width=0.75pt}} 

\begin{tikzpicture}[x=0.75pt,y=0.75pt,yscale=-1,xscale=1]

\draw  [line width=0.75]  (39,53) .. controls (39,41.95) and (54.67,33) .. (74,33) .. controls (93.33,33) and (109,41.95) .. (109,53) .. controls (109,64.05) and (93.33,73) .. (74,73) .. controls (54.67,73) and (39,64.05) .. (39,53) -- cycle ;
\draw [line width=0.75]    (109,53) -- (224.5,52) ;
\draw [color={rgb, 255:red, 208; green, 2; blue, 27 }  ,draw opacity=1 ][line width=1.5]    (132,53) -- (130.5,37) ;
\draw [color={rgb, 255:red, 208; green, 2; blue, 27 }  ,draw opacity=1 ][line width=1.5]    (153,53) -- (151.5,37) ;
\draw [color={rgb, 255:red, 208; green, 2; blue, 27 }  ,draw opacity=1 ][line width=1.5]    (224.5,52) -- (232.5,37) ;
\draw [color={rgb, 255:red, 208; green, 2; blue, 27 }  ,draw opacity=1 ][line width=1.5]    (232.5,65) -- (224.5,52) ;
\draw [line width=0.75]    (420.37,54.93) -- (308.97,55.38) ;
\draw [color={rgb, 255:red, 208; green, 2; blue, 27 }  ,draw opacity=1 ][line width=1.5]    (308.97,55.38) -- (300.19,37.81) ;
\draw [color={rgb, 255:red, 208; green, 2; blue, 27 }  ,draw opacity=1 ][line width=1.5]    (308.97,55.38) -- (293.15,64.99) ;
\draw [color={rgb, 255:red, 208; green, 2; blue, 27 }  ,draw opacity=1 ][line width=1.5]    (420.37,54.93) -- (433.88,42.87) ;
\draw [color={rgb, 255:red, 208; green, 2; blue, 27 }  ,draw opacity=1 ][line width=1.5]    (420.37,54.93) -- (437.31,68.17) ;
\draw [color={rgb, 255:red, 208; green, 2; blue, 27 }  ,draw opacity=1 ][line width=1.5]    (346,55) -- (346,39) ;
\draw [color={rgb, 255:red, 208; green, 2; blue, 27 }  ,draw opacity=1 ][line width=1.5]    (385,56) -- (385,40) ;
\draw [color={rgb, 255:red, 208; green, 2; blue, 27 }  ,draw opacity=1 ][line width=1.5]    (180,52) -- (178.5,36) ;
\draw [color={rgb, 255:red, 208; green, 2; blue, 27 }  ,draw opacity=1 ][line width=1.5]    (201,52) -- (199.5,36) ;

\draw (234.5,68.4) node [anchor=north west][inner sep=0.75pt]  [font=\footnotesize,color={rgb, 255:red, 208; green, 2; blue, 27 }  ,opacity=1 ]  {$1$};
\draw (235,23.4) node [anchor=north west][inner sep=0.75pt]  [font=\footnotesize,color={rgb, 255:red, 208; green, 2; blue, 27 }  ,opacity=1 ]  {$2$};
\draw (148,23.4) node [anchor=north west][inner sep=0.75pt]  [font=\footnotesize,color={rgb, 255:red, 208; green, 2; blue, 27 }  ,opacity=1 ]  {$5$};
\draw (126,23.4) node [anchor=north west][inner sep=0.75pt]  [font=\footnotesize,color={rgb, 255:red, 208; green, 2; blue, 27 }  ,opacity=1 ]  {$6$};
\draw (340.74,24.89) node [anchor=north west][inner sep=0.75pt]  [font=\footnotesize,color={rgb, 255:red, 208; green, 2; blue, 27 }  ,opacity=1 ,rotate=-356.58]  {$4$};
\draw (381.17,26.23) node [anchor=north west][inner sep=0.75pt]  [font=\footnotesize,color={rgb, 255:red, 208; green, 2; blue, 27 }  ,opacity=1 ,rotate=-358.17]  {$3$};
\draw (434.28,29.42) node [anchor=north west][inner sep=0.75pt]  [font=\footnotesize,color={rgb, 255:red, 208; green, 2; blue, 27 }  ,opacity=1 ,rotate=-359.91]  {$2$};
\draw (440.88,60.23) node [anchor=north west][inner sep=0.75pt]  [font=\footnotesize,color={rgb, 255:red, 208; green, 2; blue, 27 }  ,opacity=1 ,rotate=-357.22]  {$1$};
\draw (71,17.4) node [anchor=north west][inner sep=0.75pt]  [font=\footnotesize]  {$ \begin{array}{l}
x_{1}\\
\end{array}$};
\draw (111,37.4) node [anchor=north west][inner sep=0.75pt]  [font=\footnotesize]  {$x_{2}$};
\draw (134.5,38.4) node [anchor=north west][inner sep=0.75pt]  [font=\footnotesize]  {$x_{3}$};
\draw (160,37.4) node [anchor=north west][inner sep=0.75pt]  [font=\footnotesize]  {$x_{4}$};
\draw (323,38.4) node [anchor=north west][inner sep=0.75pt]  [font=\footnotesize]  {$L_{1}$};
\draw (294.74,23.89) node [anchor=north west][inner sep=0.75pt]  [font=\footnotesize,color={rgb, 255:red, 208; green, 2; blue, 27 }  ,opacity=1 ,rotate=-356.58]  {$5$};
\draw (281.74,66.89) node [anchor=north west][inner sep=0.75pt]  [font=\footnotesize,color={rgb, 255:red, 208; green, 2; blue, 27 }  ,opacity=1 ,rotate=-356.58]  {$6$};
\draw (357,38.4) node [anchor=north west][inner sep=0.75pt]  [font=\footnotesize]  {$L_{2}$};
\draw (396,38.4) node [anchor=north west][inner sep=0.75pt]  [font=\footnotesize]  {$L_{3}$};
\draw (196,22.4) node [anchor=north west][inner sep=0.75pt]  [font=\footnotesize,color={rgb, 255:red, 208; green, 2; blue, 27 }  ,opacity=1 ]  {$3$};
\draw (174,22.4) node [anchor=north west][inner sep=0.75pt]  [font=\footnotesize,color={rgb, 255:red, 208; green, 2; blue, 27 }  ,opacity=1 ]  {$4$};
\draw (182.5,37.4) node [anchor=north west][inner sep=0.75pt]  [font=\footnotesize]  {$x_{5}$};
\draw (208,36.4) node [anchor=north west][inner sep=0.75pt]  [font=\footnotesize]  {$x_{6}$};

\end{tikzpicture}
    \caption{The point $p$ in $\Mgn{1,6}^\trop \times \Mgn{0,6}^\trop$. We have $x_1<<x_2<<x_3<<x_4<<x_5<<x_6<<L_1<<L_2<<L_3$.}
    \label{fig:l=1g=1image}
\end{figure}

We start constructing $\Gamma $ and $T$ by creating the genus zero part, containing the marked points. Due to the cover being degree 3, there will be one edge cutting from the active edge to end with a transposition and at most one edge joining the active edge, because the cover is degree 3 and a genus loop ends with an active edge of degree at least 2. Therefore, we have three options for the genus zero part of $\Gamma$: join then cut, cut then join and a single cut, shown in Figure \ref{fig:g=1cutjoin}. 
\begin{figure}[tb]
    \centering
    \tikzset{every picture/.style={line width=0.75pt}} 

\begin{tikzpicture}[x=0.75pt,y=0.75pt,yscale=-1,xscale=1]

\draw [line width=1.5]    (24.5,51) -- (180.5,51) ;
\draw [line width=1.5]    (224.5,50) -- (380.5,50) ;
\draw [line width=1.5]    (424.5,50) -- (580.5,50) ;
\draw  [draw opacity=0] (24.08,15.22) .. controls (24.14,15.22) and (24.2,15.22) .. (24.26,15.22) .. controls (56.08,15.01) and (81.98,30.86) .. (82.11,50.62) .. controls (82.11,51.08) and (82.1,51.54) .. (82.08,51.99) -- (24.5,51) -- cycle ; \draw   (24.08,15.22) .. controls (24.14,15.22) and (24.2,15.22) .. (24.26,15.22) .. controls (56.08,15.01) and (81.98,30.86) .. (82.11,50.62) .. controls (82.11,51.08) and (82.1,51.54) .. (82.08,51.99) ;  
\draw  [draw opacity=0] (223.83,13.55) .. controls (223.89,13.55) and (223.95,13.55) .. (224.01,13.55) .. controls (283.51,13.16) and (331.86,28.86) .. (331.99,48.62) .. controls (331.99,49.48) and (331.91,50.34) .. (331.73,51.19) -- (224.25,49.33) -- cycle ; \draw   (223.83,13.55) .. controls (223.89,13.55) and (223.95,13.55) .. (224.01,13.55) .. controls (283.51,13.16) and (331.86,28.86) .. (331.99,48.62) .. controls (331.99,49.48) and (331.91,50.34) .. (331.73,51.19) ;  
\draw  [draw opacity=0] (180.15,84.34) .. controls (147.48,84.01) and (121.09,69.14) .. (120.7,51) -- (180.5,51) -- cycle ; \draw   (180.15,84.34) .. controls (147.48,84.01) and (121.09,69.14) .. (120.7,51) ;  
\draw  [draw opacity=0] (380.15,83.34) .. controls (320.96,82.76) and (272.9,67.93) .. (271.23,50) -- (380.5,50) -- cycle ; \draw   (380.15,83.34) .. controls (320.96,82.76) and (272.9,67.93) .. (271.23,50) ;  
\draw  [draw opacity=0] (581.63,83.53) .. controls (538.36,83.1) and (503.32,68.25) .. (502.51,50.19) -- (581.99,50.19) -- cycle ; \draw   (581.63,83.53) .. controls (538.36,83.1) and (503.32,68.25) .. (502.51,50.19) ;  

\draw (48,35.4) node [anchor=north west][inner sep=0.75pt]  [font=\footnotesize,color={rgb, 255:red, 0; green, 0; blue, 223 }  ,opacity=1 ]  {$2$};
\draw (94,35.4) node [anchor=north west][inner sep=0.75pt]  [font=\footnotesize,color={rgb, 255:red, 0; green, 0; blue, 223 }  ,opacity=1 ]  {$3$};
\draw (154,36.4) node [anchor=north west][inner sep=0.75pt]  [font=\footnotesize,color={rgb, 255:red, 0; green, 0; blue, 223 }  ,opacity=1 ]  {$2$};
\draw (239,34.4) node [anchor=north west][inner sep=0.75pt]  [font=\footnotesize,color={rgb, 255:red, 0; green, 0; blue, 223 }  ,opacity=1 ]  {$2$};
\draw (347,34.4) node [anchor=north west][inner sep=0.75pt]  [font=\footnotesize,color={rgb, 255:red, 0; green, 0; blue, 223 }  ,opacity=1 ]  {$2$};
\draw (294,34.4) node [anchor=north west][inner sep=0.75pt]  [font=\footnotesize,color={rgb, 255:red, 0; green, 0; blue, 223 }  ,opacity=1 ]  {$1$};
\draw (539,34.4) node [anchor=north west][inner sep=0.75pt]  [font=\footnotesize,color={rgb, 255:red, 0; green, 0; blue, 223 }  ,opacity=1 ]  {$2$};
\draw (458,34.4) node [anchor=north west][inner sep=0.75pt]  [font=\footnotesize,color={rgb, 255:red, 0; green, 0; blue, 223 }  ,opacity=1 ]  {$3$};

\end{tikzpicture}
    \caption{The three options for a genus zero graph containing one cut and at most one join. From left to right this includes: join then cut, cut then join and a single cut.}
    \label{fig:g=1cutjoin}
\end{figure}

When placing marked points on the options for $\Gamma$, there are three $L_i$'s, therefore three lengths need to be free to be made long. The marked point fragments shown in Figure \ref{fig:l=1g=1frags} each have 3 vertical lengths that when placed correctly on a cover, will allow for the three $L_i$'s to be made long. The fragment on the left hand side of Figure \ref{fig:l=1g=1frags} requires a cover with 3 edges on covering the tree on the left with three marked points and it needs an edge of degree $2$ covering the tree on the right hand side. Therefore, the fragment on the left side of Figure \ref{fig:l=1g=1frags} can be placed on the genus zero graph in the middle of Figure \ref{fig:g=1cutjoin}. Similarly, the fragment on the right hand side of Figure \ref{fig:l=1g=1frags} can only be placed on the graph on the left of Figure \ref{fig:g=1cutjoin}. Therefore, the rightmost graph in Figure \ref{fig:g=1cutjoin} does not contribute.

\begin{figure}[tb]
    \centering
    \tikzset{every picture/.style={line width=0.75pt}} 

\begin{tikzpicture}[x=0.75pt,y=0.75pt,yscale=-1,xscale=1]

\draw    (48.81,50.09) -- (48.81,79.33) ;
\draw [color={rgb, 255:red, 208; green, 2; blue, 27 }  ,draw opacity=1 ]   (48.81,50.09) -- (56.94,33.03) ;
\draw [color={rgb, 255:red, 208; green, 2; blue, 27 }  ,draw opacity=1 ]   (42.3,33.03) -- (48.81,50.09) ;
\draw [color={rgb, 255:red, 208; green, 2; blue, 27 }  ,draw opacity=1 ]   (49.6,62.7) -- (64.25,62.7) ;
\draw [color={rgb, 255:red, 208; green, 2; blue, 27 }  ,draw opacity=1 ]   (84.47,77.32) -- (85.27,59.2) ;
\draw    (106.1,49.08) -- (106.1,78.33) ;
\draw [color={rgb, 255:red, 208; green, 2; blue, 27 }  ,draw opacity=1 ]   (106.1,49.08) -- (114.24,32.02) ;
\draw [color={rgb, 255:red, 208; green, 2; blue, 27 }  ,draw opacity=1 ]   (99.59,32.02) -- (106.1,49.08) ;
\draw    (189.87,50.91) -- (189.87,79) ;
\draw [color={rgb, 255:red, 208; green, 2; blue, 27 }  ,draw opacity=1 ]   (189.87,50.91) -- (196.92,34.53) ;
\draw [color={rgb, 255:red, 208; green, 2; blue, 27 }  ,draw opacity=1 ]   (184.22,34.53) -- (189.87,50.91) ;
\draw [color={rgb, 255:red, 208; green, 2; blue, 27 }  ,draw opacity=1 ]   (240.48,63.99) -- (253.17,63.99) ;
\draw [color={rgb, 255:red, 208; green, 2; blue, 27 }  ,draw opacity=1 ]   (214.89,79) -- (215.58,61.6) ;
\draw    (240.48,50.91) -- (240.48,79) ;
\draw [color={rgb, 255:red, 208; green, 2; blue, 27 }  ,draw opacity=1 ]   (240.48,50.91) -- (247.53,34.53) ;
\draw [color={rgb, 255:red, 208; green, 2; blue, 27 }  ,draw opacity=1 ]   (234.83,34.53) -- (240.48,50.91) ;

\draw (67.22,59.21) node [anchor=north west][inner sep=0.75pt]  [font=\tiny,color={rgb, 255:red, 208; green, 2; blue, 27 }  ,opacity=1 ]  {$4$};
\draw (60.85,21.46) node [anchor=north west][inner sep=0.75pt]  [font=\tiny,color={rgb, 255:red, 208; green, 2; blue, 27 }  ,opacity=1 ]  {$5$};
\draw (34.59,20.96) node [anchor=north west][inner sep=0.75pt]  [font=\tiny,color={rgb, 255:red, 208; green, 2; blue, 27 }  ,opacity=1 ]  {$6$};
\draw (83.93,47.13) node [anchor=north west][inner sep=0.75pt]  [font=\tiny,color={rgb, 255:red, 208; green, 2; blue, 27 }  ,opacity=1 ]  {$3$};
\draw (118.14,20.45) node [anchor=north west][inner sep=0.75pt]  [font=\tiny,color={rgb, 255:red, 208; green, 2; blue, 27 }  ,opacity=1 ]  {$1$};
\draw (91.89,19.95) node [anchor=north west][inner sep=0.75pt]  [font=\tiny,color={rgb, 255:red, 208; green, 2; blue, 27 }  ,opacity=1 ]  {$2$};
\draw (212.24,47.97) node [anchor=north west][inner sep=0.75pt]  [font=\tiny,color={rgb, 255:red, 208; green, 2; blue, 27 }  ,opacity=1 ]  {$4$};
\draw (199.51,23.31) node [anchor=north west][inner sep=0.75pt]  [font=\tiny,color={rgb, 255:red, 208; green, 2; blue, 27 }  ,opacity=1 ]  {$5$};
\draw (176.74,22.83) node [anchor=north west][inner sep=0.75pt]  [font=\tiny,color={rgb, 255:red, 208; green, 2; blue, 27 }  ,opacity=1 ]  {$6$};
\draw (256.33,58.6) node [anchor=north west][inner sep=0.75pt]  [font=\tiny,color={rgb, 255:red, 208; green, 2; blue, 27 }  ,opacity=1 ]  {$3$};
\draw (250.12,23.31) node [anchor=north west][inner sep=0.75pt]  [font=\tiny,color={rgb, 255:red, 208; green, 2; blue, 27 }  ,opacity=1 ]  {$1$};
\draw (227.35,22.83) node [anchor=north west][inner sep=0.75pt]  [font=\tiny,color={rgb, 255:red, 208; green, 2; blue, 27 }  ,opacity=1 ]  {$2$};

\end{tikzpicture}
    \caption{Two configurations of marked points that allow for three long edges to be forgotten in the stabilization of the cover curve.}
    \label{fig:l=1g=1frags}
\end{figure}
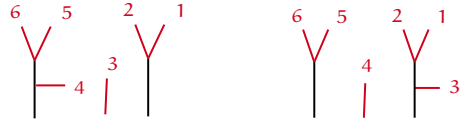

The genus zero part of the $\Gamma \to T$s that we are constructing both use three transpositions, and the Riemann-Hurwitz formula gives that there are $6$ transpositions total, therefore we have 3 transpositions left to create $\gGamma \to \gT$. There is a unique way to form a loop with 3 transpositions. Putting this loop together with the genus zero parts discussed above, we can now place the marked fragments on $\Gamma$ and $T$. 

For one cover, start with a simple loop using three transpositions, then have an edge cut from the active edge at a distance of $y_2$ away from the loop. Next, have an edge join the active edge at a distance of $y_3+y_4$ to the right of the cutting edge. Looking at $T_1$, place a tree with three marked points to the right $y_3$ from the fourth transposition. There are three trees in $\Gamma_1$ covering this tree in $T_1$. Place the marked point 6 at the end of the tree on the edge cutting from the active edge covering $T_1$, then place the marked point 5 at the end of the tree on the active edge. On the tree lying on the edge joining the active edge, place the marked point 4 at the distance $y_6$ down from 5 and 6. Moving to the right, place the marked point 3 at $y_7$ away from the edge joining, then $y_8$ to the right, place the marked points 1 and 2 on separate trees of length $y_9$. 

Place the points in a similar way for $\Gamma_2 \to T_2$, but with the cut and join switched. Together, we have the two covers shown in Figure \ref{fig:l=1g=1}.
\begin{figure}[tb]
    \centering
     \resizebox{.8\textwidth}{!}{\input{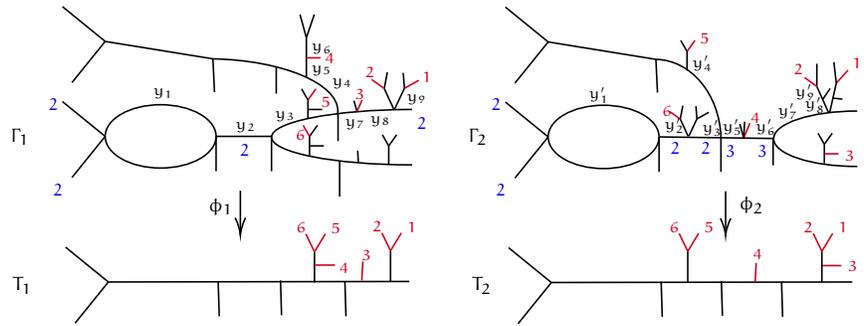}}
    \caption{The two covers in $((\Fs\times\Ft))^{-1}(p)$ for $\ell=1$ and $g=1$.}
    \label{fig:l=1g=1}
\end{figure}

We now compute the local degree of $(\Fs\times\Ft)$ at these inverse images, which gives us the multiplicities which we need to count the covers. The multiplicity of the inverse image is the product of three factors: an automorphism factor, a product of local Hurwitz numbers and a dilation factor corresponding to the determinant of the matrix representing the map $\Fs\times\Ft$. 

To calculate the local Hurwitz numbers for cover 1, every vertex in $\Gamma_1$ is either trivalent with degree one edges in all directions or two edges of degree 2 in different directions and 2 edges of degree 1 in the same direction. Both of these types of vertices have local Hurwitz number equal to 1, therefore the product of all local Hurwitz numbers is 1. Additionally, $\Gamma_2$ contains vertices with an edge of degree 3 in one direction and simple transpositions in the other 2 directions, as well as a vertex with two directions containing an edge of degree 3 and three edges of degree 1 in the third direction. The first of these vertex options has local Hurwitz number equal to 1. The second option, we are marking one of the edges of degree 1, so the local Hurwitz number is also equal to 1. All together, the product of local Hurwitz numbers is equal to 1. 

For each cover, we have one factor of 2 corresponding to switching simultaneously the two unlabeled left ends of branching type $2$ and their inverse images. Also, a second factor of 2 consists of switching the two degree $1$ edges of $\Gamma_i$ forming the loop. However, this is also a nontrivial automorphism of $\overline{\Gamma_i}$. Altogether, we have
\begin{equation}
    \frac{|\Aut(\overline{\Gamma_i})|}{|\Aut(\phi_i)|} = \frac{1}{2}.
\end{equation}

To calculate the dilation factors, we set up the following matrices representing the $x_i$'s and $L_j$'s in terms of the $y_k$'s: 
 {\scriptsize
\[\begin{array}[t]{cc}
  M_1=\begin{blockarray}{cccccccccc}
   &y_1 & y_2 & y_3 & y_4 & y_5 & y_6 & y_7 & y_8 & y_9\\
    \begin{block}{c[ccccccccc]}
      x_1 & 2 & 0 & 0 & 0 & 0 & 0 & 0 & 0 & 0\\
      x_2 & 0 & 1 & 0 & 0 & 0 & 0 & 0 & 0 & 0\\
      x_3 & 0 & 0 & 1 & 0 & 0 & 0 & 0 & 0 & 0\\
      x_4 & 0 & 0 & 0 & 1 & 0 & 0 & 0 & 0 & 0\\
      x_5 & 0 & 0 & 0 & 0 & 0 & 0 & 1 & 0 & 0\\
      x_6 & 0 & 0 & 0 & 0 & 0 & 0 & 0 & 1 & 0\\
      L_1 & 0 & 0 & 0 & 0 & 0 & 1 & 0 & 0 & 0\\
      L_2 & 0 & 0 & 0 & 1 & 1 & 0 & 2 & 0 & 0\\
      L_3 & 0 & 0 & 0 & 0 & 0 & 0 & 0 & 2 & 1\\
    \end{block}
  \end{blockarray} \\
 M_2=\begin{blockarray}{cccccccccc}
   &y'_1 & y'_2 & y'_3 & y'_4 & y'_5 & y'_6 & y'_7 & y'_8 & y'_9\\
    \begin{block}{c[ccccccccc]}
      x_1 & 2 & 0 & 0 & 0 & 0 & 0 & 0 & 0 & 0\\
      x_2 & 0 & 1 & 0 & 0 & 0 & 0 & 0 & 0 & 0\\
      x_3 & 0 & 0 & 1 & 0 & 0 & 0 & 0 & 0 & 0\\
      x_4 & 0 & 0 & 0 & 0 & 1 & 0 & 0 & 0 & 0\\
      x_5 & 0 & 0 & 0 & 0 & 0 & 1 & 0 & 0 & 0\\
      x_6 & 0 & 0 & 0 & 0 & 0 & 0 & 1 & 0 & 0\\
      L_1 & 0 & 0 & 2 & 1 & 3 & 0 & 0 & 0 & 0\\
      L_2 & 0 & 0 & 0 & 0 & 0 & 3 & 2 & 1 & 0\\
      L_3 & 0 & 0 & 0 & 0 & 0 & 0 & 0 & 0 & 1\\
    \end{block}
  \end{blockarray}
  \end{array}
\]
}
We see that $|\det M_i |=2$ for $i = 1,2$.
All together the multiplicity of each cover in $(\Fs\times\Ft)^{-1}(p)$ is $1\cdot \frac{1}{2}\cdot 2=1$. Since we have two inverse images each with multiplicity one,  we obtain $\Tev^\trop_{1,1} = 2$.

Next, looking at $g=2$, there are 4 preimages as shown in Figure \ref{fig:l=1g=2}, each with multiplicity 1, so $\Tev_{2,1}^\trop=4$.
\begin{figure}[tb]
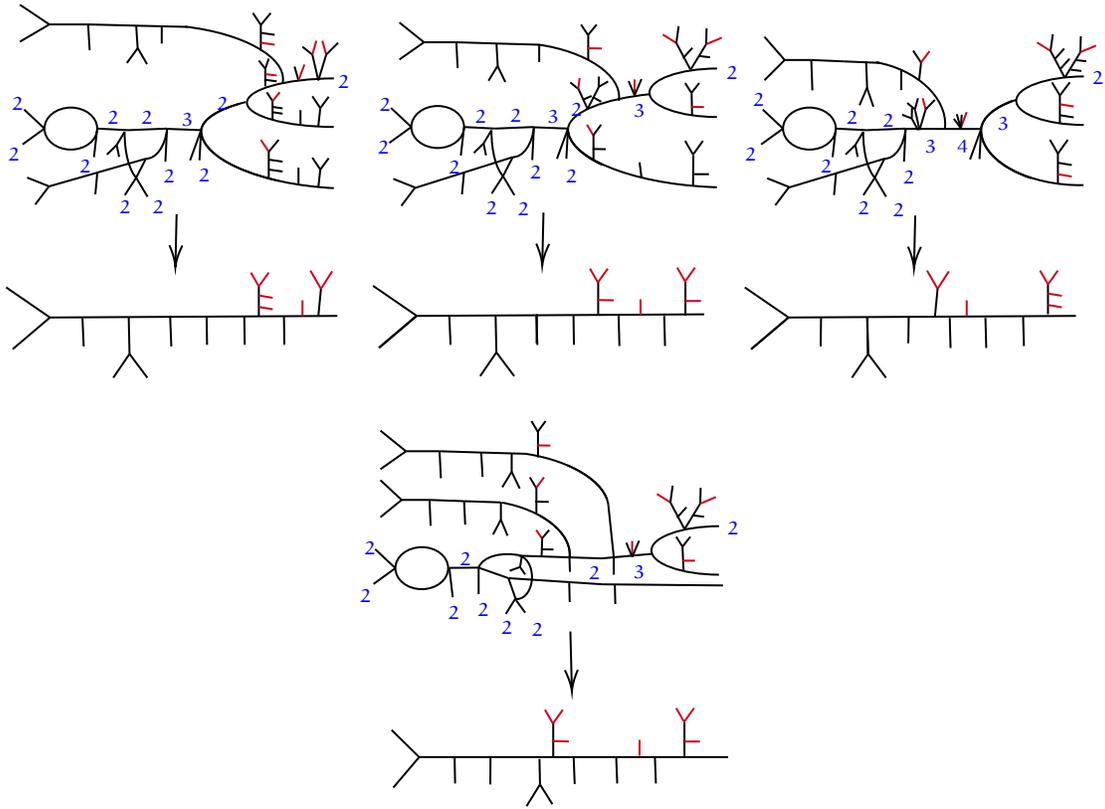

    \centering
    \include{Figures/l=1g=2}
    \caption{The four covers in $((\Fs\times\Ft))^{-1}(p)$ for $\ell=1$ and $g=2$.}
    \label{fig:l=1g=2}
\end{figure}


\subsubsection{$\ell=2$:} 
To compute $\Tev_{1,2}^\trop$, we look for the global degree of the map
\begin{equation}
    \Fs\times \Ft: \Hur{1,4,8}^\trop \to \Mgn{1,8}^\trop \times \Mgn{0,8}^\trop.
\end{equation}

We consider the point $p=(\overline{\Gamma},\overline{T})\in \Mgn{1,8}^\trop \times \Mgn{0,8}^\trop$ built according to the general form shown in Figure \ref{fig:genim}. There are 2 preimages as shown in Figure \ref{fig:l=2g=1}, each with multiplicity 1, so $\Tev_{1,2}^\trop=2$. Note that the genus is formed in 3 simple transpositions, exactly as in the $\ell=0$ and $1$ cases.

\begin{figure}[tb]
    \centering
     \resizebox{\textwidth}{!}{\input{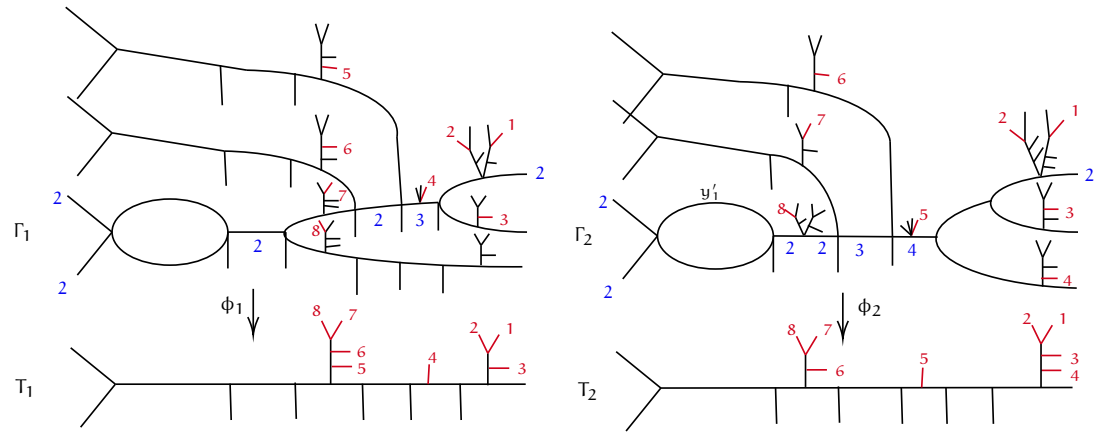}}
    \caption{The two covers in $((\Fs\times\Ft))^{-1}(p)$ for $\ell=2$ and $g=1$.}
    \label{fig:l=2g=1}
\end{figure}

\subsection{Construction of $2^g$ solutions}
In this section we make explicit and generalize the constructions from the examples in the previous section and construct $2^g$ preimages of $p$ for any genus $g$. We build covers containing all the genus, construct trees containing the $n$ marked ends, and show that the multiplicity of every cover constructed is equal to 1. 

\subsubsection*{The genus part}
We construct the genus part using the techniques in \cite{troptev}. Every cover starts with a degree 2 loop using 3 simple transpositions, and a fourth simple transposition leading into the active edge. For $g>1$, we then add one of two genus options, $U$ and $D$ as shown in Figure \ref{fig:2genusoptions}, for each loop. Note that $U$ can always be added, while $D$ can only be added if the degree of the active edge is greater than 2. If the fragment $D$ has been used $i$ times, we obtain a connected cover of degree $g+1-i$. We complete it to a degree $g+1+\ell$ cover by adding $i+\ell$ disjoint copies of the base curve, each mapping of degree one. Call these disjoint copies of the base curve that are added, \textit{joined ends}.

\subsubsection*{The marked tree part}
We construct the marked tree part by taking the fragment $F_j$ shown in Figure \ref{fig:lpositivefragments} when $1+\ell\leq j \leq n-\ell-2$. These markings are placed on the tree in the same way as in Section \ref{sec:constructing-covers} and can be placed on covers when $i+\ell+1\leq j \leq n-2-\ell-i$ where $i$ is the number of $D$ fragments. Since we are constructing the covers from the same pieces as we did in Section \ref{sec:prooftroptev}, no work is needed to show the multiplicity of every cover constructed equals 1. 

\begin{figure}[tb]
    \centering
    \tikzset{every picture/.style={line width=0.75pt}} 

\begin{tikzpicture}[x=0.65pt,y=0.65pt,yscale=-1,xscale=1]

\draw [line width=1.5]    (109.53,225.47) -- (110.33,107.67) ;
\draw [line width=1.5]    (229.53,225.87) -- (230.33,107.27) ;
\draw [color={rgb, 255:red, 208; green, 2; blue, 27 }  ,draw opacity=1 ][line width=1.5]    (169.93,190.87) -- (169.93,226.07) ;
\draw [color={rgb, 255:red, 208; green, 2; blue, 27 }  ,draw opacity=1 ][line width=1.5]    (141.53,75.27) -- (110.33,107.67) ;
\draw [color={rgb, 255:red, 208; green, 2; blue, 27 }  ,draw opacity=1 ][line width=1.5]    (261.53,74.87) -- (230.33,107.27) ;
\draw [color={rgb, 255:red, 208; green, 2; blue, 27 }  ,draw opacity=1 ][line width=1.5]    (79.93,77.27) -- (110.33,107.67) ;
\draw [color={rgb, 255:red, 208; green, 2; blue, 27 }  ,draw opacity=1 ][line width=1.5]    (199.93,76.87) -- (230.33,107.27) ;
\draw [color={rgb, 255:red, 208; green, 2; blue, 27 }  ,draw opacity=1 ][line width=1.5]    (73.93,126.87) -- (109.53,126.87) ;
\draw [color={rgb, 255:red, 208; green, 2; blue, 27 }  ,draw opacity=1 ][line width=1.5]    (73.13,186.87) -- (108.73,186.87) ;
\draw [color={rgb, 255:red, 208; green, 2; blue, 27 }  ,draw opacity=1 ][line width=1.5]    (231.13,126.47) -- (266.73,126.47) ;
\draw [color={rgb, 255:red, 208; green, 2; blue, 27 }  ,draw opacity=1 ][line width=1.5]    (230.33,196.87) -- (265.93,196.87) ;
\draw  [color={rgb, 255:red, 208; green, 2; blue, 27 }  ,draw opacity=1 ][fill={rgb, 255:red, 208; green, 2; blue, 27 }  ,fill opacity=1 ][line width=1.5]  (87.93,144.27) .. controls (87.93,143.16) and (88.83,142.27) .. (89.93,142.27) .. controls (91.04,142.27) and (91.93,143.16) .. (91.93,144.27) .. controls (91.93,145.37) and (91.04,146.27) .. (89.93,146.27) .. controls (88.83,146.27) and (87.93,145.37) .. (87.93,144.27) -- cycle ;
\draw  [color={rgb, 255:red, 208; green, 2; blue, 27 }  ,draw opacity=1 ][fill={rgb, 255:red, 208; green, 2; blue, 27 }  ,fill opacity=1 ][line width=1.5]  (87.93,157.07) .. controls (87.93,155.96) and (88.83,155.07) .. (89.93,155.07) .. controls (91.04,155.07) and (91.93,155.96) .. (91.93,157.07) .. controls (91.93,158.17) and (91.04,159.07) .. (89.93,159.07) .. controls (88.83,159.07) and (87.93,158.17) .. (87.93,157.07) -- cycle ;
\draw  [color={rgb, 255:red, 208; green, 2; blue, 27 }  ,draw opacity=1 ][fill={rgb, 255:red, 208; green, 2; blue, 27 }  ,fill opacity=1 ][line width=1.5]  (87.93,169.07) .. controls (87.93,167.96) and (88.83,167.07) .. (89.93,167.07) .. controls (91.04,167.07) and (91.93,167.96) .. (91.93,169.07) .. controls (91.93,170.17) and (91.04,171.07) .. (89.93,171.07) .. controls (88.83,171.07) and (87.93,170.17) .. (87.93,169.07) -- cycle ;
\draw  [color={rgb, 255:red, 208; green, 2; blue, 27 }  ,draw opacity=1 ][fill={rgb, 255:red, 208; green, 2; blue, 27 }  ,fill opacity=1 ][line width=1.5]  (245.53,145.07) .. controls (245.53,143.96) and (246.43,143.07) .. (247.53,143.07) .. controls (248.64,143.07) and (249.53,143.96) .. (249.53,145.07) .. controls (249.53,146.17) and (248.64,147.07) .. (247.53,147.07) .. controls (246.43,147.07) and (245.53,146.17) .. (245.53,145.07) -- cycle ;
\draw  [color={rgb, 255:red, 208; green, 2; blue, 27 }  ,draw opacity=1 ][fill={rgb, 255:red, 208; green, 2; blue, 27 }  ,fill opacity=1 ][line width=1.5]  (245.53,157.87) .. controls (245.53,156.76) and (246.43,155.87) .. (247.53,155.87) .. controls (248.64,155.87) and (249.53,156.76) .. (249.53,157.87) .. controls (249.53,158.97) and (248.64,159.87) .. (247.53,159.87) .. controls (246.43,159.87) and (245.53,158.97) .. (245.53,157.87) -- cycle ;
\draw  [color={rgb, 255:red, 208; green, 2; blue, 27 }  ,draw opacity=1 ][fill={rgb, 255:red, 208; green, 2; blue, 27 }  ,fill opacity=1 ][line width=1.5]  (245.53,169.87) .. controls (245.53,168.76) and (246.43,167.87) .. (247.53,167.87) .. controls (248.64,167.87) and (249.53,168.76) .. (249.53,169.87) .. controls (249.53,170.97) and (248.64,171.87) .. (247.53,171.87) .. controls (246.43,171.87) and (245.53,170.97) .. (245.53,169.87) -- cycle ;

\draw (158.33,13.47) node [anchor=north west][inner sep=0.75pt]  [font=\LARGE]  {$F_{j}$};
\draw (266.73,54.47) node [anchor=north west][inner sep=0.75pt]    {$\textcolor[rgb]{0.82,0.01,0.11}{1}$};
\draw (190.33,57.67) node [anchor=north west][inner sep=0.75pt]    {$\textcolor[rgb]{0.82,0.01,0.11}{2}$};
\draw (274.33,116.07) node [anchor=north west][inner sep=0.75pt]    {$\textcolor[rgb]{0.82,0.01,0.11}{3}$};
\draw (66.73,56.87) node [anchor=north west][inner sep=0.75pt]    {$\textcolor[rgb]{0.82,0.01,0.11}{n}$};
\draw (121.13,60.07) node [anchor=north west][inner sep=0.75pt]    {$\textcolor[rgb]{0.82,0.01,0.11}{n-1}$};
\draw (30.33,117.67) node [anchor=north west][inner sep=0.75pt]    {$\textcolor[rgb]{0.82,0.01,0.11}{n-2}$};
\draw (271.13,188.47) node [anchor=north west][inner sep=0.75pt]    {$\textcolor[rgb]{0.82,0.01,0.11}{j}$};
\draw (156.73,170.07) node [anchor=north west][inner sep=0.75pt]    {$\textcolor[rgb]{0.82,0.01,0.11}{j+1}$};
\draw (28.33,178.07) node [anchor=north west][inner sep=0.75pt]    {$\textcolor[rgb]{0.82,0.01,0.11}{j+2}$};

\end{tikzpicture}
    \caption{The marked fragment that attaches to the horizontal edge of the target graph to obtain the base graph $T$.}
    \label{fig:lpositivefragments}
\end{figure}
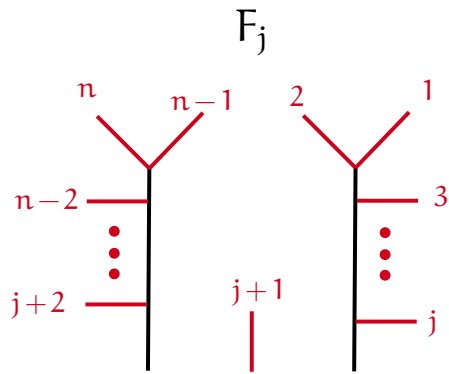
The number of marked fragments when $\ell>0$ is $n-2-2\ell=g+3+2\ell-2-2\ell=g+1$. Therefore we construct a grid of solutions as in Section \ref{sec:thegrid} with the same number as columns because there are the same number as marked fragments. The examples and construction demonstrated that the genus part of the covers $\Gamma\to T$ in the positive $\ell$ case are constructed identically to the genus parts in the $\ell=0$ case. Thus, the number of rows of the grids are equal in the positive $\ell$ and $\ell=0$ cases. There is a bijection between this solution set and the solution set constructed in \cite{troptev}. Since each cover counts with multiplicity equal to one, we have thus far shown that $\Tev_{g,\ell}^\trop\geq 2^g$ for positive $\ell$. 

\subsection{Excluding further solutions} \label{sec:ellposexclude}
In this section, we exclude any further cover $\Gamma \to T$ from mapping to  the chosen point $p= (\overline{\Gamma},\overline{T})\in \Mgn{g,n}^\trop\times \Mgn{0,n}^\trop$. We follow similar arguments made in \cite{troptev}.

\subsubsection*{Fragments attaching to the active edge}
The same argument as in Section 4.3 in \cite{troptev} excludes all marked fragments besides those shown in Figure \ref{fig:markfrag}. We now exclude marked fragments $F_0$ and $F_{n-2}$. 

Recall that marked fragments $F_0$ and $F_{n-2}$ can only be placed on covers that have zero joined ends added to complete the degree $d=g+1$ cover. When $\ell$ is positive, we must add $\ell$ joined ends to obtain a cover of degree $d+g+1+\ell$. Therefore, even when fragment $D$ is used zero times, the number of joined ends is greater than zero. Moreover, only the marked fragments shown in Figure \ref{fig:lpositivefragments} can be attached to covers.

\subsubsection*{Splitting of transpositions}
Recall that the genus part of the graph $\gGamma \to \gT $ contains at least $3g$ transpositions. 

\begin{claim}
    The marked tree part of the graph needs at least $g+2\ell$ transpositions.
\end{claim}
\begin{proof}
    There are $g+2\ell$ lengths $L_k$'s, each of which is much longer than the lengths on $\overline{\Gamma}$. Therefore, all marked points from $3$ to $n-2=g+1+2\ell$ must be stabilizing to the active path from an edge that it is the only mark on. There are at least $g+2\ell-1$ transpositions used to form such edges. Finally, in order to not have a relation between the lengths $x_{4g+2\ell}$ and $L_{g+2\ell}$, the branch that the marks $1$ and $2$ lie on must attach to an edge of degree at least $2$. In total, at least $g+2\ell$ transpositions are needed for the marked tree part of $\Gamma \to T$.
\end{proof}

The total number of transpositions is $4g+2\ell$, and therefore the number of transpositions on $\gGamma \to \gT$ is exactly $3g$ and the number on the marked tree part of $\Gamma \to T$ is $g+2\ell$.

Knowing that any new loop is added with exactly three transpositions we use the proof in Section 4.3 in \cite{troptev} to rule out all loop fragments besides $U$ and $D$. That proof also gives that all covers have a marked tree part that looks like some number of cuts followed by all joins and ending with the rest of the cuts. But then all possible solutions to our problem must be of the form of those we have exhibited, and there can be no more solutions and we have shown Theorem \ref{thm:ellpositive}.

\section{Tropical Tevelev degrees $\Tev^\trop_{g,\ell}$ for $\ell<0$}
\label{sec:ell<0}
When $\ell$ is negative, both the degree of the covers and the number of marked points become less than what they are in the $\ell=0$ case for a fixed genus $g$. Due to the degree of the covers being lower, we can not form the genus section of the covers in all of the same ways. For example, a genus part of the cover, $\gGamma \to \gT$, from the $\ell=0$ case used all of the degree to form the genus, that can not be made when the degree decreases. This leads to fewer ways to form the genus section and thus less rows in the grid. We can count the number of preimages using the same grid as used in the $\ell=0$ case but there are rows removed corresponding to the $\gGamma \to \gT$ from the $\ell=0$ case that are not possible. In this section, we go through examples in low genera when $\ell=-1$, then construct the number of solutions, and finally rule out any other possible solutions.

\subsection{Examples in low genera and $\ell$ close to zero}
\subsubsection{$\ell=-1$:} 

Starting with $g=2$, in order to compute $\Tev_{2,-1}^\trop$, we must compute the global degree of the map \[(\Fs\times\Ft):\Hur{2,2,3}^\trop \to \Mgn{2,3}^\trop \times \Mgn{0,3}^\trop.\]
Consider the point $p=(\overline{\Gamma},\overline{T})\in \Mgn{2,3}^\trop \times \Mgn{0,3}^\trop$ depicted in Figure \ref{fig:genim}. 

We start by constructing $\gGamma \to \gT$. The Riemann-Hurwitz formula states that there are 6 simple transpositions. There is one way to form a degree 2 cover that contains two disjoint loops. 

We then complete $\gGamma \to \gT$ to $\Gamma \to T$ by adding 3 marked points. Place marked point 3 at distance $y_5$ from the second loop, and place marked points 1 and 2 $y_6$ to the right on 3. There are no other ways to place the marked points so that $\Gamma \to T$ is a preimage of $p$. Therefore, we get one cover in $(\Fs\times \Ft)^{-1}(p)$ shown in Figure \ref{fig:l=-1,g=2}. We now compute the local degree of $(\Fs\times\Ft)$ at this inverse image. 

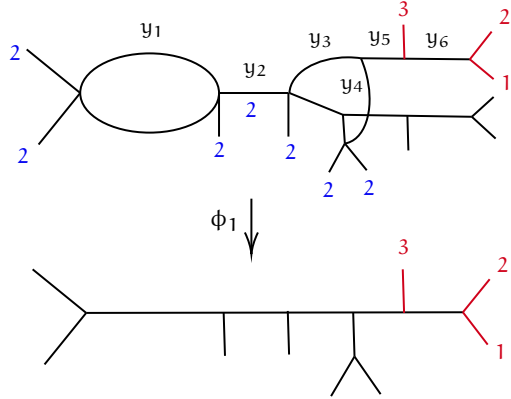
\begin{figure}[tb]
    \centering
    \tikzset{every picture/.style={line width=0.75pt}} 

\begin{tikzpicture}[x=0.75pt,y=0.75pt,yscale=-1,xscale=1]

\draw    (91,67) -- (64,45) ;
\draw    (91,67) -- (70,93) ;
\draw  [line width=0.75]  (91,67) .. controls (91,55.95) and (106.67,47) .. (126,47) .. controls (145.33,47) and (161,55.95) .. (161,67) .. controls (161,78.05) and (145.33,87) .. (126,87) .. controls (106.67,87) and (91,78.05) .. (91,67) -- cycle ;
\draw    (161,67) -- (196,67) ;
\draw    (161,67) -- (161,89) ;
\draw    (196,67) -- (196,89) ;
\draw  [draw opacity=0] (196.51,67) .. controls (196.93,57.01) and (209.97,49) .. (226,49) .. controls (228.26,49) and (230.45,49.16) .. (232.57,49.46) -- (226,67.5) -- cycle ; \draw   (196.51,67) .. controls (196.93,57.01) and (209.97,49) .. (226,49) .. controls (228.26,49) and (230.45,49.16) .. (232.57,49.46) ;  
\draw  [draw opacity=0] (232.57,49.46) .. controls (233.42,51.47) and (234.18,53.59) .. (234.82,55.8) .. controls (239.79,73.02) and (235.73,89.32) .. (225.75,92.2) .. controls (225.32,92.32) and (224.89,92.42) .. (224.45,92.49) -- (216.75,61.01) -- cycle ; \draw   (232.57,49.46) .. controls (233.42,51.47) and (234.18,53.59) .. (234.82,55.8) .. controls (239.79,73.02) and (235.73,89.32) .. (225.75,92.2) .. controls (225.32,92.32) and (224.89,92.42) .. (224.45,92.49) ;  
\draw    (196.51,67) -- (223.5,78) ;
\draw    (223.5,78) -- (224.45,92.49) ;
\draw    (224.45,92.49) -- (235.5,106) ;
\draw    (224.45,92.49) -- (216.5,107) ;
\draw    (223.5,78) -- (288.5,79) ;
\draw    (232.57,49.46) -- (287.5,50) ;
\draw    (256,78.5) -- (256.5,96) ;
\draw [color={rgb, 255:red, 208; green, 2; blue, 27 }  ,draw opacity=1 ]   (254,32.5) -- (254.5,50) ;
\draw [color={rgb, 255:red, 208; green, 2; blue, 27 }  ,draw opacity=1 ]   (298.5,36) -- (287.5,50) ;
\draw [color={rgb, 255:red, 208; green, 2; blue, 27 }  ,draw opacity=1 ]   (287.5,50) -- (299.5,60) ;
\draw    (299.5,69) -- (288.5,79) ;
\draw    (288.5,79) -- (300.5,90) ;
\draw    (94,178) -- (67,156) ;
\draw    (94,178) -- (73,204) ;
\draw    (94,178) -- (284,177.67) ;
\draw    (163.33,177.67) -- (164,199.67) ;
\draw    (195.67,177) -- (196.33,199) ;
\draw    (177.33,120.33) -- (177.33,146.33) ;
\draw [shift={(177.33,148.33)}, rotate = 270] [color={rgb, 255:red, 0; green, 0; blue, 0 }  ][line width=0.75]    (10.93,-3.29) .. controls (6.95,-1.4) and (3.31,-0.3) .. (0,0) .. controls (3.31,0.3) and (6.95,1.4) .. (10.93,3.29)   ;
\draw    (228.67,178) -- (229.33,200) ;
\draw    (229.33,200) -- (217.5,220) ;
\draw    (229.33,200) -- (242.5,219) ;
\draw [color={rgb, 255:red, 208; green, 2; blue, 27 }  ,draw opacity=1 ]   (253.67,156) -- (254.33,178) ;
\draw [color={rgb, 255:red, 208; green, 2; blue, 27 }  ,draw opacity=1 ]   (297.5,161) -- (284,177.67) ;
\draw [color={rgb, 255:red, 208; green, 2; blue, 27 }  ,draw opacity=1 ]   (284,177.67) -- (296.5,194) ;

\draw (54,41.4) node [anchor=north west][inner sep=0.75pt]  [font=\scriptsize,color={rgb, 255:red, 1; green, 1; blue, 240 }  ,opacity=1 ]  {$2$};
\draw (58,93.4) node [anchor=north west][inner sep=0.75pt]  [font=\scriptsize,color={rgb, 255:red, 1; green, 1; blue, 240 }  ,opacity=1 ]  {$2$};
\draw (157,90.4) node [anchor=north west][inner sep=0.75pt]  [font=\scriptsize,color={rgb, 255:red, 1; green, 1; blue, 240 }  ,opacity=1 ]  {$2$};
\draw (193,91.4) node [anchor=north west][inner sep=0.75pt]  [font=\scriptsize,color={rgb, 255:red, 1; green, 1; blue, 240 }  ,opacity=1 ]  {$2$};
\draw (212,110.4) node [anchor=north west][inner sep=0.75pt]  [font=\scriptsize,color={rgb, 255:red, 1; green, 1; blue, 240 }  ,opacity=1 ]  {$2$};
\draw (232.5,109.4) node [anchor=north west][inner sep=0.75pt]  [font=\scriptsize,color={rgb, 255:red, 1; green, 1; blue, 240 }  ,opacity=1 ]  {$2$};
\draw (173,70.4) node [anchor=north west][inner sep=0.75pt]  [font=\scriptsize,color={rgb, 255:red, 1; green, 1; blue, 240 }  ,opacity=1 ]  {$2$};
\draw (250,19.4) node [anchor=north west][inner sep=0.75pt]  [font=\scriptsize,color={rgb, 255:red, 208; green, 2; blue, 27 }  ,opacity=1 ]  {$3$};
\draw (301,23.4) node [anchor=north west][inner sep=0.75pt]  [font=\scriptsize,color={rgb, 255:red, 208; green, 2; blue, 27 }  ,opacity=1 ]  {$2$};
\draw (302,56.4) node [anchor=north west][inner sep=0.75pt]  [font=\scriptsize,color={rgb, 255:red, 208; green, 2; blue, 27 }  ,opacity=1 ]  {$1$};
\draw (120,30.4) node [anchor=north west][inner sep=0.75pt]  [font=\scriptsize]  {$y_{1}$};
\draw (172,50.4) node [anchor=north west][inner sep=0.75pt]  [font=\scriptsize]  {$y_{2}$};
\draw (205,35.4) node [anchor=north west][inner sep=0.75pt]  [font=\scriptsize]  {$y_{3}$};
\draw (221,58.4) node [anchor=north west][inner sep=0.75pt]  [font=\scriptsize]  {$y_{4}$};
\draw (235,34.4) node [anchor=north west][inner sep=0.75pt]  [font=\scriptsize]  {$y_{5}$};
\draw (264,35.4) node [anchor=north west][inner sep=0.75pt]  [font=\scriptsize]  {$y_{6}$};
\draw (250,139.4) node [anchor=north west][inner sep=0.75pt]  [font=\scriptsize,color={rgb, 255:red, 208; green, 2; blue, 27 }  ,opacity=1 ]  {$3$};
\draw (300,148.4) node [anchor=north west][inner sep=0.75pt]  [font=\scriptsize,color={rgb, 255:red, 208; green, 2; blue, 27 }  ,opacity=1 ]  {$2$};
\draw (299,192.4) node [anchor=north west][inner sep=0.75pt]  [font=\scriptsize,color={rgb, 255:red, 208; green, 2; blue, 27 }  ,opacity=1 ]  {$1$};
\draw (155,124.4) node [anchor=north west][inner sep=0.75pt]  [font=\footnotesize]  {$\phi _{1}$};

\end{tikzpicture}
    \caption{The one cover in $(\Fs\times \Ft)^{-1}(p)$ for $\ell=-1$ and $g=2$.}
    \label{fig:l=-1,g=2}
\end{figure}

To calculate the local Hurwitz number, every vertex in $\Gamma_1$ is either trivalent with degree one edges in all directions or two edges of degree 2 in different directions and 2 edges of degree 1 in the same direction. Both of these types of vertices have local Hurwitz number equal to 1, therefore the product of all local Hurwitz numbers is 1. 

The cover has an automorphism factor of 2 corresponding to switching each pair of ends of branching type $2$ and their inverse images. Altogether, we have \begin{equation*}
    \frac{|\Aut(\overline{\Gamma_1})|}{|\Aut(\phi_1)|} = \frac{1}{4}.
\end{equation*}

To calculate the dilation factor, we set up the following matrix representing the $x_i$'s in terms of the $y_k$'s: 
\[M_1=\begin{blockarray}{ccccccc}
   &y_1 & y_2 & y_3 & y_4 & y_5 & y_6 \\
    \begin{block}{c[cccccc]}
      x_1 & 2 & 0 & 0 & 0 & 0 & 0\\
      x_2 & 0 & 1 & 0 & 0 & 0 & 0\\
      x_3 & 0 & 0 & 1 & 0 & 0 & 0\\
      x_4 & 0 & 0 & 1 & 2 & 0 & 0\\
      x_5 & 0 & 0 & 0 & 0 & 1 & 0\\
      x_6 & 0 & 0 & 0 & 0 & 0 & 1\\
    \end{block}
  \end{blockarray}
\]
We see that $|\det M_1 |=4$.
All together the multiplicity of each cover in $(\Fs\times\Ft)^{-1}(p)$ is $1\cdot \frac{1}{4}\cdot 4=1$. Since we have one inverse image with multiplicity one,  we obtain $\Tev^\trop_{2,-1} = 1$.

To interpret $\Tev^\trop_{2,-1}$ in another way, we compare the grid of solutions for $\ell = -1$ and $g=2$ with the grid of solutions for $\ell=0$ and $g=2$ shown in Figure \ref{fig:l=0g=2}. Although the marked point configurations and total degrees of covers are different, we obtain organizations according to the same active edge degrees. When $\ell=-1$, we lose the top row of the $\ell=0$ grid corresponding to the genus configuration having an active edge of degree 3. The cover $\phi_{2,2} $ in Figure \ref{fig:l=0g=2} with the joined end removed and no marked points is the same as $\phi_1$ in Figure \ref{fig:l=-1,g=2} with no marked points as shown in Figure \ref{fig:removejoinedend}. 

\begin{figure}[tb]
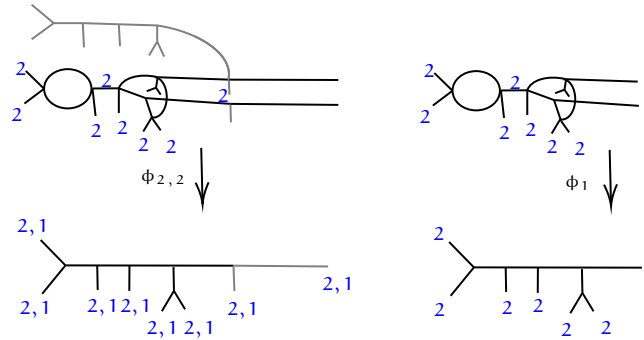

    \centering
    \include{Figures/removejoinedend}
    \caption{The cover $\phi_{2,2} $ with the joined end removed and no marked points is the same as $\phi_1$ with no marked points.}
    \label{fig:removejoinedend}
\end{figure}

Moving to genus 3, to compute $\Tev_{3,-1}^\trop$, we must compute the global degree of the map \[(\Fs\times\Ft):\Hur{3,3,4}^\trop \to \Mgn{3,4}^\trop \times \Mgn{0,4}^\trop.\]

\begin{figure}[tb]
    \centering
     \resizebox{\textwidth}{!}{\input{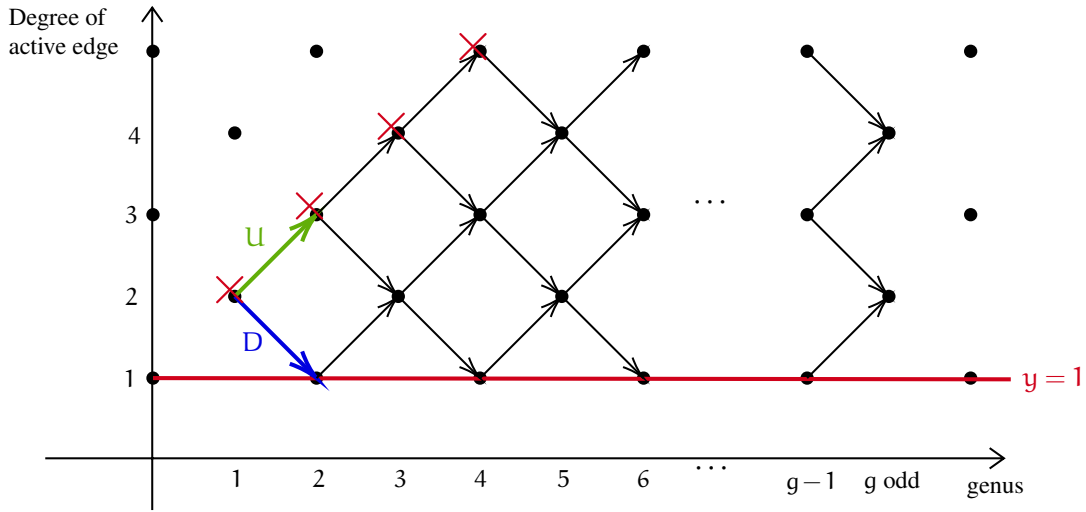}}
    \caption{A graph showing for genus $g$ which degrees are possible for the active edge by adding $U$ and $D$. For $\ell=-1$, paths ending at vertices along the diagonal are not possible because $d=g$.}
    \label{fig:l=-1genuspaths}
\end{figure}
We start by constructing the genus part of the cover $\gGamma \to \gT$. Looking at Figure \ref{fig:l=-1genuspaths}, for $g=3$, there are 3 possible paths, but one of those paths ends at a vertex with active edge degree 4. Since $\ell=-1$, the degree of the cover is $3$, therefore, paths to that vertex are not possible. We conclude that there are two possible ways to construct $\gGamma \to \gT$, one by adding $U$ then $D$ and one by doing the opposite. Recall $U$ and $D$ are introduced in Figure \ref{fig:2genusoptions}. Note, both of these options have an active edge of degree 2. 

To complete the cover, we add $4$ marked points. The task of adding $4$ marked points to an active edge of degree 2 was done in detail in Section \ref{sec:constructing-covers}, there are 2 ways of doing so. Putting together the ways to form $\gGamma \to \gT$ with the ways to add marked points, we have $4$ total covers in the preimage of $p$, giving that $\Tev_{3,-1}^\trop=4$. Comparing the genus 3 grid of solutions when $\ell=-1$ in Figure \ref{fig:l=-1,g=3} to the grid of solutions when $\ell=0$ shown in Figure \ref{fig:g=3grid}, notice that the prior is missing the top row of solutions. 

\begin{figure}[tb]
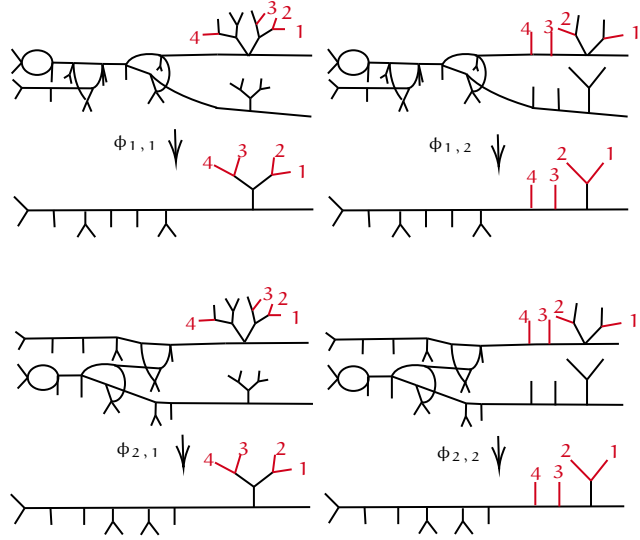

    \centering
    \include{Figures/l=-1,g=3}
    \caption{The four covers in $(\Fs \times \Ft)^{-1}(p)$ when $\ell=-1$ and $g=3$.}
    \label{fig:l=-1,g=3}
\end{figure}

\subsection{Construction of solutions}
\subsubsection*{$\ell=-1$:}
We start constructing the number of solutions for general $g$ for the case when $\ell=-1$. We construct the genus part in a similar fashion to the construction in the previous section, but there are fewer options for any given genus $g$ due to the condition that $d=g+1+\ell$. When $\ell=-1$, $d=g$, which restricts paths from ending at vertices on the diagonal, i.e. there is at least one $D$ added, as shown in Figure \ref{fig:l=-1genuspaths}. In terms of the grid of solutions built in the $\ell=0$ case, we lose the rows of solutions corresponding to an active edge of degree $g+1$. This number of rows is equal to the number of paths ending at a vertex on the diagonal, which is equal to ${g\choose0}=1$. 

The number of marked points $n=g+3-2=g+1$, is the same number of marked points as in the $\ell=0$ case when the genus is $2$ less. The marked fragments presented in Section \ref{sec:prooftroptev} for $n$ marked points can be placed on the same active edge degrees as we have in our grid of solutions. 

Altogether, solutions can be thought of as the number of solutions when $\ell=0$ minus the solutions corresponding to an active edge of degree $g+1$. The number of columns that are filled in these rows is equal to the number of marked fragments, $g+1$. The total number of solutions constructed is equal to $2^g-(g+1){g\choose0}$.

\subsubsection*{$\ell<0$: The genus part}
We again construct the genus part in a similar fashion to the construction in the $\ell=0$ case, but there are fewer options for any given genus $g$. When $\ell$ is negative, $d=g+1+\ell$, there must be at least $-\ell$ $D$'s added, giving an active edge of degree less than or equal to $g+1+2\ell$ in addition to $-\ell$ free ends coming from each of the $-\ell$ $D$'s added. Therefore, when $\ell $ is negative, paths cannot end at vertices less than or equal to $\ell$ steps down from the diagonal. In terms of the grid of solutions, we lose the rows of solutions corresponding to active edge degrees greater than $g+1+2\ell$. The number of rows of the grid that come from a given active edge degree is equal to the number of paths to the corresponding vertex, which is equal to $({g\choose{-\ell}}-{g\choose{-\ell-1}})$ for active edge degree $g+1+2\ell$.

\subsubsection*{The marked tree part}
For genus $g$, there are $n-2=g+1+2\ell$ fragments coming from the $g+2\ell$ case of the $\ell=0$ case. These are placed on the marked part of the cover in the same way. Fragment $F_0$ can be placed only on covers when $\ell=-i$, where $i$ is the number of $D$ fragments added. Fragment $F_j$ can be placed on covers when $i+\ell+1\leq j \leq d-i$. Since we are constructing the covers from the same pieces as used in Section \ref{sec:prooftroptev}, no work is needed to show the multiplicity of every cover constructed equals 1. 

When $\ell=0$, we have shown that every type of fragment can be placed on the cover with the highest possible active degree, therefore not having those covers, takes away $(g+1) {g\choose{0}}$ covers from $2^g$. For every vertex down from the diagonal in Figure \ref{fig:l=-1genuspaths} that corresponds to unattainable paths, there are two fewer marked fragments that can be placed. Each vertex corresponds to $(g+1-2i)({g\choose{i}}-{g\choose{i-1}})$ covers where $i$ is the number of vertices down, also the number of $D$ fragments added, because there are $(g+1-2i)$ marked fragment options and $({g\choose{i}}-{g\choose{i-1}})$ genus part options. 

For positive integer $g$ and negative integer $\ell$, \[\Tev^\trop_{g,\ell}\geq 2^g - \sum_{i=0}^{-\ell-1}\bigg(g-2i+1\bigg)\bigg({g\choose i}-{g \choose i-1}\bigg).\]

\subsection{Excluding further solutions}
In this section, we exclude any further cover $\Gamma \to T$ from mapping to  the chosen point $p= (\overline{\Gamma},\overline{T})\in \Mgn{g,n}^\trop\times \Mgn{0,n}^\trop$. 

For $\ell$ negative, there are $n=g+3+2\ell$ markings. This is the same number of marked points as the $g+2\ell$ case of $\ell=0$. These marked points can be placed in marked fragments as shown in Figure \ref{fig:markfrag}. The same argument as in Section \ref{sec:prooftroptev} excludes all other marked fragments.

Recall from Section \ref{sec:ellposexclude} that the total number of transpositions is $4g+2\ell$, and therefore the number of transpositions on $\gGamma \to \gT$ is exactly $3g$ and the number on the marked tree part of $\Gamma \to T$ is $g+2\ell$.

Knowing that any new loop is added with exactly three transpositions we use the proof in Section \ref{sec:prooftroptev} to rule out all loop fragments besides $U$ and $D$. That proof also gives that all joins happen in a row. Then all possible solutions to our problem must be of the form of those we have exhibited, and there can be no more solutions and thus we have shown Theorem \ref{thm:ellnegative}.

\section{Tropical generalized Tevelev degrees}\label{sec:higherram}
In this section we exhibit a combinatorial computation for tropical generalized Tevelev degrees. Recall the conditions $d=g+1+\ell$ and $n=g+3+2\ell$. We start by restricting $\mu_h$ to be a vector of all $1$s for $h=1,\dots,k$. In Section \ref{sec:genellpos}, we further specialize to the case $\ell=0$ and go through examples and set the groundwork for how to compute the tropical generalized Tevelev degrees. Then in Section \ref{sec:genvaryell}, we generalize constructing and the formula counting tropical generalized Tevelev degrees to all integer values of $\ell$. Finally, in Section \ref{sec:genramgentev}, we generalize $\mu_h$ to arbitrary ramification profiles for $h=1,\dots,k$. We assume $|\mu_h|\geq|\mu_j|$ for all $h\leq j$.

 \subsection{$\ell=0$ case}\label{sec:genellpos}
 We choose a point $p = (\overline{\Gamma}, \overline{T})$ in the interior of a maximal cone of the refinement of $\Mgn{g,n}^\trop\times \Mgn{0,k}^\trop$ induced by the map $\Fs\times \Ft$. The pair of tropical curves parameterized by $p$ are depicted in Figure \ref{fig:newgenimage}.

 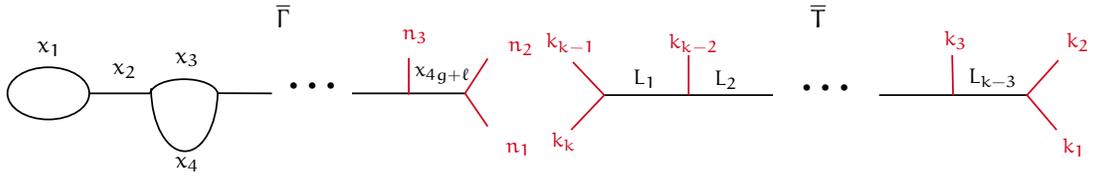
\begin{figure}[tb]
    \centering
     \resizebox{\textwidth}{!}{\tikzset{every picture/.style={line width=0.75pt}} 

\begin{tikzpicture}[x=0.75pt,y=0.75pt,yscale=-1,xscale=1]

\draw   (26,66) .. controls (26,57.72) and (36.52,51) .. (49.5,51) .. controls (62.48,51) and (73,57.72) .. (73,66) .. controls (73,74.28) and (62.48,81) .. (49.5,81) .. controls (36.52,81) and (26,74.28) .. (26,66) -- cycle ;
\draw    (73,66) -- (108,66) ;
\draw  [draw opacity=0] (108,66) .. controls (108.94,61.43) and (117.27,57.85) .. (127.41,57.85) .. controls (137.48,57.84) and (145.78,61.36) .. (146.81,65.89) -- (127.41,66.85) -- cycle ; \draw   (108,66) .. controls (108.94,61.43) and (117.27,57.85) .. (127.41,57.85) .. controls (137.48,57.84) and (145.78,61.36) .. (146.81,65.89) ;  
\draw  [draw opacity=0] (146.81,65.89) .. controls (146.36,84.42) and (138.25,99.16) .. (128.12,99.28) .. controls (117.54,99.42) and (108.77,83.59) .. (108.53,63.94) .. controls (108.53,63.65) and (108.52,63.37) .. (108.52,63.09) -- (127.67,63.7) -- cycle ; \draw   (146.81,65.89) .. controls (146.36,84.42) and (138.25,99.16) .. (128.12,99.28) .. controls (117.54,99.42) and (108.77,83.59) .. (108.53,63.94) .. controls (108.53,63.65) and (108.52,63.37) .. (108.52,63.09) ;  
\draw    (146.81,65.89) -- (178,66) ;
\draw    (224,66) -- (289,66) ;
\draw [color={rgb, 255:red, 208; green, 2; blue, 27 }  ,draw opacity=1 ]   (256.67,66.67) -- (256.67,45.67) ;
\draw [color={rgb, 255:red, 208; green, 2; blue, 27 }  ,draw opacity=1 ]   (289,66) -- (302,46) ;
\draw [color={rgb, 255:red, 208; green, 2; blue, 27 }  ,draw opacity=1 ]   (289,66) -- (302,85) ;
\draw    (369,67) -- (466,67) ;
\draw [color={rgb, 255:red, 208; green, 2; blue, 27 }  ,draw opacity=1 ]   (369,67) -- (351,48) ;
\draw [color={rgb, 255:red, 208; green, 2; blue, 27 }  ,draw opacity=1 ]   (369,67) -- (350,87) ;
\draw [color={rgb, 255:red, 208; green, 2; blue, 27 }  ,draw opacity=1 ]   (417.5,67) -- (417.5,44) ;
\draw    (527,67) -- (612,67) ;
\draw [color={rgb, 255:red, 208; green, 2; blue, 27 }  ,draw opacity=1 ]   (569.5,67) -- (569,44) ;
\draw [color={rgb, 255:red, 208; green, 2; blue, 27 }  ,draw opacity=1 ]   (612,67) -- (630,47) ;
\draw [color={rgb, 255:red, 208; green, 2; blue, 27 }  ,draw opacity=1 ]   (612,67) -- (629,87) ;

\draw (41.67,34.73) node [anchor=north west][inner sep=0.75pt]  [font=\small]  {$x_{1}$};
\draw (85,47.4) node [anchor=north west][inner sep=0.75pt]  [font=\small]  {$x_{2}$};
\draw (121,40.73) node [anchor=north west][inner sep=0.75pt]  [font=\small]  {$x_{3}$};
\draw (121,101.07) node [anchor=north west][inner sep=0.75pt]  [font=\small]  {$x_{4}$};
\draw (304.09,88.34) node [anchor=north west][inner sep=0.75pt]  [font=\footnotesize,color={rgb, 255:red, 208; green, 2; blue, 27 }  ,opacity=1 ,rotate=-358.45]  {$ \begin{array}{l}
n_{1}\\
\end{array}$};
\draw (305.63,30.16) node [anchor=north west][inner sep=0.75pt]  [font=\footnotesize,color={rgb, 255:red, 208; green, 2; blue, 27 }  ,opacity=1 ,rotate=-359.63]  {$ \begin{array}{l}
n_{2}\\
\end{array}$};
\draw (251.68,27.69) node [anchor=north west][inner sep=0.75pt]  [font=\footnotesize,color={rgb, 255:red, 208; green, 2; blue, 27 }  ,opacity=1 ,rotate=-0.17]  {$n_{3}$};
\draw (631,90.4) node [anchor=north west][inner sep=0.75pt]  [font=\footnotesize,color={rgb, 255:red, 208; green, 2; blue, 27 }  ,opacity=1 ]  {$k_{1}$};
\draw (633.33,29.73) node [anchor=north west][inner sep=0.75pt]  [font=\footnotesize,color={rgb, 255:red, 208; green, 2; blue, 27 }  ,opacity=1 ]  {$k_{2}$};
\draw (563,28.73) node [anchor=north west][inner sep=0.75pt]  [font=\footnotesize,color={rgb, 255:red, 208; green, 2; blue, 27 }  ,opacity=1 ]  {$k_{3}$};
\draw (337.28,86.79) node [anchor=north west][inner sep=0.75pt]  [font=\footnotesize,color={rgb, 255:red, 208; green, 2; blue, 27 }  ,opacity=1 ,rotate=-359.4]  {$k_{k}$};
\draw (334.08,29.85) node [anchor=north west][inner sep=0.75pt]  [font=\footnotesize,color={rgb, 255:red, 208; green, 2; blue, 27 }  ,opacity=1 ,rotate=-0.85]  {$k_{k}{}_{-1}$};
\draw (404.67,30.05) node [anchor=north west][inner sep=0.75pt]  [font=\footnotesize,color={rgb, 255:red, 208; green, 2; blue, 27 }  ,opacity=1 ,rotate=-0.08]  {$k_{k}{}_{-2}$};
\draw (384.33,50.73) node [anchor=north west][inner sep=0.75pt]  [font=\footnotesize]  {$L_{1}$};
\draw (431,51.4) node [anchor=north west][inner sep=0.75pt]  [font=\footnotesize]  {$L_{2}$};
\draw (577,50.4) node [anchor=north west][inner sep=0.75pt]  [font=\footnotesize]  {$L_{k-3}$};
\draw (185,58.4) node [anchor=north west][inner sep=0.75pt]  [font=\huge]  {$\dotsc $};
\draw (480,59.4) node [anchor=north west][inner sep=0.75pt]  [font=\huge]  {$\dotsc $};
\draw (258.67,49.07) node [anchor=north west][inner sep=0.75pt]  [font=\footnotesize]  {$x_{4g+2\ell }$};
\draw (179,13.4) node [anchor=north west][inner sep=0.75pt]    {$\overline{\Gamma }$};
\draw (486,13.4) node [anchor=north west][inner sep=0.75pt]    {$\overline{T}$};

\end{tikzpicture}}
    \caption{The graphs $\overline{\Gamma}, \overline{T}$ defining the chosen point $p$ of $\Mgn{g,n}^\trop\times\Mgn{0,k}^\trop$ where the set of $n$ marked points, $\{n_1,\dots,n_n\}$, is $\{1_1,\dots,1_{r_1},(r_1+1)_1,\dots,(r_1+1)_{r_2},\dots,(r_1+r_2+\dots+r_{k-1}+1)_1,(r_1+r_2+\dots+r_{k-1}+1)_{r_k}\}$ and the set of $k$ marked points, $\{k_1,\dots,k_k\}$, is $\{1,r_1+1,r_1+r_2+1,\dots,r_1+\dots+r_{k-1}+1\}$. To simplify notation, we drop the subscript when $r_h=1$.}
    \label{fig:newgenimage}
\end{figure}
We start with going through a few examples when $g=2$ and then construct the number of preimages of $p$, that we count with multiplicities. Finally, we exclude the possiblity of other preimages.  

    \subsubsection*{Examples when $g=2$}

When $g=2$, $d=3$, with $n=5$ marked points, we start by considering all possible $\mu_h$s. Given $5$ marked points and requiring $k\geq3$, there are $4$ options: \begin{itemize}
    \item $(1),(1),(1),(1),(1)$
    \item $(1,1),(1),(1),(1)$
    \item $(1,1),(1,1),(1)$
    \item $(1,1,1),(1),(1)$.
\end{itemize}
The first case is proven in \cite{troptev}, and will be referred to as the original case. Their paper details the construction of the covers in $((\Fs\times\Ft))^{-1}(p)$ for the original case, shown in Figure \ref{fig:l=0g=2}, each with multiplicity $1$.


Next, in order to compute $\Tev^\trop_{2,0,(1,1),(1),(1),(1)}$, we must compute the global degree of the map 
\[(\Fs\times\Ft):\Hur{2,3,(1,1),(1),(1),(1)}^\trop \to \Mgn{2,5}^\trop \times \Mgn{0,4}^\trop.\]
Consider the point $p=(\overline{\Gamma},\overline{T})\in\Mgn{2,5}^\trop\times\Mgn{0,4}^\trop$ depicted in Figure \ref{fig:newgenimage}. We start by constructing the genus part $\gGamma\to\gT$, which is constructed in the same way by starting with a degree $2$ loop and continuing by adding a $U$ or $D$ genus fragments.
We complete $\gGamma\to \gT$ to $\Gamma \to T$ by adding the marked points. Starting by placing the four marked points on the target, there are three ways to organize four marked points that allows for one long edge $L_1>>x_m$ for all $m$, where $L_1$ is the length on $\overline{T}$ and $x_m$ are lengths on $\overline{\Gamma}$. There three ways are a single tree with marked points $1$ and $3$ together and points $4$ and $5$ together; marked points $1$ and $3$ placed alone and a tree containing $4$ and $5$ on the left; and marked points $4$ and $5$ placed alone and a tree containing $1$ and $3$ on the right. Figure \ref{fig:n=4treefail} shows the two possible ways to generalize the placement of marked points on a single tree to the case with more than one preimage. Due to the $1$ having two preimages, its preimage must be marked on both copies of the cover of the tree on the degree two edge. This causes the stabilization to not be to $\overline{\Gamma}$, because of the incorrect shape (shown on left of Figure) or the incorrect order of marked points (shown on right). The two remaining possibilities of marked fragments are placed on the covers $\gGamma\to\gT$ as shown in Figure 
\ref{fig:l=0g=2(1,1)}. Comparing to Figure \ref{fig:l=0g=2}, there is one less cover, corresponding to the inability to place the marked points on a single tree. The multiplicity of each cover can be computed in the same way as previous cases in this paper and can be shown to be $1$ for each cover. Therefore, $\Tev^\trop_{2,0,(1,1),(1),(1),(1)}=3$.
\begin{figure}[tb]
    \centering
     \resizebox{.7\textwidth}{!}{\input{Figures/n=4treefail}}
    \caption{Two covers \textbf{not} in $((\Fs\times\Ft))^{-1}(p)$ for $\ell=0$, $g=2$, when attempting to use a marked point fragment of a single tree.}
    \label{fig:n=4treefail}
\end{figure}

        \begin{figure}[tb]
        \centering
     \resizebox{.6\textwidth}{!}{\input{Figures/l=0g=2_1,1_}}
        \caption{The covers in $((\Fs\times\Ft))^{-1}(p)$ for $\ell=0$, $g=2$, and $\mu_h$'s: $(1,1),(1),(1),(1).$}
        \label{fig:l=0g=2(1,1)}
    \end{figure}

Continuing to increase the number of preimages of marked points that are marked, we look at $\Tev^\trop_{2,0,(1,1),(1,1),(1)}$. We must compute the global degree of the map 
\[(\Fs\times\Ft):\Hur{2,3,(1,1),(1,1),(1)}^\trop \to \Mgn{2,5}^\trop \times \Mgn{0,3}^\trop.\]
Consider the point $p=(\overline{\Gamma},\overline{T})\in\Mgn{2,5}^\trop\times\Mgn{0,3}^\trop$. The genus part $\gGamma\to\gT$ is again constructed in the same way, leading to two possibilities. The two possibilities tells us that we either have an active edge of degree $3$ leaving the genus with a single cut before ending with a simple ramification point at the end or there is an active edge of degree $1$ with a single join before ending with a simple ramification point. There are four possible marked point fragments for three markings: a single tree with all three marks; a tree with marked points $2$ and $3$ with marked point $1$ to the right; a tree with marked points $1$ and $2$ with marked point $3$ to the left; and all three marked points lying alone. Given a tree with all markings on it, due to the number of preimages of $1$ and $2$, there will be five points marked covering the tree. In order to get the correct stabilization, all but two of these markings need to be on their own branch (cut, join or active edge). There are only 2 branches, either active edge and a cut or the active edge with a join, so this is not possible. The same argument of number of preimages compared to the $2$ branches rules out a single marking with a tree containing two markings.
There is a single marked fragment option remaining for three marked points, all individually placed on target. This marked fragment can be placed on either genus part option, as shown in Figure \ref{fig:l=0g=2(1,1)(1,1)}. Comparing to Figure \ref{fig:l=0g=2}, there are two less covers. The multiplicity of each cover can be shown to be $1$ and we conclude that $\Tev^\trop_{2,0,(1,1),(1,1),(1)}=2$.

        \begin{figure}[tb]
        \centering
     \resizebox{.3\textwidth}{!}{\input{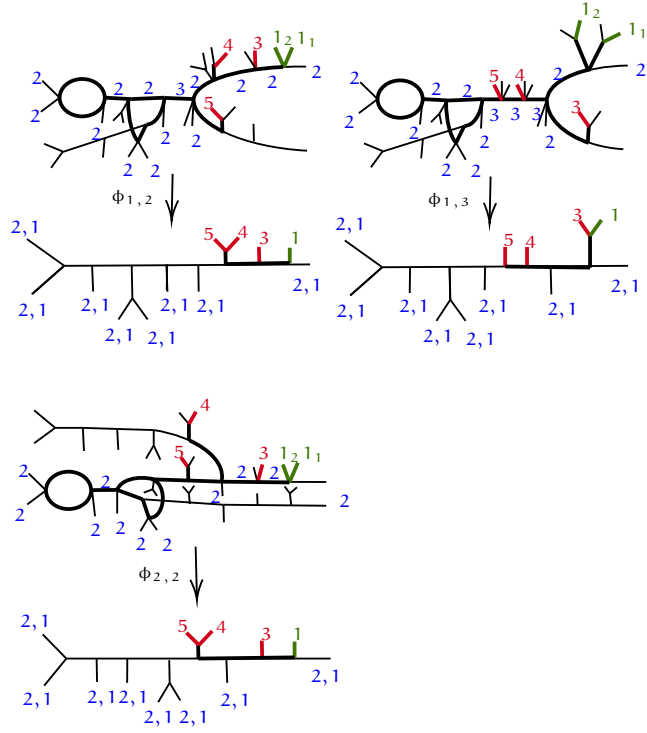}}
        \caption{The covers in $((\Fs\times\Ft))^{-1}(p)$ for $\ell=0$, $g=2$, and $\mu_h$'s: $(1,1),(1,1),(1).$ }
        \label{fig:l=0g=2(1,1)(1,1)}
    \end{figure}

Finally, we look at $\Tev^\trop_{2,0,(1,1,1),(1),(1)}$. We must compute the global degree of the map 
\[(\Fs\times\Ft):\Hur{2,3,(1,1,1),(1),(1)}^\trop \to \Mgn{2,5}^\trop \times \Mgn{0,3}^\trop.\]
Consider the point $p=(\overline{\Gamma},\overline{T})\in\Mgn{2,5}^\trop\times\Mgn{0,3}^\trop$. Similar to the previous case, there are three marked points on the target, giving the same fragment option. Due to the first marked point having three preimages marked, this fragment can only be placed on the degree $3$ active edge cover, as shown in Figure \ref{fig:l=0g=2(1,1,1)}. Comparing to Figure \ref{fig:l=0g=2}, there are three less covers, as if the first two columns were removed. The multiplicity of this cover is $1$, so $\Tev^\trop_{2,0,(1,1,1),(1),(1)}=1$.

        \begin{figure}[tb]
        \centering
     \resizebox{.3\textwidth}{!}{\input{Figures/l=0g=2_1,1,1_}}
        \caption{The cover in $((\Fs\times\Ft))^{-1}(p)$ for $\ell=0$, $g=2$, and $\mu_h$'s: $(1,1,1),(1),(1).$}
        \label{fig:l=0g=2(1,1,1)}
    \end{figure}

    \subsubsection*{Construction}
In this section we make explicit and generalize the constructions from the examples in the previous section and construct the preimages of $p$ for any genus $g$, $\ell=0$ and any set $\mu_1,\dots,\mu_k$ where $\mu_h$ is a vector of all 1s. Note that changing $\mu_h$, does not impact the degree or the genus of the cover, therefore the genus part $\gGamma\to\gT$ is constructed in the same way as explained in Section \ref{sec:prooftroptev}. This section focuses on the genus zero section containing the marked points.

\begin{figure}[tb]
    \centering
     \resizebox{.8\textwidth}{!}{\input{Figures/l=0g=2fragments}}
    \caption{The possible marked fragments when $\ell=0, g=2$ for different $\mu_1,\dots,\mu_k$. Only the target fragments are shown in this figure. Although grouping
reduces the number of displayed target markings from \(n\) to \(k\), all
\(n\) marked source ends remain. Observe that the second row of marked fragments are obtained by combining the ends labeled $1$ and $2$ in the top row; the third row is obtained by combining the ends labeled $1$ and $2$ and moving either $3$ or $4$ to combine with $4$ or respectively, $3$, in the middle; and the bottom row is obtained by combining the ends labeled $1,2$ and $3$. The marked points that come from action $C$ are colored light blue while the marked points coming from action $M$ are colored orange.}
    \label{fig:l=0g=2fragments}
\end{figure}

Throughout this subsection, each profile is of the form
\(\mu_h=(1^{r_h})\), so \(r_h=|\mu_h|\) and
\[
n=\sum_{h=1}^k r_h.
\]
We begin with a marked fragment for the datum
\(((1),\ldots,(1))\), which has \(n\) distinct marked ends on the target.
For each \(h\), the \(r_h\) original target markings corresponding to
\(\mu_h\) are replaced by the single target marking \(q_h\). The
corresponding marked source ends are not forgotten: they remain as the
\(r_h\) distinct marked ends
\[
p_{h,1},\ldots,p_{h,r_h},
\]
all mapping to \(q_h\). Thus, the following operations reduce the number
of target markings from \(n\) to \(k\), but preserve the \(n\) marked
source ends. On the stabilized target tree, replacing a group of target
markings has the same effect as retaining one marking from the group, forgetting the others, and stabilizing. There are two possible operations to obtain new marked fragments from the original marked fragments introduced in \cite{troptev}, demonstrated with the marked fragments from the previous example in Figure \ref{fig:l=0g=2fragments}. 

\begin{enumerate}
    \item[\textbf{Action C}] (\emph{Combining marks on the same tree}).  
    Given $r_i$ marked points lying on a single tree of the fragment $F_j$, these may be combined provided that the tree contains at least $r_i$ original marked points. This action is shown in Figure \ref{fig:l=0g=2fragments} in light blue, for example from the top line to the second line, marked points $1$ and $2$ in the two rightmost columns combine to the light blue marked points $1$ in the second row.
    
    \item[\textbf{Action M}] (\emph{Combining marks from different trees}).  
    Here $r_i$ marked points from distinct trees of $F_j$ are merged by moving marks from one or both outer trees and combining them with the central tree.  
    To move $a$ points from a given tree, that tree must initially contain at least $a+1$ marks so that, after the move, all three trees remain disjoint. This is shown in Figure \ref{fig:l=0g=2fragments} in orange, for example going from the second to top row far right fragment to third from the top row fragment, the $3 $ is combined with the $4$ to make the orange $3$.
\end{enumerate}

\begin{lemma}
Let $F_j$ be a marked fragment with $2 \le j \le n-2$, and let $\mu_1,\dots,\mu_k$ be ordered partitions of the marked points.  
The fragment $F_j$ transforms into a valid fragment for a generalized Tevelev degree if and only if its markings may be produced by iteratively applying Actions~C and~M according to the following rules:
\begin{enumerate}
    \item for each index $h$, if $\sum_{t=1}^h |\mu_t| < j$, then $|\mu_h|$ marks combine on a single tree via Action~C;
    \item if $\sum_{t=1}^h |\mu_t| = j$, then a fragment of type $F_h$ is produced by preforming Action~C for the first $h$ marked points, with Action~M being used for marked point $h+1$ and Action~C used to combine marked points for the remaining marked points in the left tree;
    \item if $\sum_{t=1}^h |\mu_t| = j + m$ with $m$ nonzero, then $\left( |\mu_h| - m\right)$ marks move from the right tree and $(m-1)$ marks from the left tree into the center via Action~M, producing a fragment of type $F_{h-1}$.
\end{enumerate}
Furthermore, when $h = k-1$, the final part $\mu_k$ must satisfy $|\mu_k| = n - j - 1$; otherwise no valid fragment can be formed.
\end{lemma}

\begin{proof}
The procedure follows from the combinatorial structure of the active edge construction in \cite{troptev}.  
Action~C corresponds to combining marks that lie on the same tree of the fragment, while Action~M corresponds to moving marks from the left or right tree to the central tree so that all three trees remain disjoint.  
The case distinctions arise from the necessity that the number of combined marks matches the number of cuts and joins covering the left and right components of the fragment.

If the partial sum $\sum_{t=1}^h |\mu_t|$ is strictly less than $j$, the markings remain on the same tree and Action~C applies and one moves on to $\mu_{h+1}$.  
If the sum equals $j$, Action~C applies to a fragment of type $F_h$.  
When the sum exceeds $j$, marks are redistributed across the trees via Action~M to match the combinatorics of the covering curve, resulting in a fragment of type $F_{h-1}$.

The final condition on $|\mu_k|$ ensures that the number of marks on the rightmost tree matches the number of cuts covering that branch; otherwise the fragment cannot correspond to any cover in the generalized Tevelev degree.  
Thus all valid transformations of $F_j$ arise precisely from the stated recursive procedure.
\end{proof}

\vspace{0.5em}
\noindent\textbf{Example: The case of the fragment $F_0$.}\\
Starting with the fragment $F_0$ and a tuple of partitions $\mu_1,\dots,\mu_k$, we first apply Action~C to each $\mu_h$.  
This produces a new $F_0$ fragment whose central tree carries $k$ marked points.

We claim that this new fragment cannot appear on a cover mapping to the chosen point $p$.  
To see this, consider the simplest new instance: take $\mu_1=(1,1)$ and (1) for $\mu_2,\dots,\mu_k$.  
Following the procedure of \cite{troptev}, the base cover is obtained by a sequence of cuts along the active edge.  
We mark the last edge of the active path and attach the resulting $F_0$ fragment at that position.

On the covering curve, the marks are placed in descending order on the connected components of the inverse image of $F_0$, so that mark $k$ stabilizes at the vertex created by the first cut, and so forth.

The last connected component in the preimage attaches to an edge of degree two, and therefore consists of two copies of the fragment $F_2$.  
There are then two possibilities:

\begin{itemize}
    \item Place marks $\{3,1_2\}$ on one copy and $\{2,1_1\}$ on the other.  
    This fails to stabilize to a caterpillar curve of the correct genus.
    
    \item Place $\{1_2\}$ on one copy and $\{1_1,2,3\}$ on the other.  
    This stabilizes to the correct shape, but the order of the marked points is incorrect.
\end{itemize}

Thus, the fragment $F_0$ cannot appear in generalized Tevelev degrees.

\vspace{0.5em}
\noindent\textbf{Iterative procedure for general fragments $F_j$.}\\
Let $F_j$ be a fragment with $2 \le j \le n-2$, and let $\mu_1,\dots,\mu_k$ be ordered partitions.  
The following iterative procedure produces valid transformed fragments as demonstrated in Figure \ref{fig:l=0g=2fragments}:

\begin{enumerate}
\item[$\boldsymbol{\mu_1:}$]
\begin{itemize}
    \item If $|\mu_1| < j$, combine $|\mu_1|$ points via Action~C and continue.
    \item If $|\mu_1| = j$, 
    \begin{itemize}
        \item if $|\mu_2| = 1$, form fragment $F_1$ and apply Action~C to $\mu_3,\dots,\mu_k$;
        \item if $|\mu_2| > 1$, apply Action~M to $\mu_2$ and Action~C to $\mu_3,\dots,\mu_k$ to form $F_1$.
    \end{itemize}
    \item If $|\mu_1| > j$, no valid fragment is possible.
\end{itemize}

\item[$\boldsymbol{\mu_2:}$]
\begin{itemize}
    \item If $|\mu_1|+|\mu_2| < j$, combine $|\mu_2|$ points via Action~C and continue.
    \item If $|\mu_1|+|\mu_2| = j$, form $F_2$ and apply Action~C or Action~M as above for $\mu_3,\dots,\mu_k$.
    \item If $|\mu_1|+|\mu_2| = j + m$, move $(|\mu_2|-m)$ points from the right tree and $(m-1)$ from the left tree to the center to form $F_1$, then apply Action~C to remaining marks.
\end{itemize}

\item[\vdots]

\item[$\boldsymbol{\mu_h:}$]
\begin{itemize}
    \item If $\sum_{h=1}^i |\mu_h| < j$, combine $|\mu_h|$ points via Action~C and continue.
    \item If $\sum_{h=1}^i |\mu_h| = j$, form $F_i$ and apply Action~C or Action~M to $\mu_{i+1},\dots,\mu_k$.
    \item If $\sum_{h=1}^i |\mu_h| = j + m$, move $(\sum_{h=1}^i |\mu_h| - m)$ points from the right tree and $(m-1)$ points from the left tree to the center to form $F_{i-1}$, then apply Action~C to remaining marks.
    \item If $i=k-1$, $|\mu_k|$ must equal $n-j-1$, otherwise no valid fragment is possible.
\end{itemize}
\end{enumerate}

This algorithm produces fragments of type $F_j$ with $1\leq j\leq k-2$. The marked fragment $F_1$ is new for this generalization of the computation. It is defined to be $F_j$ with $j=1$ as in Figure \ref{fig:markfrag}, which is the first two marked points are single marked points and there is a tree on the left containing the rest of the marked points. In the original case, when all marked points have $\mu_h=(1)$, the fragment $F_1$ was ruled out due to infinitely many covers using $F_1$ mapping to the same point $p$. Due to the fact that we require $\lvert\mu_h\rvert\geq \lvert \mu_j\rvert$ when $i<j$, $\lvert\mu_1\rvert>1$ when differing from previous cases already proven. Since at least two preimages of the first marked point on the target will be marked in the source, the first two stabilize together with the rest in a caterpillar tree to the genus part of the image. Therefore, the length between the first two marked points on the target is not lost when stabilizing, solving the problem of infinite preimages, $\Gamma\to T$. These ideas are demonstrated in Figure \ref{fig:F1}. 

\begin{figure}[tb]
    \centering
     \resizebox{\textwidth}{!}{\input{Figures/F1}}
    \caption{On the left, there is an example of the rightmost end of a $\Gamma\to T$ when all marked points have one preimage marked, along with the rightmost ends of the pair of curves, $p$, this $\Gamma\to T$ maps to. Note that the length $x$, in blue, is lost when mapping to $p$, therefore changing the value of $x$ does not change the image. On the right, there is a similar example with the difference being that the first marked point has $2$ preimages marked. Observe that $x$ now appears in the image, fixing the length $x$ that maps to $p$.}
    \label{fig:F1}
\end{figure}

The fragments $F_j$ are placed on the genus zero sections of the covers $\Gamma \to T$ following the same process as how they are placed in previous cases. Since the fragments come from fragments of the original tropical Tevelev degree case and there are the same number of marked points on $\Gamma$, all marked points on $\Gamma$ are on the same cut or join they were on originally. When action M is preformed to obtain the new fragment, the cuts and joins corresponding to the moved marked points move with them to allow the cut or join to remain covering the point. Figure \ref{fig:cutsandjoinsmovingex} demonstrates a cut and a join moving.

\begin{figure}[tb]
    \centering
     \resizebox{.7\textwidth}{!}{\input{Figures/cutsandjoinsmovingex}}
    \caption{The left half of this figure shows the genus zero part of one $\Gamma\to T$ from the case when all $\mu_h=(1)$. The unique marked fragment $F_4$ that can be placed on it is listed above. The right side of the figure shows how the marked fragment and the cover change when the $\mu_h$ become $(1,1,1),(1,1,1),(1),(1)$. Observe that $1,2,$ and $3$, all in green, become $1_1,1_2$ and $1_3$. Similarly, $4,5$ and $6$, colored in orange, become $4_1,4_2$ and $4_3$ with the join, respectively cut, move with the marked point 6, respectively 4. The marked fragment on the right is $F_1$.}
    \label{fig:cutsandjoinsmovingex}
\end{figure}

The examples in the previous section illustrate that distinct constructions of 
$\Gamma \to T$ may nevertheless produce equivalent covers. Such coincidences 
arise exactly when two marked fragments become identical and the associated 
sequence of cuts and joins along the active edge agrees. In the join–free case 
($i=0$), this forces the cover to be unique. For $i>0$, the situation is more 
subtle, and the following lemma describes precisely when equivalence occurs. 

\begin{lemma}\label{coversbecomesame}
Let $F_j$ and $F_{j+1}$ be consecutive marked fragments with $i>0$ joins along 
the active edge. Then $F_j$ and $F_{j+1}$ produce equivalent covers 
$\Gamma \to T$ precisely when the numbers of cuts moved from the left and 
right trees satisfy the following condition: for some integers $a,b \ge 1$,
\begin{itemize}
    \item from $F_j$: $b$ cuts and all $i$ joins move from the left tree to the center,
    and $a-1$ cuts move from the right tree, and
    \item from $F_{j+1}$: $b-1$ cuts and all $i$ joins move from the left tree
    to the center, and $a$ cuts move from the right tree.
\end{itemize}
In this case, both constructions yield a cover in which
\[
\begin{aligned}
\text{(left tree)} &\qquad d - j - i - b \ \text{cuts},\\
\text{(middle mark)} &\qquad a + b - 1 \ \text{cuts and } i \text{ joins},\\
\text{(right tree)} &\qquad j - a - 1 \ \text{cuts}.
\end{aligned}
\]
Varying $a$ and $b$ produces exactly $a+b$ equivalent marked fragments.
\end{lemma}

\begin{proof}
The examples in the previous section show that two covers $\Gamma \to T$ can 
be equivalent precisely when (i) at least two marked fragments coincide and 
(ii) the sequence of cuts and joins along the active edge is identical. Recall that $i$ is the number of $D$ fragments, which since $\ell=0$, is equal to the number of joins. When 
$i=0$, the absence of joins forces the construction to be unique, so the 
resulting covers are always identical. We therefore assume $i>0$.

Marked fragments $F_j$ and $F_{j+1}$ differ by a single marked point moving 
from the right tree of $F_{j+1}$ to the center. Thus the two fragments can 
coincide only if the resulting distribution of cuts and joins along the active 
path is adjusted appropriately.

Suppose first that fewer than $i$ marked points move from the left tree of 
$F_j$ to the center. Then the joins split: the covering of the left tree 
acquires $d - j - i$ cuts and fewer than $i$ joins. Such a configuration 
cannot arise from $F_{j+1}$, since $F_{j+1}$ has one fewer cut covering the 
leftmost tree. To match the two constructions, a marked point must be moved 
from the right tree of $F_{j+1}$ to the left tree, shifting the corresponding 
cut. Thus all $i$ joins must move to the middle marked point, ensuring that at 
least one cut also moves so that the number of cuts covering the left tree is 
the same in both constructions.

Now consider the general situation. Starting from $F_j$, suppose that
\begin{itemize}
    \item $b$ cuts and all $i$ joins move from the left tree to the center, and
    \item $a-1$ cuts move from the right tree.
\end{itemize}
Starting instead from $F_{j+1}$, suppose that
\begin{itemize}
    \item $b-1$ cuts and all $i$ joins move from the left tree to the center, and
    \item $a$ cuts move from the right tree.
\end{itemize}

In both cases, a direct count shows that
\[
\begin{aligned}
\text{left tree:} &\quad d - j - i - b \ \text{cuts},\\
\text{middle mark:} &\quad a + b - 1 \ \text{cuts and } i \text{ joins},\\
\text{right tree:} &\quad j - a - 1 \ \text{cuts}.
\end{aligned}
\]
Because the resulting data agree, the two constructions yield the same cover 
$\Gamma \to T$. Conversely, any equivalence between covers arising from 
$F_j$ and $F_{j+1}$ must arise from such a redistribution of cuts and joins, 
since these are the only moves that preserve the total number of cuts on each 
subtree while matching the join structure.

Allowing $a$ and $b$ to vary produces exactly $a+b$ distinct ways in which 
$F_j$ and $F_{j+1}$ may yield the same cover.
\end{proof}

We now organize the covers constructed in a way that allows us to count them. Following the technique of \cite{troptev}, we put the covers $\Gamma \to T$ inside (but not filling) a rectangular array, where rows correspond to solutions with the same genus part, and the columns to covers with the same marked fragment type. The rows are ordered so that the degree of the active edge is non-increasing, and the columns are ordered by the index $j$ of the marked fragments $F_j$. We consider all $n-2$ columns in order to compare to grid in the original case when $k=n$. 

The genus part $\gGamma\to \gT$ of each cover corresponds to a word of length $g-1$ in the letters $U$ and $D$ with the condition that the degree of the active edge remains positive. Recall Lemma 4.3 from \cite{troptev}:
\begin{lemma} \label{lemma:deg active}
For $d = g+1 \geq 2$, $0\leq i\leq \lfloor\frac{d-1}{2}\rfloor$, denote by $A_{d, \geq d-2i}$ denote the number of paths described above with endpoint of coordinate $(d, y)$, $y\geq d-2i$. 
\begin{equation}
  A_{d,\geq d-2i}
  = {{g}\choose{i}}.
\end{equation}
\end{lemma}
The original case when all $\mu_h$ are equal to $(1)$ is proven to be $2^g$. We now examine the other cases by comparing the grids. Recall that a $F_j$ can be placed on covers when $i+1\leq j \leq d-i$ is satisfied where $i$ is the number of $D$ fragments used. Thus the height of the column corresponding to $F_j$ is ${{d-1}\choose{\min(d-j,j-1)}}$. There are three ways that the grid changes:
\begin{enumerate}
    \item if $|\mu_1|>j$: a column is lost for each $j$ this holds for due to the lose of the corresponding marked fragment. 
    \begin{enumerate}
        \item When $|\mu_1|=2$, lose $F_0$ column, which has height ${{g}\choose{0}}$.
        \item When $|\mu_1|>2$, lose $F_0$ column and $F_2,\dots,F_{|\mu_1|-1}$ columns where the height of column $F_j$ is ${{g}\choose{j-1}}$.
    \end{enumerate}
    \item if $|\mu_k|>n-j-1$: a column is lost for each $j$ this holds for due to the lose of the corresponding marked fragment. Each column $F_j$ has height ${{g}\choose{g-j+1}}={{g}\choose{j-1}}$.
    \item if $2$ or more fragments become the same with equivalent covers. When this occurs, all but one copy of the covers with given active edge degree are removed. There are $a+b=|\mu_h|-i$ covers that are equivalent, each have ${{g}\choose{i}}-{{g}\choose{i-1}}$ unique ways to form the genus part that produces the given active degree.
\end{enumerate}
Combining all of these we have 
 \begin{equation}
    \Tev^\trop_{g,0,\mu_1,\dots,\mu_k}\geq2^g -\sum_{i=0}^{\lvert\mu_1\rvert-2}{g\choose i}-\sum_{i=0}^{\lvert\mu_k\rvert-2}{g\choose i}-\sum_{h=2}^{k-1}\sum_{i=0}^{\lvert\mu_h\rvert-2}\bigg(\lvert\mu_h\rvert-i-1\bigg)\bigg({g\choose i}-{g \choose i-1}\bigg).
    \end{equation}

 \subsubsection*{Exclusion}\label{sec:excludegentev}
In this section, we exclude any further cover $\Gamma \to T$ from mapping to  the chosen point $p= (\overline{\Gamma},\overline{T})\in \Mgn{g,n}^\trop\times \Mgn{0,k}^\trop$. The proof from \cite{troptev} rules out any other way to form a genus part of the cover with independent cycle lengths and any other marked fragment options. What remains is to slightly alter their proof to show that the joins occur in two groups.

Recall that every time a cut occurs, the degree of the active edge decreases by $1$, and conversely, every time a join occurs, the degree of the active edge increases by $1$. Since joins come from the left and cuts go to the right, in order for both to be covering the same tree of marked points the cut must be to the left of the join. There are $3$ trees in the target with markings on them, as this is the format of the marked fragment $F_j$. The rightmost tree must be covered by all cuts because any join, occurring to the right of the cuts, would leave an active path of degree greater than two. The other two trees can be covered by some number of cuts followed by some number of joins. All joins covering the leftmost tree must occur before any cuts covering the middle tree (which is a single mark) in order to stabilize to $p$. Therefore, joins occur in up to two groups. all possible solutions to our problem must be of the form of those we have exhibited, and there can be no more solutions.
Thus we have concluded the proof of Theorem \ref{thm:higherram} when $\ell=0$ and $\mu_h$ are vectors of all $1$s.
    
 \subsection{Varying $\ell$} \label{sec:genvaryell}
 
 In this section we generalize the results from the previous section to positive and negative integer values of $\ell$. Recall that tropical generalized Tevelev degrees are defined for the conditions $d=g+1+\ell$ and $n=g+3+2\ell$. We start this section by examining the impact that increasing $\ell$ has on the proof in the previous section. We then do the same looking at negative $\ell$ values to conclude the proof of Theorem \ref{thm:higherram} when $\mu_h$ are vectors of all 1s.
 
    \subsubsection*{Positive $\ell$}
    \vspace{0.5em}
    \noindent\textbf{Examples.} We start by examining examples when $\ell=1$ and $g=1$ with different ramification profiles $\mu_h$. Details are left out of these examples and we focus on observing differences from the $\ell=0$ case in the previous section. 

First, looking at $\mu_1=(1,1)$ and the rest of $\mu_h=(1)$, the preimages are shown in Figure \ref{fig:ell=1,g=1,(1,1)}. There are $2$ preimages, each with multiplicity equal to $1$, therefore $\Tev_{1,1,(1,1),(1),(1),(1),(1)}^\trop=2$, which agrees with the original case shown in Figure \ref{fig:l=1g=1}. A cover is not lost here because the original case did not have the marked fragment $F_0$. Observe in Figures \ref{fig:ell=1,g=1,(1,1)(1,1)} and \ref{fig:ell=1,g=1,(1,1)(1,1)(1,1)}, add more $\mu_h$s that are $(1,1)$ continues to not change the tropical generalized Tevelev degree. 

    \begin{figure}[tb]
        \centering
     \resizebox{.8\textwidth}{!}{\input{Figures/ell=1,g=1,_1,1_}}
        \caption{The covers in $((\Fs\times\Ft))^{-1}(p)$ for $\ell=1$, $g=1$, and $\mu_h$'s: $(1,1),(1),(1),(1),(1).$}
        \label{fig:ell=1,g=1,(1,1)}
    \end{figure}
    \begin{figure}[tb]
        \centering
     \resizebox{.8\textwidth}{!}{\input{Figures/ell=1,g=1,_1,1__1,1_}}
        \caption{The covers in $((\Fs\times\Ft))^{-1}(p)$ for $\ell=1$, $g=1$, and $\mu_h$'s: $(1,1),(1,1),(1),(1).$}
        \label{fig:ell=1,g=1,(1,1)(1,1)}
    \end{figure}
        \begin{figure}[tb]
        \centering
     \resizebox{.8\textwidth}{!}{\input{Figures/ell=1,g=1,_1,1__1,1__1,1_}}
        \caption{The covers in $((\Fs\times\Ft))^{-1}(p)$ for $\ell=1$, $g=1$, and $\mu_h$'s: $(1,1),(1,1),(1,1).$}
        \label{fig:ell=1,g=1,(1,1)(1,1)(1,1)}
    \end{figure}

Figure \ref{fig:ell=1,g=1,(1,1,1)} shows the first case of positive $\ell$ where the tropical generalized Tevelev degree differs. Since $\mu_1=(1,1,1)$, one of the marked fragments can not be transformed to a new marked fragment, leading to one preimage with multiplicity 1 and therefore $\Tev_{1,1,(1,1,1),(1),(1),(1)}^\trop=1$. The next example, Figure \ref{fig:ell=1,g=1,(1,1,1)(1,1)}, shows again that adding a $\mu_h=(1,1)$ does not change the number tropical generalized Tevelev degree. 
    
    \begin{figure}[tb]
        \centering
     \resizebox{.4\textwidth}{!}{\tikzset{every picture/.style={line width=0.75pt}} 

\begin{tikzpicture}[x=0.75pt,y=0.75pt,yscale=-1,xscale=1]

\draw    (350.67,101.33) -- (323.67,79.33) ;
\draw    (350.67,101.33) -- (329.67,127.33) ;
\draw  [line width=0.75]  (350.67,101.33) .. controls (350.67,90.29) and (366.34,81.33) .. (385.67,81.33) .. controls (405,81.33) and (420.67,90.29) .. (420.67,101.33) .. controls (420.67,112.38) and (405,121.33) .. (385.67,121.33) .. controls (366.34,121.33) and (350.67,112.38) .. (350.67,101.33) -- cycle ;
\draw    (420.67,101.33) -- (492.67,102.33) ;
\draw    (420.67,101.33) -- (420.67,123.33) ;
\draw    (459.74,101.37) -- (459.74,123.37) ;
\draw  [draw opacity=0] (545.8,120) .. controls (545.7,120) and (545.6,120) .. (545.5,120) .. controls (516.51,120) and (493,112.16) .. (493,102.5) .. controls (493,92.84) and (516.51,85) .. (545.5,85) .. controls (545.6,85) and (545.69,85) .. (545.79,85) -- (545.5,102.5) -- cycle ; \draw   (545.8,120) .. controls (545.7,120) and (545.6,120) .. (545.5,120) .. controls (516.51,120) and (493,112.16) .. (493,102.5) .. controls (493,92.84) and (516.51,85) .. (545.5,85) .. controls (545.6,85) and (545.69,85) .. (545.79,85) ;  
\draw    (350.67,41.33) -- (323.67,19.33) ;
\draw    (350.67,41.33) -- (329.67,67.33) ;
\draw    (350.67,41.33) -- (418.67,52.33) ;
\draw    (418.67,52.33) -- (419.67,73.33) ;
\draw  [draw opacity=0] (418.67,52.33) .. controls (440.92,52.77) and (458.75,74.4) .. (459.74,101.37) -- (417.69,104.11) -- cycle ; \draw   (418.67,52.33) .. controls (440.92,52.77) and (458.75,74.4) .. (459.74,101.37) ;  
\draw    (492.67,102.33) -- (492.67,119.33) ;
\draw [color={rgb, 255:red, 65; green, 117; blue, 5 }  ,draw opacity=1 ]   (525,106) -- (525,118) ;
\draw [color={rgb, 255:red, 65; green, 117; blue, 5 }  ,draw opacity=1 ]   (532.55,67.37) -- (528.32,86.28) ;
\draw [color={rgb, 255:red, 65; green, 117; blue, 5 }  ,draw opacity=1 ]   (520.27,69.49) -- (527.61,85.85) ;
\draw    (434.05,90.07) -- (439.06,100.98) ;
\draw    (434.05,90.07) -- (435.67,81.62) ;
\draw [color={rgb, 255:red, 208; green, 2; blue, 27 }  ,draw opacity=1 ]   (427.49,85.38) -- (434.05,90.07) ;
\draw    (444.57,89.97) -- (439.45,100.82) ;
\draw [color={rgb, 255:red, 0; green, 0; blue, 0 }  ,draw opacity=1 ]   (444.57,89.97) -- (452.08,85.78) ;
\draw    (443.95,81.93) -- (444.57,89.97) ;
\draw    (474.1,102) -- (471.67,93.19) ;
\draw [color={rgb, 255:red, 208; green, 2; blue, 27 }  ,draw opacity=1 ]   (474.1,102) -- (478,92.67) ;
\draw    (474.1,102) -- (474.67,91.67) ;
\draw    (438.33,47.33) -- (438.33,59.33) ;
\draw [color={rgb, 255:red, 208; green, 2; blue, 27 }  ,draw opacity=1 ]   (438.33,47.33) -- (443.33,40.33) ;
\draw    (434.33,40.33) -- (438.33,47.33) ;
\draw    (352.67,193.33) -- (325.67,171.33) ;
\draw    (352.67,193.33) -- (331.67,219.33) ;
\draw    (352.67,193.33) -- (542.67,193) ;
\draw    (422,193) -- (422.67,215) ;
\draw    (460,192.33) -- (460.67,214.33) ;
\draw    (498,193) -- (498.67,215) ;
\draw    (438.79,173.63) -- (438.79,193) ;
\draw [color={rgb, 255:red, 208; green, 2; blue, 27 }  ,draw opacity=1 ]   (438.79,173.63) -- (445.6,162.33) ;
\draw [color={rgb, 255:red, 208; green, 2; blue, 27 }  ,draw opacity=1 ]   (433.33,162.33) -- (438.79,173.63) ;
\draw [color={rgb, 255:red, 208; green, 2; blue, 27 }  ,draw opacity=1 ]   (481.33,193) -- (482,181) ;
\draw [color={rgb, 255:red, 65; green, 117; blue, 5 }  ,draw opacity=1 ]   (523.33,179.33) -- (523.45,193) ;
\draw    (462.67,135.67) -- (462.67,161.67) ;
\draw [shift={(462.67,163.67)}, rotate = 270] [color={rgb, 255:red, 0; green, 0; blue, 0 }  ][line width=0.75]    (10.93,-3.29) .. controls (6.95,-1.4) and (3.31,-0.3) .. (0,0) .. controls (3.31,0.3) and (6.95,1.4) .. (10.93,3.29)   ;

\draw (313.67,75.73) node [anchor=north west][inner sep=0.75pt]  [font=\scriptsize,color={rgb, 255:red, 1; green, 1; blue, 240 }  ,opacity=1 ]  {$2$};
\draw (316.67,129.73) node [anchor=north west][inner sep=0.75pt]  [font=\scriptsize,color={rgb, 255:red, 1; green, 1; blue, 240 }  ,opacity=1 ]  {$2$};
\draw (426.67,103.4) node [anchor=north west][inner sep=0.75pt]  [font=\scriptsize,color={rgb, 255:red, 1; green, 1; blue, 240 }  ,opacity=1 ]  {$2$};
\draw (478.57,170.79) node [anchor=north west][inner sep=0.75pt]  [font=\tiny,color={rgb, 255:red, 208; green, 2; blue, 27 }  ,opacity=1 ]  {$4$};
\draw (447.9,153.79) node [anchor=north west][inner sep=0.75pt]  [font=\tiny,color={rgb, 255:red, 208; green, 2; blue, 27 }  ,opacity=1 ]  {$5$};
\draw (425.9,153.46) node [anchor=north west][inner sep=0.75pt]  [font=\tiny,color={rgb, 255:red, 208; green, 2; blue, 27 }  ,opacity=1 ]  {$6$};
\draw (520.57,169.79) node [anchor=north west][inner sep=0.75pt]  [font=\tiny,color={rgb, 255:red, 65; green, 117; blue, 5 }  ,opacity=1 ]  {$1$};
\draw (446,103.4) node [anchor=north west][inner sep=0.75pt]  [font=\scriptsize,color={rgb, 255:red, 1; green, 1; blue, 240 }  ,opacity=1 ]  {$2$};
\draw (461.74,104.77) node [anchor=north west][inner sep=0.75pt]  [font=\scriptsize,color={rgb, 255:red, 1; green, 1; blue, 240 }  ,opacity=1 ]  {$3$};
\draw (548.67,80.07) node [anchor=north west][inner sep=0.75pt]  [font=\scriptsize,color={rgb, 255:red, 1; green, 1; blue, 240 }  ,opacity=1 ]  {$2$};
\draw (481.74,104.77) node [anchor=north west][inner sep=0.75pt]  [font=\scriptsize,color={rgb, 255:red, 1; green, 1; blue, 240 }  ,opacity=1 ]  {$3$};
\draw (421.24,80.13) node [anchor=north west][inner sep=0.75pt]  [font=\tiny,color={rgb, 255:red, 208; green, 2; blue, 27 }  ,opacity=1 ]  {$6$};
\draw (445.24,34.46) node [anchor=north west][inner sep=0.75pt]  [font=\tiny,color={rgb, 255:red, 208; green, 2; blue, 27 }  ,opacity=1 ]  {$5$};
\draw (477.24,84.46) node [anchor=north west][inner sep=0.75pt]  [font=\tiny,color={rgb, 255:red, 208; green, 2; blue, 27 }  ,opacity=1 ]  {$4$};
\draw (375,68.4) node [anchor=north west][inner sep=0.75pt]  [font=\tiny]  {$y'_{1}$};
\draw (422.67,87.73) node [anchor=north west][inner sep=0.75pt]  [font=\tiny]  {$y'_{2}$};
\draw (446.57,89.37) node [anchor=north west][inner sep=0.75pt]  [font=\tiny]  {$y'_{3}$};
\draw (440.33,48.73) node [anchor=north west][inner sep=0.75pt]  [font=\tiny]  {$y'_{4}$};
\draw (459.33,88.73) node [anchor=north west][inner sep=0.75pt]  [font=\tiny]  {$y'_{5}$};
\draw (480.33,88.73) node [anchor=north west][inner sep=0.75pt]  [font=\tiny]  {$y'_{6}$};
\draw (494.33,78.73) node [anchor=north west][inner sep=0.75pt]  [font=\tiny]  {$y'_{7}$};
\draw (470,137.4) node [anchor=north west][inner sep=0.75pt]  [font=\footnotesize]  {$\phi _{2}$};
\draw (534.57,57.13) node [anchor=north west][inner sep=0.75pt]  [font=\tiny,color={rgb, 255:red, 65; green, 117; blue, 5 }  ,opacity=1 ]  {$1_{1}$};
\draw (512.57,59.13) node [anchor=north west][inner sep=0.75pt]  [font=\tiny,color={rgb, 255:red, 65; green, 117; blue, 5 }  ,opacity=1 ]  {$1_{2}$};
\draw (522.57,97.13) node [anchor=north west][inner sep=0.75pt]  [font=\tiny,color={rgb, 255:red, 65; green, 117; blue, 5 }  ,opacity=1 ]  {$1_{3}$};

\end{tikzpicture}}
        \caption{The covers in $((\Fs\times\Ft))^{-1}(p)$ for $\ell=1$, $g=1$, and $\mu_h$'s: $(1,1,1),(1),(1),(1).$}
        \label{fig:ell=1,g=1,(1,1,1)}
    \end{figure}
        \begin{figure}[tb]
        \centering
     \resizebox{.4\textwidth}{!}{\tikzset{every picture/.style={line width=0.75pt}} 

\begin{tikzpicture}[x=0.75pt,y=0.75pt,yscale=-1,xscale=1]

\draw    (350.67,101.33) -- (323.67,79.33) ;
\draw    (350.67,101.33) -- (329.67,127.33) ;
\draw  [line width=0.75]  (350.67,101.33) .. controls (350.67,90.29) and (366.34,81.33) .. (385.67,81.33) .. controls (405,81.33) and (420.67,90.29) .. (420.67,101.33) .. controls (420.67,112.38) and (405,121.33) .. (385.67,121.33) .. controls (366.34,121.33) and (350.67,112.38) .. (350.67,101.33) -- cycle ;
\draw    (420.67,101.33) -- (492.67,102.33) ;
\draw    (420.67,101.33) -- (420.67,123.33) ;
\draw    (484.82,102.25) -- (484.82,124.25) ;
\draw  [draw opacity=0] (545.8,120) .. controls (545.7,120) and (545.6,120) .. (545.5,120) .. controls (516.51,120) and (493,112.16) .. (493,102.5) .. controls (493,92.84) and (516.51,85) .. (545.5,85) .. controls (545.6,85) and (545.69,85) .. (545.79,85) -- (545.5,102.5) -- cycle ; \draw   (545.8,120) .. controls (545.7,120) and (545.6,120) .. (545.5,120) .. controls (516.51,120) and (493,112.16) .. (493,102.5) .. controls (493,92.84) and (516.51,85) .. (545.5,85) .. controls (545.6,85) and (545.69,85) .. (545.79,85) ;  
\draw    (350.67,41.33) -- (323.67,19.33) ;
\draw    (350.67,41.33) -- (329.67,67.33) ;
\draw    (350.67,41.33) -- (418.67,52.33) ;
\draw    (418.67,52.33) -- (419.67,73.33) ;
\draw  [draw opacity=0] (418.67,52.33) .. controls (455.67,53) and (485.28,75.33) .. (484.82,102.24) .. controls (484.82,102.24) and (484.82,102.25) .. (484.82,102.25) -- (417.75,101.09) -- cycle ; \draw   (418.67,52.33) .. controls (455.67,53) and (485.28,75.33) .. (484.82,102.24) .. controls (484.82,102.24) and (484.82,102.25) .. (484.82,102.25) ;  
\draw    (492.67,102.33) -- (492.67,119.33) ;
\draw [color={rgb, 255:red, 65; green, 117; blue, 5 }  ,draw opacity=1 ]   (525,106) -- (525,118) ;
\draw [color={rgb, 255:red, 65; green, 117; blue, 5 }  ,draw opacity=1 ]   (532.55,67.37) -- (528.32,86.28) ;
\draw [color={rgb, 255:red, 65; green, 117; blue, 5 }  ,draw opacity=1 ]   (520.27,69.49) -- (527.61,85.85) ;
\draw [color={rgb, 255:red, 208; green, 2; blue, 27 }  ,draw opacity=1 ]   (434.05,90.07) -- (439.06,100.98) ;
\draw    (444.57,89.97) -- (439.45,100.82) ;
\draw    (474.1,102) -- (471.67,93.19) ;
\draw [color={rgb, 255:red, 65; green, 117; blue, 5 }  ,draw opacity=1 ]   (474.1,102) -- (478,92.67) ;
\draw    (438.33,47.33) -- (438,54.75) ;
\draw    (352.67,193.33) -- (325.67,171.33) ;
\draw    (352.67,193.33) -- (331.67,219.33) ;
\draw    (352.67,193.33) -- (542.67,193) ;
\draw    (422,193) -- (422.67,215) ;
\draw    (488,192.83) -- (488.67,214.83) ;
\draw    (498,193) -- (498.67,215) ;
\draw [color={rgb, 255:red, 208; green, 2; blue, 27 }  ,draw opacity=1 ]   (439.5,180.25) -- (438.79,193) ;
\draw [color={rgb, 255:red, 65; green, 117; blue, 5 }  ,draw opacity=1 ]   (481.33,193) -- (482,181) ;
\draw [color={rgb, 255:red, 65; green, 117; blue, 5 }  ,draw opacity=1 ]   (523.33,179.33) -- (523.45,193) ;
\draw    (462.67,135.67) -- (462.67,161.67) ;
\draw [shift={(462.67,163.67)}, rotate = 270] [color={rgb, 255:red, 0; green, 0; blue, 0 }  ][line width=0.75]    (10.93,-3.29) .. controls (6.95,-1.4) and (3.31,-0.3) .. (0,0) .. controls (3.31,0.3) and (6.95,1.4) .. (10.93,3.29)   ;
\draw [color={rgb, 255:red, 65; green, 117; blue, 5 }  ,draw opacity=1 ]   (471.6,72.5) -- (472.17,62.17) ;

\draw (313.67,75.73) node [anchor=north west][inner sep=0.75pt]  [font=\scriptsize,color={rgb, 255:red, 1; green, 1; blue, 240 }  ,opacity=1 ]  {$2$};
\draw (316.67,129.73) node [anchor=north west][inner sep=0.75pt]  [font=\scriptsize,color={rgb, 255:red, 1; green, 1; blue, 240 }  ,opacity=1 ]  {$2$};
\draw (426.67,103.4) node [anchor=north west][inner sep=0.75pt]  [font=\scriptsize,color={rgb, 255:red, 1; green, 1; blue, 240 }  ,opacity=1 ]  {$2$};
\draw (478.57,170.79) node [anchor=north west][inner sep=0.75pt]  [font=\tiny,color={rgb, 255:red, 65; green, 117; blue, 5 }  ,opacity=1 ]  {$4$};
\draw (520.57,169.79) node [anchor=north west][inner sep=0.75pt]  [font=\tiny,color={rgb, 255:red, 65; green, 117; blue, 5 }  ,opacity=1 ]  {$1$};
\draw (446,103.4) node [anchor=north west][inner sep=0.75pt]  [font=\scriptsize,color={rgb, 255:red, 1; green, 1; blue, 240 }  ,opacity=1 ]  {$2$};
\draw (548.67,80.07) node [anchor=north west][inner sep=0.75pt]  [font=\scriptsize,color={rgb, 255:red, 1; green, 1; blue, 240 }  ,opacity=1 ]  {$2$};
\draw (428.24,79.46) node [anchor=north west][inner sep=0.75pt]  [font=\tiny,color={rgb, 255:red, 208; green, 2; blue, 27 }  ,opacity=1 ]  {$6$};
\draw (470,137.4) node [anchor=north west][inner sep=0.75pt]  [font=\footnotesize]  {$\phi _{2}$};
\draw (534.57,57.13) node [anchor=north west][inner sep=0.75pt]  [font=\tiny,color={rgb, 255:red, 65; green, 117; blue, 5 }  ,opacity=1 ]  {$1_{1}$};
\draw (512.57,59.13) node [anchor=north west][inner sep=0.75pt]  [font=\tiny,color={rgb, 255:red, 65; green, 117; blue, 5 }  ,opacity=1 ]  {$1_{2}$};
\draw (522.57,97.13) node [anchor=north west][inner sep=0.75pt]  [font=\tiny,color={rgb, 255:red, 65; green, 117; blue, 5 }  ,opacity=1 ]  {$1_{3}$};
\draw (436.24,169.96) node [anchor=north west][inner sep=0.75pt]  [font=\tiny,color={rgb, 255:red, 208; green, 2; blue, 27 }  ,opacity=1 ]  {$6$};
\draw (467.57,49.29) node [anchor=north west][inner sep=0.75pt]  [font=\tiny,color={rgb, 255:red, 65; green, 117; blue, 5 }  ,opacity=1 ]  {$4_{1}$};
\draw (467.57,82.29) node [anchor=north west][inner sep=0.75pt]  [font=\tiny,color={rgb, 255:red, 65; green, 117; blue, 5 }  ,opacity=1 ]  {$4_{2}$};

\end{tikzpicture}}
        \caption{The cover in $((\Fs\times\Ft))^{-1}(p)$ for $\ell=1$, $g=1$, and $\mu_h$'s: $(1,1,1),(1,1),(1).$}
        \label{fig:ell=1,g=1,(1,1,1)(1,1)}
    \end{figure}
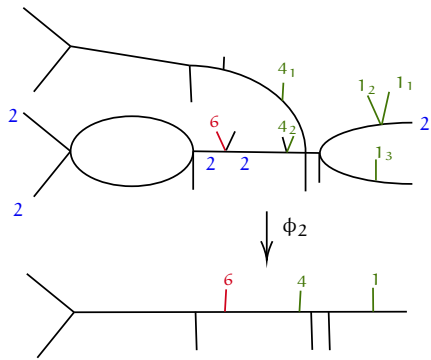

    \vspace{0.5em}
    \noindent\textbf{Construction.}
In this section we make explicit the construction of solutions in the case of $\ell$ positive and $\mu_h$ are vectors of all 1s demonstrated in the previous section. As observed in the examples, solutions are constructed following the same algorithm as detailed in the $\ell=0$ case. There is one difference coming from the fact that as $\ell$ increases, so does the number of joined ends.  

    As in the previous case, the genus part of the covers $\Gamma\to T$ are constructed in the same way as in the original case, so we focus our attention on the genus zero section containing the marked points. Recall from Section \ref{sec:ell>0}, that the marked fragments that can be placed on covers when $\ell$ is positive are of the form $F_j$ when $1+\ell+i\leq j \leq d-i$. The marked fragments can be turned into new marked fragments following the algorithm in the $\ell=0$ case.

    We now organize the covers in a rectangular grid to count using previous techniques. In Section \ref{sec:ell>0}, it was proven that $\Tev_{g,\ell}^\trop=2^g$ for positive $\ell$ and $\mu_h$ all equal to $(1)$. We now examine the other cases by comparing the grids. There are three ways the grid changes:
    \begin{enumerate}
    \item if $|\mu_1|>j$: a column is lost for each $j$ this holds for due to the lose of the corresponding marked fragment. Since $1+\ell\leq j\leq n-\ell-2$, columns are lost when $|\mu_1|>\ell+1$ and therefore lose columns corresponding to $F_{1+\ell},\dots,F_{|\mu_1|-1}$. The height of column $F_i$ is ${{g}\choose{i-1-\ell}}$. 
     \item if $|\mu_k|>n-j-1$: a column is lost for each $j$ this holds for due to the lose of the corresponding marked fragment. Each column $F_j$ has height ${{g}\choose{g-j+1+\ell}}={{g}\choose{j-1-\ell}}$.
    \item if $2$ or more fragments become the same with equivalent covers. When this occurs, all but one copy of the covers with given active edge degree are removed. Recall that $i+\ell$ is the total number of joined ends, which in the case of $\ell>0$, increases by $\ell$ from the $\ell=0$ case. the number of equivalent covers is
\(
a+b=\lvert\mu_h\rvert-i-\ell.
\)
Thus this case occurs only when
\(
\lvert\mu_h\rvert-i-\ell\geq 2,
\)
and the number of redundant copies is
\(a+b-1=\lvert\mu_h\rvert-i-\ell-1.
\)
The number of genus parts producing the corresponding active-edge degree is
\(
\binom{g}{i}-\binom{g}{i-1}.
\)
Consequently, this equivalence class removes
\[
\bigl(\lvert\mu_h\rvert-i-\ell-1\bigr)
\left(\binom{g}{i}-\binom{g}{i-1}\right)
\]
covers.
\end{enumerate}
Combining all of these and shifting the summations, we have 
 \begin{equation}
    \Tev^\trop_{g,\ell,\mu_1,\dots,\mu_k}\geq2^g -\sum_{i=0}^{\lvert\mu_1\rvert-\ell-2}{g\choose i}-\sum_{i=0}^{\lvert\mu_k\rvert-\ell-2}{g\choose i}-\sum_{h=2}^{k-1}\sum_{i=0}^{\lvert\mu_h\rvert-\ell-2}\bigg(\lvert\mu_h\rvert-i-\ell-1\bigg)\bigg({g\choose i}-{g \choose i-1}\bigg)
    \end{equation}
 when $\ell>0$.  \\  
    \vspace{0.5em}
            
    \noindent\textbf{Exclusion.}
    Section \ref{sec:excludegentev} excludes all other possible other covers $\Gamma\to T$ in the preimage of $p$, therefore we have concluded the proof of Theorem \ref{thm:higherram} when $\ell>0$ and $\mu_h$ are vectors of all 1s. 

    \subsubsection*{Negative $\ell$}
    
    \vspace{0.5em}
    \noindent\textbf{Examples.}
     Moving on to negative $\ell$ values, we start by looking at an example, shown in Figure \ref{fig:ell=-1,g=3,(1,1)}, when $\ell=-1, g=3,$ and $\mu_1=(1,1).$ Comparing this example with Figure \ref{fig:l=-1,g=3}, there is one column missing corresponding to losing that marked fragment. This is following the patterns seen in the $\ell=0 $ and positive $\ell$ cases. 

     \begin{figure}[tb]
         \centering
     \resizebox{.3\textwidth}{!}{\input{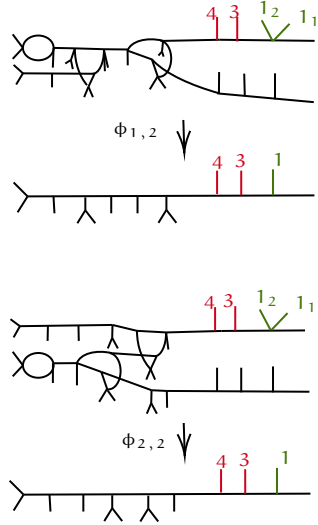}}
         \caption{The covers in $((\Fs \times \Ft))^{-1}(p)$ for $\ell=-1,g=3$, and $\mu_h$'s:$(1,1),(1),(1)$.}
         \label{fig:ell=-1,g=3,(1,1)}
     \end{figure}
    
    \vspace{0.5em}
    \noindent\textbf{Construction.}
We now make explicit the construction of solutions in the case of $\ell$ negative and $\mu_h$ are vectors of all $1$s demonstrated in the example. As in the positive $\ell$ case, solutions are constructed following the same algorithm as detailed in the $\ell=0$ case. The difference is that when $\ell$ is negative sections of the solution grid are removed in the original marked points case shown in Section \ref{sec:ell<0}. When removing sections of the grid as the $\mu_h$ change, some sections need to be added back in due to double removing. 

Recall from Theorem \ref{thm:ellnegative} that multiple rectangular section of the solution grid are removed, each with width $g-2i+1$ and height ${g\choose i} - {g\choose i-1}$. Next, we examine which covers are removed as $\mu_h$ differ from $(1)$. Similar to the positive $\ell$ case, there are $3$ situations when covers are removed from the grid:

 \begin{enumerate}
    \item if $|\mu_1|>j$: a column is lost for each $j$ this holds for due to the lose of the corresponding marked fragment. 
    \begin{enumerate}
        \item When $|\mu_1|=2$, lose $F_0$ column, which has height ${{g}\choose{-\ell}}$.
        \item When $|\mu_1|>2$, lose $F_0$ column and $F_2,\dots,F_{|\mu_1|-1}$ columns where the height of column $F_j$ is ${{g}\choose{j-1-\ell}}$.
    \end{enumerate}
     \item if $|\mu_k|>n-j-1$: a column is lost for each $j$ this holds for due to the lose of the corresponding marked fragment. Each column $F_j$ has height ${{g}\choose{g-j+1+\ell}}={{g}\choose{j-1-\ell}}$.
    \item if $2$ or more fragments become the same with equivalent covers. When this occurs, all but one copy of the covers with given active edge degree are removed. Recall that $i+\ell$ is the total number of joined ends. Since \(\ell<0\), this requires
\(
i\geq-\ell.
\)
The number of equivalent covers is
\(
a+b=\lvert\mu_h\rvert-i-\ell.
\)
Thus equivalence occurs only when
\(
\lvert\mu_h\rvert-i-\ell\geq2,
\)
and the number of redundant copies is
\(
a+b-1=\lvert\mu_h\rvert-i-\ell-1.
\)
The number of genus parts producing the corresponding active-edge degree is
\(
\binom{g}{i}-\binom{g}{i-1}.
\)
Consequently, this equivalence class removes
\[
\bigl(\lvert\mu_h\rvert-i-\ell-1\bigr)
\left(\binom{g}{i}-\binom{g}{i-1}\right)
\]
covers.
\end{enumerate}

The height of the columns in the first two ways the grid changes is subtracting some of the same covers that were removed in the rectangle sections removed in Theorem \ref{thm:ellnegative}. We must add back in these covers. There are $(\lvert\mu_1\rvert+\lvert\mu_k\rvert-2){g\choose-\ell-1}$ covers double counted. The $(\lvert\mu_1\rvert+\lvert\mu_k\rvert-2)$ is due to the fact that in (1) $\lvert\mu_1\rvert$ must be greater than 1 to lose a column and similarly $\lvert\mu_k\rvert$ must be great than 1 to lose a column in (2). Then looking at $\sum_{i=0}^{-\ell-1}\bigg({g\choose i}-{g \choose i-1}\bigg)$ and $\sum_{i=-\ell}^{\lvert\mu_1\rvert-\ell-2}{g\choose i}$, the overlap is ${g\choose-\ell-1}$. Combining all of this together we get 
\begin{align*}
        \Tev^\trop_{g,\ell,\mu_1,\dots,\mu_k}\geq&2^g - \sum_{i=0}^{-\ell-1}\bigg(g-2i+1\bigg)\bigg({g\choose i}-{g \choose i-1}\bigg)-\sum_{i=-\ell}^{\lvert\mu_1\rvert-\ell-2}{g\choose i}-\sum_{i=-\ell}^{\lvert\mu_k\rvert-\ell-2}{g\choose i}\\&+\bigg(\lvert\mu_1\rvert+\lvert\mu_k\rvert-2\bigg){g\choose-\ell-1}-\sum_{h=2}^{k-1}\sum_{i=-\ell}^{\lvert\mu_h\rvert-\ell-2}\bigg(\lvert\mu_h\rvert-i-\ell-1\bigg)\bigg({g\choose i}-{g \choose i-1}\bigg)
    \end{align*}
when $\ell<0$. 
     \vspace{0.5em}
     
    \noindent\textbf{Exclusion.}
Section \ref{sec:excludegentev} excludes all other possible other covers $\Gamma\to T$ in the preimage of $p$, therefore we have concluded the proof of Theorem \ref{thm:higherram} when $\ell<0$ and $\mu_h$ are vectors of all ones.
    
 \subsection{General ramification profiles cases come from unramified case}\label{sec:genramgentev}


In the previous two subsections, we computed tropical generalized Tevelev degrees under the assumption that every part of each ramification profile is equal to one. We now show that the degree for an arbitrary collection of ramification profiles depends only on the
integers
\[
|\mu_1|,\ldots,|\mu_k|,
\]
and not on how these integers are partitioned, which will complete the proof of Theorem \ref{thm:higherram}.

The basic step is to replace a chosen part $\alpha\geq 2$ of one
profile $\mu_h$ by the two parts $(\alpha-1,1)$. We construct a
bijection between the covers contributing to the corresponding tropical generalized Tevelev degrees and show that this bijection preserves their local degrees. Iterating this operation replaces
every profile $\mu_h$ by
\[
(1^{|\mu_h|}),
\]
thereby reducing the general case to the computation of the previous subsections.

We carry out this reduction in three steps. First, we isolate the portion of the cover containing the marked preimages associated to a single part of one ramification profile. Additional marked points
lying on the same tree are forced, introduce no new combinatorial choices, and contribute a factor of one to the local degree. We may therefore study one part at a time. Second, according to whether its marked preimages lie on cuts, on the active edge, or on joins, we
construct a reversible local modification replacing the part $\alpha$ by $(\alpha-1,1)$ while leaving the remainder of the cover unchanged. Finally, we compare the multiplicities of the two covers
using Equation~\eqref{eq:locdegforlater}. The automorphism contribution is unchanged, and in the cases where new factors occur, the changes in the local Hurwitz number, the lattice index, and the
normalization weight cancel. Hence the local modification preserves the local degree. It follows that the tropical generalized Tevelev degree depends only on the values $|\mu_h|$, completing the reduction
to profiles with all parts equal to one.

\subsubsection*{Example}
Figures \ref{fig:ell=0,g=1,(2)(1,1)} and \ref{fig:ell=0,g=1,(2)(2)} together with previous Figure \ref{fig:l=0g=2(1,1)(1,1)} are an example that the number of covers counted when computing the tropical generalized Tevelev degrees only depends on the magnitude of each $\mu_h$. 
     \begin{figure}[tb]
         \centering
     \resizebox{.3\textwidth}{!}{\input{Figures/l=0,g=2,_1,1__2_}}
         \caption{The covers in $((\Fs \times \Ft))^{-1}(p)$ for $\ell=0,g=2$, and $\mu_h$'s:$(1,1),(2),(1)$.}
         \label{fig:ell=0,g=1,(2)(1,1)}
     \end{figure}

\begin{figure}[tb]
         \centering
     \resizebox{.3\textwidth}{!}{\input{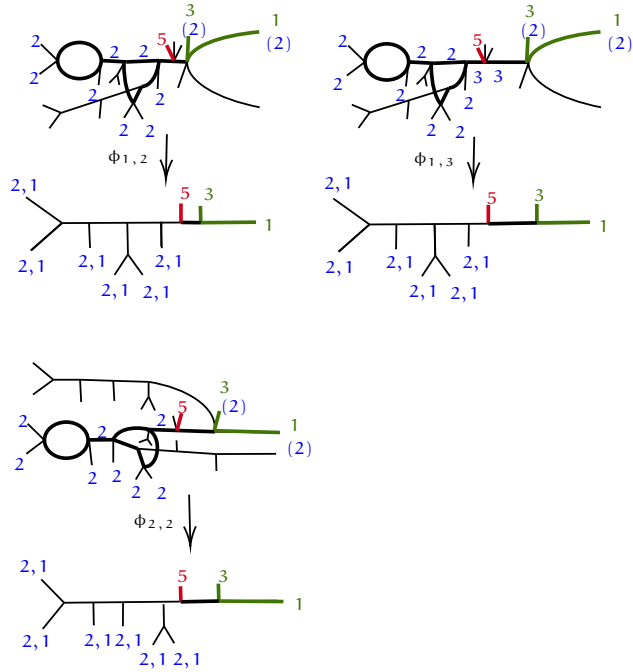}}
         \caption{The covers in $((\Fs \times \Ft))^{-1}(p)$ for $\ell=0,g=2$, and $\mu_h$'s:$(2),(2),(1)$.}
         \label{fig:ell=0,g=1,(2)(2)}
     \end{figure}

\subsubsection*{Isolation}

To simplify the problem we look at each marked point individually. Due to the chosen point $p=(\overline{\Gamma},\overline{T})\in \Mgn{g,n}^\trop\times\Mgn{0,k}^\trop$, all of the preimages of a single marked point must stabilize to the active edge in a row. In particular, the preimages must stabilize in the order of the partition $\mu_h$ associated to the given marked point. Therefore, a cover can be cut along the active edge on either side of any number of cuts and joins that the given marked points stabilize along, allowing the problem to be specialized to one part of a single partition. 

To avoid a case-by-case analysis of all possible numbers and configurations of marked
points lying on a given tree, we reduce the problem to the contribution of a single
marked point at a time. The justification for this reduction is that any additional
marked points on the same tree are forced to behave uniformly: they introduce no new
combinatorial choices and do not affect the multiplicity of a cover. The following
lemma makes this precise.

\begin{lemma}
Let \(\varphi\colon\Gamma\to T\) be a tropical admissible cover
contributing to \(\deg(\Fs\times\Ft)\), and fix a marked point \(p_h\)
of \(T\) with associated ramification profile \(\mu_h\). Suppose that
\(p_h\) lies on a marked tree containing additional marked points.

For each part \(a=\mu_{h,s}\), let \(p_{h,s}\) be the corresponding
marked end of \(\Gamma\), and let \(B_{h,s}\) be the edge joining
\(p_{h,s}\) to the active path. If \(q\) is another marked point on the
same target tree, then the vertex of \(B_{h,s}\) lying over the vertex
of \(T\) adjacent to \(q\) has two adjacent edges of expansion factor
\(a\) and \(a\) adjacent unmarked ends of expansion factor \(1\) mapping
to \(q\).

Consequently, the additional marked points introduce no choices in the
construction of \(B_{h,s}\). Moreover, for each
such placement, the product of the factors in the local-degree formula
arising from the additional marked points is equal to \(1\).
Consequently, their presence affects neither the enumeration nor the
multiplicity of the covers.
\end{lemma}

\begin{proof}
Fix a part \(a=\mu_{h,s}\) and the corresponding marked end of
\(\Gamma\). Since the preimages associated with different marked
points stabilize to the active path separately, the edge containing
this marked end contains no marked preimages of the other marked
points on the same target tree. Therefore, the ends on this edge
mapping to an additional marked point are unmarked and have expansion
factor \(1\).

Harmonicity and the local Riemann--Hurwitz condition then force the
local ramification profiles to be
\[
(a),\qquad (a),\qquad (1^a).
\]
Thus the additional marked point introduces no further combinatorial
choice on this edge. The corresponding local Hurwitz number is
\[
H_0\bigl((a),(a),(1^a)\bigr)=\frac{1}{a}.
\]

The contribution of the determinant and edge product divided by least common multiple factors is
\(a\). If \(b\) is the expansion factor, or degree, of the edge retained
after stabilization and \(M=\operatorname{lcm}(a,b)\), then these two
factors are \(M/b\) and \(ab/M\), respectively, and
\[
\frac{M}{b}\cdot\frac{ab}{M}=a.
\]
The automorphism factor is unchanged. Hence the total factor contributed
by the additional marked point is
\[
\frac{1}{a}\cdot a=1.
\]
Applying the same argument to every additional marked point proves the
claim.
\end{proof}


\subsubsection*{Bijection}

Set
\[
m=\sum_{h=1}^k r_h.
\]
Let
\[
b(\mu)=2g+2d-2-\sum_{h=1}^k|\mu_h|+m.
\]
Replacing the entry $\alpha$ by $(\alpha-1,1)$ increases $m$ to $m+1$ and leaves $\sum_h|\mu_h|$ unchanged. Hence
\(
b(\mu')=b(\mu)+1.
\)
Equivalently, the ramification contribution of $\alpha$ is $\alpha-1$, whereas that of $(\alpha-1,1)$ is $\alpha-2$. The Riemann--Hurwitz condition therefore forces the $\mu'$-cover to have one additional simple branch end. This end and all of its inverse-image ends are unmarked. Consequently, the number of retained target markings remains $k$, while the number of retained source markings increases from $m$ to $m+1$. Thus the two evaluation--forgetful morphisms have targets
\[
\mathcal M_{g,m}^{\trop}\times\mathcal M_{0,k}^{\trop}
\qquad\text{and}\qquad
\mathcal M_{g,m+1}^{\trop}\times\mathcal M_{0,k}^{\trop},
\]
respectively.

Let \(x_\star\) denote the new source marking corresponding to the
inserted part \(1\), and let
\[
\pi_\star\colon\Mgn{g,m+1}^{\trop}\longrightarrow
\Mgn{g,m}^{\trop}
\]
forget \(x_\star\) and stabilize. Choose a generic point
\(
p'=(\overline{\Gamma}',\overline{T})
\)
in the caterpillar cone obtained by inserting \(x_\star\) immediately
after the marking corresponding to \(\alpha-1\), and set
\[
p=(\pi_\star(\overline{\Gamma}'),\overline{T}).
\]
We call \(p\) and \(p'\) compatible.

\begin{proposition}
\label{prop:splitting-part}
Let \(\mu'\) be obtained from \(\mu\) by replacing an entry
\(\
e_{i,j}=\alpha\geq2
\)
by \((\alpha-1,1)\). For compatible generic generic
points \(p\) and \(p'\) as above, the local modification described
below induces a multiplicity-preserving bijection
\[
(\Fs^\mu\times\Ft^\mu)^{-1}(p)
\longleftrightarrow
(\Fs^{\mu'}\times\Ft^{\mu'})^{-1}(p').
\]
Consequently,
\[
\Tev^{\trop}_{g,\ell,\mu}
=
\Tev^{\trop}_{g,\ell,\mu'}.
\]
\end{proposition}

We now describe a local bijection between tropical covers contributing with ramification
profile $(\alpha)$ and those contributing with profile $(\alpha-1,1)$. The bijection is defined
entirely at the level of the local configuration of marked preimages along the active edge, while
leaving the remainder of the cover fixed.

Fix a marked point of $T$ and consider the marked preimages corresponding to a single part
$\alpha$ of its ramification profile. By stabilization, these preimages occur along the active
path in a contiguous block, appearing in the following order: first on cuts, then on the active
edge, and finally on joins. The bijection consists of a single local operation that either combines
two adjacent marked preimages into one or separates one marked preimage into two adjacent
ones, thereby changing the local ramification profile from $(\alpha)$ to $(\alpha-1,1)$ or vice
versa.

Recall that in order for a point in the target $T$ to have preimages on cuts, joins and the active edge, the cuts most occur to the left of the joins with the marked preimage on the active edge in the middle. Therefore we can classify the possible configurations of the marked preimage based on if they lie on cuts and/or the active edge and/or joins. We now describe the unique local modification in each case:
\begin{itemize}
    \item \emph{All preimages lie on joins:} a preimage can be combined or separated on a join but not on a cut because there is not preimage marked on the active edge. A preimage on the active edge can be combined resulting in a vertex on the active edge containing the marked preimage only when $\alpha=2$.
    \item \emph{Preimages lie on joins and the active edge:} a preimage can be combined or separated join but not on a cut because that would violate the order of $(\alpha-1,1)$ when stabilizing. A preimage on the active edge can be separated. 
    \item \emph{Preimages lie on joins, active edge, and cuts:} similar to above, a preimage can be combined or separated on a join. Due to the cuts, a preimage on the active edge can not be separated. 
    \item \emph{All preimages lie on cuts:} a preimage can be combined or separated on a cut but not a join. A preimage on the active edge can be combined resulting in a a vertex on the active edge containing the marked preimage.
    \item \emph{Preimages lie on cuts and active edge:} a preimage can combined on a join or a preimage can be separated on the active edge. 
    \item \emph{Preimage lies on the active edge:} this case appears on the $(\alpha)$ side of the active edge with joins case as well as the active edge with cuts case, therefore it is not a new case.  

\end{itemize}
Figures \ref{fig:alljoins}, \ref{fig:joins+active}, \ref{fig:cut+joins+active}, \ref{fig:allcuts}, and \ref{fig:allcuts+activeedge} show the first 5 cases, respectively, the last one comes from the previous. In each of the cases above, exactly one local modification is possible, determined uniquely by the
relative positions of the marked preimages on cuts, joins, and the active edge. Moreover,
the operation is reversible: applying the inverse local modification recovers the original
configuration. Therefore, the constructions above define a bijection between the sets of
covers contributing with ramification profiles $(\alpha)$ and $(\alpha-1,1)$.

When passing from $(\alpha)$ to $(\alpha-1,1)$, the local modification also inserts the additional unmarked simple branch end forced by the Riemann--Hurwitz condition; the inverse modification removes it. More precisely, the inverse removes the unique such branch end adjacent to the local configuration in which the two indicated marked preimages are combined. Outside this local configuration, the cover is unchanged. In Figures~39--43, this is the additional target end denoted by $2$ on the $(\alpha-1,1)$-side.

\begin{figure}
    \centering
     \resizebox{.8\textwidth}{!}{\tikzset{every picture/.style={line width=0.75pt}} 

\begin{tikzpicture}[x=0.75pt,y=0.75pt,yscale=-1,xscale=1]

\draw    (59,131) -- (220,131) ;
\draw    (101,82) .. controls (139,84) and (157,116) .. (159,131) ;
\draw    (101,82) -- (58,67) ;
\draw    (101,82) -- (59,97) ;
\draw    (101,82) -- (59,76) ;
\draw    (98,131) -- (89.5,118) ;
\draw    (98,131) -- (109.5,120) ;
\draw [color={rgb, 255:red, 208; green, 2; blue, 27 }  ,draw opacity=1 ][line width=1.5]    (101,82) -- (101,67) ;
\draw    (60,211) -- (218,211) ;
\draw [color={rgb, 255:red, 208; green, 2; blue, 27 }  ,draw opacity=1 ][line width=1.5]    (99,211) -- (99,195) ;
\draw    (159,211) -- (159,229) ;
\draw    (145,154) -- (145,179) ;
\draw [shift={(145,181)}, rotate = 270] [color={rgb, 255:red, 0; green, 0; blue, 0 }  ][line width=0.75]    (10.93,-3.29) .. controls (6.95,-1.4) and (3.31,-0.3) .. (0,0) .. controls (3.31,0.3) and (6.95,1.4) .. (10.93,3.29)   ;
\draw    (311,131) -- (472,131) ;
\draw    (353,82) .. controls (391,84) and (409,116) .. (411,131) ;
\draw    (353,82) -- (310,67) ;
\draw    (353,82) -- (311,97) ;
\draw    (353,82) -- (311,76) ;
\draw    (350,131) -- (341.5,118) ;
\draw    (350,131) -- (361.5,119) ;
\draw [color={rgb, 255:red, 208; green, 2; blue, 27 }  ,draw opacity=1 ][line width=1.5]    (353,82) -- (353,67) ;
\draw    (312,211) -- (470,211) ;
\draw [color={rgb, 255:red, 208; green, 2; blue, 27 }  ,draw opacity=1 ][line width=1.5]    (351,211) -- (351,195) ;
\draw    (411,211) -- (411,229) ;
\draw    (397,154) -- (397,179) ;
\draw [shift={(397,181)}, rotate = 270] [color={rgb, 255:red, 0; green, 0; blue, 0 }  ][line width=0.75]    (10.93,-3.29) .. controls (6.95,-1.4) and (3.31,-0.3) .. (0,0) .. controls (3.31,0.3) and (6.95,1.4) .. (10.93,3.29)   ;
\draw    (352,47) .. controls (415,45) and (438,117) .. (440,131) ;
\draw    (352,47) -- (311,46) ;
\draw [color={rgb, 255:red, 208; green, 2; blue, 27 }  ,draw opacity=1 ][line width=1.5]    (352,47) -- (352,32) ;
\draw    (442,211) -- (442,229) ;
\draw    (245,162) -- (280,162) ;
\draw [shift={(282,162)}, rotate = 180] [color={rgb, 255:red, 0; green, 0; blue, 0 }  ][line width=0.75]    (10.93,-3.29) .. controls (6.95,-1.4) and (3.31,-0.3) .. (0,0) .. controls (3.31,0.3) and (6.95,1.4) .. (10.93,3.29)   ;
\draw    (282,162) -- (245,162) ;
\draw    (282,162) -- (247,162) ;
\draw [shift={(245,162)}, rotate = 360] [color={rgb, 255:red, 0; green, 0; blue, 0 }  ][line width=0.75]    (10.93,-3.29) .. controls (6.95,-1.4) and (3.31,-0.3) .. (0,0) .. controls (3.31,0.3) and (6.95,1.4) .. (10.93,3.29)   ;

\draw (55,72.4) node [anchor=north west][inner sep=0.75pt]  [font=\small]  {$\vdots $};
\draw (61,119.4) node [anchor=north west][inner sep=0.75pt]  [font=\tiny,color={rgb, 255:red, 31; green, 47; blue, 255 }  ,opacity=1 ]  {$d_{0}$};
\draw (194,120.4) node [anchor=north west][inner sep=0.75pt]  [font=\tiny,color={rgb, 255:red, 31; green, 47; blue, 255 }  ,opacity=1 ]  {$d_{0} +\alpha $};
\draw (138,84.4) node [anchor=north west][inner sep=0.75pt]  [font=\tiny,color={rgb, 255:red, 31; green, 47; blue, 255 }  ,opacity=1 ]  {$\alpha $};
\draw (27,76.4) node [anchor=north west][inner sep=0.75pt]  [font=\tiny,color={rgb, 255:red, 31; green, 47; blue, 255 }  ,opacity=1 ]  {$\alpha $};
\draw (38,62.4) node [anchor=north west][inner sep=0.75pt]  [font=\Huge]  {$\{$};
\draw (50,61.4) node [anchor=north west][inner sep=0.75pt]  [font=\tiny,color={rgb, 255:red, 31; green, 47; blue, 255 }  ,opacity=1 ]  {$1$};
\draw (50,93.4) node [anchor=north west][inner sep=0.75pt]  [font=\tiny,color={rgb, 255:red, 31; green, 47; blue, 255 }  ,opacity=1 ]  {$1$};
\draw (91,117.4) node [anchor=north west][inner sep=0.75pt]  [font=\tiny]  {$\dotsc $};
\draw (97,56.4) node [anchor=north west][inner sep=0.75pt]  [font=\tiny,color={rgb, 255:red, 31; green, 47; blue, 255 }  ,opacity=1 ]  {$\alpha $};
\draw (156,233.4) node [anchor=north west][inner sep=0.75pt]  [font=\tiny,color={rgb, 255:red, 31; green, 47; blue, 255 }  ,opacity=1 ]  {$2$};
\draw (93,182.4) node [anchor=north west][inner sep=0.75pt]  [font=\tiny,color={rgb, 255:red, 31; green, 47; blue, 255 }  ,opacity=1 ]  {$( \alpha )$};
\draw (307,72.4) node [anchor=north west][inner sep=0.75pt]  [font=\small]  {$\vdots $};
\draw (313,119.4) node [anchor=north west][inner sep=0.75pt]  [font=\tiny,color={rgb, 255:red, 31; green, 47; blue, 255 }  ,opacity=1 ]  {$d_{0}$};
\draw (447,119.4) node [anchor=north west][inner sep=0.75pt]  [font=\tiny,color={rgb, 255:red, 31; green, 47; blue, 255 }  ,opacity=1 ]  {$d_{0} +\alpha $};
\draw (381,80.4) node [anchor=north west][inner sep=0.75pt]  [font=\tiny,color={rgb, 255:red, 31; green, 47; blue, 255 }  ,opacity=1 ]  {$\alpha -1$};
\draw (263,78.4) node [anchor=north west][inner sep=0.75pt]  [font=\tiny,color={rgb, 255:red, 31; green, 47; blue, 255 }  ,opacity=1 ]  {$\alpha -1$};
\draw (289,62.4) node [anchor=north west][inner sep=0.75pt]  [font=\Huge]  {$\{$};
\draw (302,61.4) node [anchor=north west][inner sep=0.75pt]  [font=\tiny,color={rgb, 255:red, 31; green, 47; blue, 255 }  ,opacity=1 ]  {$1$};
\draw (302,93.4) node [anchor=north west][inner sep=0.75pt]  [font=\tiny,color={rgb, 255:red, 31; green, 47; blue, 255 }  ,opacity=1 ]  {$1$};
\draw (343,117.4) node [anchor=north west][inner sep=0.75pt]  [font=\tiny]  {$\dotsc $};
\draw (349,56.4) node [anchor=north west][inner sep=0.75pt]  [font=\tiny,color={rgb, 255:red, 31; green, 47; blue, 255 }  ,opacity=1 ]  {$\alpha -1$};
\draw (408,233.4) node [anchor=north west][inner sep=0.75pt]  [font=\tiny,color={rgb, 255:red, 31; green, 47; blue, 255 }  ,opacity=1 ]  {$2$};
\draw (336,181.4) node [anchor=north west][inner sep=0.75pt]  [font=\tiny,color={rgb, 255:red, 31; green, 47; blue, 255 }  ,opacity=1 ]  {$( \alpha -1,1)$};
\draw (121,120.4) node [anchor=north west][inner sep=0.75pt]  [font=\tiny,color={rgb, 255:red, 31; green, 47; blue, 255 }  ,opacity=1 ]  {$d_{0}$};
\draw (374,120.4) node [anchor=north west][inner sep=0.75pt]  [font=\tiny,color={rgb, 255:red, 31; green, 47; blue, 255 }  ,opacity=1 ]  {$d_{0}$};
\draw (348,21.4) node [anchor=north west][inner sep=0.75pt]  [font=\tiny,color={rgb, 255:red, 31; green, 47; blue, 255 }  ,opacity=1 ]  {$1$};
\draw (413,134.4) node [anchor=north west][inner sep=0.75pt]  [font=\tiny,color={rgb, 255:red, 31; green, 47; blue, 255 }  ,opacity=1 ]  {$d_{0} +\alpha -1$};
\draw (402,54.4) node [anchor=north west][inner sep=0.75pt]  [font=\tiny,color={rgb, 255:red, 31; green, 47; blue, 255 }  ,opacity=1 ]  {$1$};
\draw (439,233.4) node [anchor=north west][inner sep=0.75pt]  [font=\tiny,color={rgb, 255:red, 31; green, 47; blue, 255 }  ,opacity=1 ]  {$2$};

\end{tikzpicture}}
    \caption{The all joins case of transitioning between the ramification indices $(\alpha)$ and $(\alpha-1,1)$.}
    \label{fig:alljoins}
\end{figure}

\begin{figure}
    \centering
     \resizebox{.8\textwidth}{!}{\tikzset{every picture/.style={line width=0.75pt}} 

\begin{tikzpicture}[x=0.75pt,y=0.75pt,yscale=-1,xscale=1]

\draw    (59,131) -- (220,131) ;
\draw    (98,131) -- (63,103) ;
\draw    (98,131) -- (62,123) ;
\draw    (98,131) -- (98,119) ;
\draw    (98,131) -- (114,124) ;
\draw [color={rgb, 255:red, 208; green, 2; blue, 27 }  ,draw opacity=1 ][line width=1.5]    (98,131) -- (92,117) ;
\draw    (60,211) -- (218,211) ;
\draw [color={rgb, 255:red, 208; green, 2; blue, 27 }  ,draw opacity=1 ][line width=1.5]    (99,211) -- (99,195) ;
\draw    (145,154) -- (145,179) ;
\draw [shift={(145,181)}, rotate = 270] [color={rgb, 255:red, 0; green, 0; blue, 0 }  ][line width=0.75]    (10.93,-3.29) .. controls (6.95,-1.4) and (3.31,-0.3) .. (0,0) .. controls (3.31,0.3) and (6.95,1.4) .. (10.93,3.29)   ;
\draw    (311,131) -- (472,131) ;
\draw    (312,211) -- (470,211) ;
\draw [color={rgb, 255:red, 208; green, 2; blue, 27 }  ,draw opacity=1 ][line width=1.5]    (351,211) -- (351,195) ;
\draw    (397,154) -- (397,179) ;
\draw [shift={(397,181)}, rotate = 270] [color={rgb, 255:red, 0; green, 0; blue, 0 }  ][line width=0.75]    (10.93,-3.29) .. controls (6.95,-1.4) and (3.31,-0.3) .. (0,0) .. controls (3.31,0.3) and (6.95,1.4) .. (10.93,3.29)   ;
\draw    (352,47) .. controls (415,45) and (438,117) .. (440,131) ;
\draw    (352,47) -- (311,46) ;
\draw [color={rgb, 255:red, 208; green, 2; blue, 27 }  ,draw opacity=1 ][line width=1.5]    (352,47) -- (352,32) ;
\draw    (442,211) -- (442,229) ;
\draw    (351,130) -- (316,102) ;
\draw    (351,130) -- (315,122) ;
\draw    (351,130) -- (351,118) ;
\draw    (351,130) -- (367,123) ;
\draw [color={rgb, 255:red, 208; green, 2; blue, 27 }  ,draw opacity=1 ][line width=1.5]    (351,130) -- (345,116) ;
\draw    (247,165) -- (282,165) ;
\draw [shift={(284,165)}, rotate = 180] [color={rgb, 255:red, 0; green, 0; blue, 0 }  ][line width=0.75]    (10.93,-3.29) .. controls (6.95,-1.4) and (3.31,-0.3) .. (0,0) .. controls (3.31,0.3) and (6.95,1.4) .. (10.93,3.29)   ;
\draw    (284,165) -- (247,165) ;
\draw    (284,165) -- (249,165) ;
\draw [shift={(247,165)}, rotate = 360] [color={rgb, 255:red, 0; green, 0; blue, 0 }  ][line width=0.75]    (10.93,-3.29) .. controls (6.95,-1.4) and (3.31,-0.3) .. (0,0) .. controls (3.31,0.3) and (6.95,1.4) .. (10.93,3.29)   ;

\draw (61,99.4) node [anchor=north west][inner sep=0.75pt]  [font=\tiny]  {$\vdots $};
\draw (45,125.4) node [anchor=north west][inner sep=0.75pt]  [font=\tiny,color={rgb, 255:red, 31; green, 47; blue, 255 }  ,opacity=1 ]  {$d_{0}$};
\draw (181,120.4) node [anchor=north west][inner sep=0.75pt]  [font=\tiny,color={rgb, 255:red, 31; green, 47; blue, 255 }  ,opacity=1 ]  {$d_{0} +\alpha -1$};
\draw (16,107.4) node [anchor=north west][inner sep=0.75pt]  [font=\tiny,color={rgb, 255:red, 31; green, 47; blue, 255 }  ,opacity=1 ]  {$\alpha -1$};
\draw (43,95.4) node [anchor=north west][inner sep=0.75pt]  [font=\huge]  {$\{$};
\draw (53,118.4) node [anchor=north west][inner sep=0.75pt]  [font=\tiny,color={rgb, 255:red, 31; green, 47; blue, 255 }  ,opacity=1 ]  {$1$};
\draw (54,97.4) node [anchor=north west][inner sep=0.75pt]  [font=\tiny,color={rgb, 255:red, 31; green, 47; blue, 255 }  ,opacity=1 ]  {$1$};
\draw (98.5,120.4) node [anchor=north west][inner sep=0.75pt]  [font=\tiny]  {$\dotsc $};
\draw (85,107.4) node [anchor=north west][inner sep=0.75pt]  [font=\tiny,color={rgb, 255:red, 31; green, 47; blue, 255 }  ,opacity=1 ]  {$\alpha $};
\draw (93,182.4) node [anchor=north west][inner sep=0.75pt]  [font=\tiny,color={rgb, 255:red, 31; green, 47; blue, 255 }  ,opacity=1 ]  {$( \alpha )$};
\draw (447,119.4) node [anchor=north west][inner sep=0.75pt]  [font=\tiny,color={rgb, 255:red, 31; green, 47; blue, 255 }  ,opacity=1 ]  {$d_{0} +\alpha -1$};
\draw (336,181.4) node [anchor=north west][inner sep=0.75pt]  [font=\tiny,color={rgb, 255:red, 31; green, 47; blue, 255 }  ,opacity=1 ]  {$( \alpha -1,1)$};
\draw (348,21.4) node [anchor=north west][inner sep=0.75pt]  [font=\tiny,color={rgb, 255:red, 31; green, 47; blue, 255 }  ,opacity=1 ]  {$1$};
\draw (383,119.4) node [anchor=north west][inner sep=0.75pt]  [font=\tiny,color={rgb, 255:red, 31; green, 47; blue, 255 }  ,opacity=1 ]  {$d_{0} +\alpha -2$};
\draw (402,54.4) node [anchor=north west][inner sep=0.75pt]  [font=\tiny,color={rgb, 255:red, 31; green, 47; blue, 255 }  ,opacity=1 ]  {$1$};
\draw (439,233.4) node [anchor=north west][inner sep=0.75pt]  [font=\tiny,color={rgb, 255:red, 31; green, 47; blue, 255 }  ,opacity=1 ]  {$2$};
\draw (318,99.4) node [anchor=north west][inner sep=0.75pt]  [font=\tiny]  {$\vdots $};
\draw (298,124.4) node [anchor=north west][inner sep=0.75pt]  [font=\tiny,color={rgb, 255:red, 31; green, 47; blue, 255 }  ,opacity=1 ]  {$d_{0}$};
\draw (269,106.4) node [anchor=north west][inner sep=0.75pt]  [font=\tiny,color={rgb, 255:red, 31; green, 47; blue, 255 }  ,opacity=1 ]  {$\alpha -2$};
\draw (296,94.4) node [anchor=north west][inner sep=0.75pt]  [font=\huge]  {$\{$};
\draw (306,117.4) node [anchor=north west][inner sep=0.75pt]  [font=\tiny,color={rgb, 255:red, 31; green, 47; blue, 255 }  ,opacity=1 ]  {$1$};
\draw (307,96.4) node [anchor=north west][inner sep=0.75pt]  [font=\tiny,color={rgb, 255:red, 31; green, 47; blue, 255 }  ,opacity=1 ]  {$1$};
\draw (351.5,119.4) node [anchor=north west][inner sep=0.75pt]  [font=\tiny]  {$\dotsc $};
\draw (338,106.4) node [anchor=north west][inner sep=0.75pt]  [font=\tiny,color={rgb, 255:red, 31; green, 47; blue, 255 }  ,opacity=1 ]  {$\alpha -1$};

\end{tikzpicture}}
    \caption{The active edge with joins case of transitioning between the ramification indices $(\alpha)$ and $(\alpha-1,1)$.}
    \label{fig:joins+active}
\end{figure}

\begin{figure}
    \centering
     \resizebox{.8\textwidth}{!}{\tikzset{every picture/.style={line width=0.75pt}} 

\begin{tikzpicture}[x=0.75pt,y=0.75pt,yscale=-1,xscale=1]

\draw    (59,131) -- (220,131) ;
\draw    (98,131) -- (63,103) ;
\draw    (98,131) -- (62,123) ;
\draw    (98,131) -- (98,119) ;
\draw    (98,131) -- (114,124) ;
\draw [color={rgb, 255:red, 208; green, 2; blue, 27 }  ,draw opacity=1 ][line width=1.5]    (98,131) -- (92,117) ;
\draw    (60,211) -- (218,211) ;
\draw [color={rgb, 255:red, 208; green, 2; blue, 27 }  ,draw opacity=1 ][line width=1.5]    (99,211) -- (99,195) ;
\draw    (145,154) -- (145,179) ;
\draw [shift={(145,181)}, rotate = 270] [color={rgb, 255:red, 0; green, 0; blue, 0 }  ][line width=0.75]    (10.93,-3.29) .. controls (6.95,-1.4) and (3.31,-0.3) .. (0,0) .. controls (3.31,0.3) and (6.95,1.4) .. (10.93,3.29)   ;
\draw    (311,131) -- (472,131) ;
\draw    (312,211) -- (470,211) ;
\draw [color={rgb, 255:red, 208; green, 2; blue, 27 }  ,draw opacity=1 ][line width=1.5]    (351,211) -- (351,195) ;
\draw    (397,154) -- (397,179) ;
\draw [shift={(397,181)}, rotate = 270] [color={rgb, 255:red, 0; green, 0; blue, 0 }  ][line width=0.75]    (10.93,-3.29) .. controls (6.95,-1.4) and (3.31,-0.3) .. (0,0) .. controls (3.31,0.3) and (6.95,1.4) .. (10.93,3.29)   ;
\draw    (352,47) .. controls (415,45) and (438,117) .. (440,131) ;
\draw    (352,47) -- (311,46) ;
\draw [color={rgb, 255:red, 208; green, 2; blue, 27 }  ,draw opacity=1 ][line width=1.5]    (352,47) -- (352,32) ;
\draw    (442,211) -- (442,229) ;
\draw    (351,130) -- (316,102) ;
\draw    (351,130) -- (315,122) ;
\draw    (351,130) -- (351,118) ;
\draw    (351,130) -- (367,123) ;
\draw [color={rgb, 255:red, 208; green, 2; blue, 27 }  ,draw opacity=1 ][line width=1.5]    (351,130) -- (345,116) ;
\draw    (98,131) -- (220,139) ;
\draw    (351,130) -- (473,138) ;
\draw    (245,164) -- (280,164) ;
\draw [shift={(282,164)}, rotate = 180] [color={rgb, 255:red, 0; green, 0; blue, 0 }  ][line width=0.75]    (10.93,-3.29) .. controls (6.95,-1.4) and (3.31,-0.3) .. (0,0) .. controls (3.31,0.3) and (6.95,1.4) .. (10.93,3.29)   ;
\draw    (282,164) -- (245,164) ;
\draw    (282,164) -- (247,164) ;
\draw [shift={(245,164)}, rotate = 360] [color={rgb, 255:red, 0; green, 0; blue, 0 }  ][line width=0.75]    (10.93,-3.29) .. controls (6.95,-1.4) and (3.31,-0.3) .. (0,0) .. controls (3.31,0.3) and (6.95,1.4) .. (10.93,3.29)   ;

\draw (61,99.4) node [anchor=north west][inner sep=0.75pt]  [font=\tiny]  {$\vdots $};
\draw (45,125.4) node [anchor=north west][inner sep=0.75pt]  [font=\tiny,color={rgb, 255:red, 31; green, 47; blue, 255 }  ,opacity=1 ]  {$d_{0}$};
\draw (181,120.4) node [anchor=north west][inner sep=0.75pt]  [font=\tiny,color={rgb, 255:red, 31; green, 47; blue, 255 }  ,opacity=1 ]  {$d_{0} +\alpha -3$};
\draw (16,107.4) node [anchor=north west][inner sep=0.75pt]  [font=\tiny,color={rgb, 255:red, 31; green, 47; blue, 255 }  ,opacity=1 ]  {$\alpha -2$};
\draw (43,95.4) node [anchor=north west][inner sep=0.75pt]  [font=\huge]  {$\{$};
\draw (53,118.4) node [anchor=north west][inner sep=0.75pt]  [font=\tiny,color={rgb, 255:red, 31; green, 47; blue, 255 }  ,opacity=1 ]  {$1$};
\draw (54,97.4) node [anchor=north west][inner sep=0.75pt]  [font=\tiny,color={rgb, 255:red, 31; green, 47; blue, 255 }  ,opacity=1 ]  {$1$};
\draw (98.5,120.4) node [anchor=north west][inner sep=0.75pt]  [font=\tiny]  {$\dotsc $};
\draw (85,107.4) node [anchor=north west][inner sep=0.75pt]  [font=\tiny,color={rgb, 255:red, 31; green, 47; blue, 255 }  ,opacity=1 ]  {$\alpha $};
\draw (93,182.4) node [anchor=north west][inner sep=0.75pt]  [font=\tiny,color={rgb, 255:red, 31; green, 47; blue, 255 }  ,opacity=1 ]  {$( \alpha )$};
\draw (447,119.4) node [anchor=north west][inner sep=0.75pt]  [font=\tiny,color={rgb, 255:red, 31; green, 47; blue, 255 }  ,opacity=1 ]  {$d_{0} +\alpha -3$};
\draw (336,181.4) node [anchor=north west][inner sep=0.75pt]  [font=\tiny,color={rgb, 255:red, 31; green, 47; blue, 255 }  ,opacity=1 ]  {$( \alpha -1,1)$};
\draw (348,21.4) node [anchor=north west][inner sep=0.75pt]  [font=\tiny,color={rgb, 255:red, 31; green, 47; blue, 255 }  ,opacity=1 ]  {$1$};
\draw (383,119.4) node [anchor=north west][inner sep=0.75pt]  [font=\tiny,color={rgb, 255:red, 31; green, 47; blue, 255 }  ,opacity=1 ]  {$d_{0} +\alpha -4$};
\draw (402,54.4) node [anchor=north west][inner sep=0.75pt]  [font=\tiny,color={rgb, 255:red, 31; green, 47; blue, 255 }  ,opacity=1 ]  {$1$};
\draw (439,233.4) node [anchor=north west][inner sep=0.75pt]  [font=\tiny,color={rgb, 255:red, 31; green, 47; blue, 255 }  ,opacity=1 ]  {$2$};
\draw (318,99.4) node [anchor=north west][inner sep=0.75pt]  [font=\tiny]  {$\vdots $};
\draw (298,124.4) node [anchor=north west][inner sep=0.75pt]  [font=\tiny,color={rgb, 255:red, 31; green, 47; blue, 255 }  ,opacity=1 ]  {$d_{0}$};
\draw (269,106.4) node [anchor=north west][inner sep=0.75pt]  [font=\tiny,color={rgb, 255:red, 31; green, 47; blue, 255 }  ,opacity=1 ]  {$\alpha -3$};
\draw (296,94.4) node [anchor=north west][inner sep=0.75pt]  [font=\huge]  {$\{$};
\draw (306,117.4) node [anchor=north west][inner sep=0.75pt]  [font=\tiny,color={rgb, 255:red, 31; green, 47; blue, 255 }  ,opacity=1 ]  {$1$};
\draw (307,96.4) node [anchor=north west][inner sep=0.75pt]  [font=\tiny,color={rgb, 255:red, 31; green, 47; blue, 255 }  ,opacity=1 ]  {$1$};
\draw (351.5,119.4) node [anchor=north west][inner sep=0.75pt]  [font=\tiny]  {$\dotsc $};
\draw (338,106.4) node [anchor=north west][inner sep=0.75pt]  [font=\tiny,color={rgb, 255:red, 31; green, 47; blue, 255 }  ,opacity=1 ]  {$\alpha -1$};

\end{tikzpicture}}
    \caption{The active edge with joins and cuts case of transitioning between the ramification indices $(\alpha)$ and $(\alpha-1,1)$. For simplicity of the figure this shows when there is one cut, for $c$ cuts the bijection works the same just the degree of the active edge after the cuts and joins becomes $d_0+\alpha-2-c$.}
    \label{fig:cut+joins+active}
\end{figure}

\begin{figure}
    \centering
     \resizebox{.6\textwidth}{!}{\tikzset{every picture/.style={line width=0.75pt}} 

\begin{tikzpicture}[x=0.75pt,y=0.75pt,yscale=-1,xscale=1]

\draw    (57.67,82.33) -- (181,82) ;
\draw    (132.93,131.65) .. controls (94.95,129.4) and (77.15,97.29) .. (75.25,82.27) ;
\draw    (132.93,131.65) -- (175.83,146.92) ;
\draw    (132.93,131.65) -- (175.03,116.92) ;
\draw    (132.93,131.65) -- (174.89,137.92) ;
\draw    (130,82.33) -- (121.5,69.33) ;
\draw    (130,82.33) -- (141.5,71.33) ;
\draw [color={rgb, 255:red, 208; green, 2; blue, 27 }  ,draw opacity=1 ][line width=1.5]    (132.93,131.65) -- (133,117.33) ;
\draw    (60,211) -- (179.67,211.33) ;
\draw [color={rgb, 255:red, 208; green, 2; blue, 27 }  ,draw opacity=1 ][line width=1.5]    (135,211) -- (135,195) ;
\draw    (75,210.33) -- (75,228.33) ;
\draw    (129.67,142) -- (129.67,167) ;
\draw [shift={(129.67,169)}, rotate = 270] [color={rgb, 255:red, 0; green, 0; blue, 0 }  ][line width=0.75]    (10.93,-3.29) .. controls (6.95,-1.4) and (3.31,-0.3) .. (0,0) .. controls (3.31,0.3) and (6.95,1.4) .. (10.93,3.29)   ;
\draw    (251,83.33) -- (374.33,83) ;
\draw    (326.27,132.65) .. controls (288.28,130.4) and (270.49,98.29) .. (268.58,83.27) ;
\draw    (326.27,132.65) -- (369.17,147.92) ;
\draw    (326.27,132.65) -- (368.36,117.92) ;
\draw    (326.27,132.65) -- (368.23,138.92) ;
\draw    (323.33,83.33) -- (314.83,70.33) ;
\draw    (323.33,83.33) -- (334.83,72.33) ;
\draw [color={rgb, 255:red, 208; green, 2; blue, 27 }  ,draw opacity=1 ][line width=1.5]    (326.27,132.65) -- (326.33,118.33) ;
\draw    (253.33,212) -- (373,212.33) ;
\draw [color={rgb, 255:red, 208; green, 2; blue, 27 }  ,draw opacity=1 ][line width=1.5]    (328.33,212) -- (328.33,196) ;
\draw    (268.33,211.33) -- (268.33,229.33) ;
\draw    (319.67,141) -- (319.67,166) ;
\draw [shift={(319.67,168)}, rotate = 270] [color={rgb, 255:red, 0; green, 0; blue, 0 }  ][line width=0.75]    (10.93,-3.29) .. controls (6.95,-1.4) and (3.31,-0.3) .. (0,0) .. controls (3.31,0.3) and (6.95,1.4) .. (10.93,3.29)   ;
\draw    (370.33,108) .. controls (332.35,105.75) and (294.49,97.62) .. (292.58,82.61) ;
\draw [color={rgb, 255:red, 208; green, 2; blue, 27 }  ,draw opacity=1 ][line width=1.5]    (324.27,101.31) -- (324.33,90) ;
\draw    (294.33,212) -- (294.33,230) ;
\draw    (205,143) -- (240,143) ;
\draw [shift={(242,143)}, rotate = 180] [color={rgb, 255:red, 0; green, 0; blue, 0 }  ][line width=0.75]    (10.93,-3.29) .. controls (6.95,-1.4) and (3.31,-0.3) .. (0,0) .. controls (3.31,0.3) and (6.95,1.4) .. (10.93,3.29)   ;
\draw    (242,143) -- (205,143) ;
\draw    (242,143) -- (207,143) ;
\draw [shift={(205,143)}, rotate = 360] [color={rgb, 255:red, 0; green, 0; blue, 0 }  ][line width=0.75]    (10.93,-3.29) .. controls (6.95,-1.4) and (3.31,-0.3) .. (0,0) .. controls (3.31,0.3) and (6.95,1.4) .. (10.93,3.29)   ;

\draw (173.86,143.51) node [anchor=north west][inner sep=0.75pt]  [font=\small,rotate=-180.37]  {$\vdots $};
\draw (59.67,70.73) node [anchor=north west][inner sep=0.75pt]  [font=\tiny,color={rgb, 255:red, 31; green, 47; blue, 255 }  ,opacity=1 ]  {$d_{0}$};
\draw (166.67,71.07) node [anchor=north west][inner sep=0.75pt]  [font=\tiny,color={rgb, 255:red, 31; green, 47; blue, 255 }  ,opacity=1 ]  {$d_{0} -\alpha $};
\draw (128.33,105.73) node [anchor=north west][inner sep=0.75pt]  [font=\tiny,color={rgb, 255:red, 31; green, 47; blue, 255 }  ,opacity=1 ]  {$\alpha $};
\draw (123,68.73) node [anchor=north west][inner sep=0.75pt]  [font=\tiny]  {$\dotsc $};
\draw (72,232.73) node [anchor=north west][inner sep=0.75pt]  [font=\tiny,color={rgb, 255:red, 31; green, 47; blue, 255 }  ,opacity=1 ]  {$2$};
\draw (129,182.4) node [anchor=north west][inner sep=0.75pt]  [font=\tiny,color={rgb, 255:red, 31; green, 47; blue, 255 }  ,opacity=1 ]  {$( \alpha )$};
\draw (367.19,144.51) node [anchor=north west][inner sep=0.75pt]  [font=\small,rotate=-180.37]  {$\vdots $};
\draw (253,71.73) node [anchor=north west][inner sep=0.75pt]  [font=\tiny,color={rgb, 255:red, 31; green, 47; blue, 255 }  ,opacity=1 ]  {$d_{0}$};
\draw (360,72.07) node [anchor=north west][inner sep=0.75pt]  [font=\tiny,color={rgb, 255:red, 31; green, 47; blue, 255 }  ,opacity=1 ]  {$d_{0} -\alpha $};
\draw (319.33,108.07) node [anchor=north west][inner sep=0.75pt]  [font=\tiny,color={rgb, 255:red, 31; green, 47; blue, 255 }  ,opacity=1 ]  {$\alpha -1$};
\draw (316.33,69.73) node [anchor=north west][inner sep=0.75pt]  [font=\tiny]  {$\dotsc $};
\draw (265.33,233.73) node [anchor=north west][inner sep=0.75pt]  [font=\tiny,color={rgb, 255:red, 31; green, 47; blue, 255 }  ,opacity=1 ]  {$2$};
\draw (318.33,182.73) node [anchor=north west][inner sep=0.75pt]  [font=\tiny,color={rgb, 255:red, 31; green, 47; blue, 255 }  ,opacity=1 ]  {$( \alpha -1,1)$};
\draw (91.67,103.07) node [anchor=north west][inner sep=0.75pt]  [font=\tiny,color={rgb, 255:red, 31; green, 47; blue, 255 }  ,opacity=1 ]  {$\alpha $};
\draw (258.33,109.4) node [anchor=north west][inner sep=0.75pt]  [font=\tiny,color={rgb, 255:red, 31; green, 47; blue, 255 }  ,opacity=1 ]  {$\alpha -1$};
\draw (329,87.07) node [anchor=north west][inner sep=0.75pt]  [font=\tiny,color={rgb, 255:red, 31; green, 47; blue, 255 }  ,opacity=1 ]  {$1$};
\draw (291.33,234.4) node [anchor=north west][inner sep=0.75pt]  [font=\tiny,color={rgb, 255:red, 31; green, 47; blue, 255 }  ,opacity=1 ]  {$2$};

\end{tikzpicture}}
    \caption{The all cuts case of transitioning between the ramification indices $(\alpha)$ and $(\alpha-1,1)$.}
    \label{fig:allcuts}
\end{figure}

\begin{figure}
    \centering
     \resizebox{.6\textwidth}{!}{\tikzset{every picture/.style={line width=0.75pt}} 

\begin{tikzpicture}[x=0.75pt,y=0.75pt,yscale=-1,xscale=1]

\draw    (57.67,100.33) -- (181,100) ;
\draw    (130,100.33) -- (172.83,116.42) ;
\draw    (130,100.33) -- (175.82,102.19) ;
\draw    (130,100.33) -- (126.25,86.5) ;
\draw    (130,100.33) -- (141.5,89.33) ;
\draw [color={rgb, 255:red, 208; green, 2; blue, 27 }  ,draw opacity=1 ][line width=1.5]    (130,100.33) -- (120.24,89.86) ;
\draw    (60,211) -- (179.67,211.33) ;
\draw [color={rgb, 255:red, 208; green, 2; blue, 27 }  ,draw opacity=1 ][line width=1.5]    (135,211) -- (135,195) ;
\draw    (129.67,142) -- (129.67,167) ;
\draw [shift={(129.67,169)}, rotate = 270] [color={rgb, 255:red, 0; green, 0; blue, 0 }  ][line width=0.75]    (10.93,-3.29) .. controls (6.95,-1.4) and (3.31,-0.3) .. (0,0) .. controls (3.31,0.3) and (6.95,1.4) .. (10.93,3.29)   ;
\draw    (251,91.33) -- (374.33,91) ;
\draw    (326.27,140.65) .. controls (288.28,138.4) and (270.49,106.29) .. (268.58,91.27) ;
\draw    (326.27,140.65) -- (369.17,155.92) ;
\draw    (326.27,140.65) -- (368.36,125.92) ;
\draw    (326.27,140.65) -- (368.23,146.92) ;
\draw    (323.33,91.33) -- (327.25,78) ;
\draw    (323.33,91.33) -- (334.83,80.33) ;
\draw [color={rgb, 255:red, 208; green, 2; blue, 27 }  ,draw opacity=1 ][line width=1.5]    (326.27,140.65) -- (326.33,126.33) ;
\draw    (253.33,212) -- (373,212.33) ;
\draw [color={rgb, 255:red, 208; green, 2; blue, 27 }  ,draw opacity=1 ][line width=1.5]    (328.33,212) -- (328.33,196) ;
\draw    (268.33,211.33) -- (268.33,229.33) ;
\draw    (319.67,149) -- (319.67,174) ;
\draw [shift={(319.67,176)}, rotate = 270] [color={rgb, 255:red, 0; green, 0; blue, 0 }  ][line width=0.75]    (10.93,-3.29) .. controls (6.95,-1.4) and (3.31,-0.3) .. (0,0) .. controls (3.31,0.3) and (6.95,1.4) .. (10.93,3.29)   ;
\draw [color={rgb, 255:red, 208; green, 2; blue, 27 }  ,draw opacity=1 ][line width=1.5]    (323.33,91.33) -- (319.25,77.5) ;
\draw    (205,143) -- (240,143) ;
\draw [shift={(242,143)}, rotate = 180] [color={rgb, 255:red, 0; green, 0; blue, 0 }  ][line width=0.75]    (10.93,-3.29) .. controls (6.95,-1.4) and (3.31,-0.3) .. (0,0) .. controls (3.31,0.3) and (6.95,1.4) .. (10.93,3.29)   ;
\draw    (242,143) -- (205,143) ;
\draw    (242,143) -- (207,143) ;
\draw [shift={(205,143)}, rotate = 360] [color={rgb, 255:red, 0; green, 0; blue, 0 }  ][line width=0.75]    (10.93,-3.29) .. controls (6.95,-1.4) and (3.31,-0.3) .. (0,0) .. controls (3.31,0.3) and (6.95,1.4) .. (10.93,3.29)   ;

\draw (175.36,123.01) node [anchor=north west][inner sep=0.75pt]  [font=\tiny,rotate=-180.32]  {$\vdots $};
\draw (59.67,88.73) node [anchor=north west][inner sep=0.75pt]  [font=\tiny,color={rgb, 255:red, 31; green, 47; blue, 255 }  ,opacity=1 ]  {$d_{0}$};
\draw (157.67,88.07) node [anchor=north west][inner sep=0.75pt]  [font=\tiny,color={rgb, 255:red, 31; green, 47; blue, 255 }  ,opacity=1 ]  {$d_{0} -\alpha +1$};
\draw (110.7,79.9) node [anchor=north west][inner sep=0.75pt]  [font=\tiny,color={rgb, 255:red, 31; green, 47; blue, 255 }  ,opacity=1 ,rotate=-359.59]  {$\alpha $};
\draw (128.5,87.73) node [anchor=north west][inner sep=0.75pt]  [font=\tiny]  {$\dotsc $};
\draw (129,182.4) node [anchor=north west][inner sep=0.75pt]  [font=\tiny,color={rgb, 255:red, 31; green, 47; blue, 255 }  ,opacity=1 ]  {$( \alpha )$};
\draw (367.19,152.51) node [anchor=north west][inner sep=0.75pt]  [font=\small,rotate=-180.37]  {$\vdots $};
\draw (253,79.73) node [anchor=north west][inner sep=0.75pt]  [font=\tiny,color={rgb, 255:red, 31; green, 47; blue, 255 }  ,opacity=1 ]  {$d_{0}$};
\draw (360,80.07) node [anchor=north west][inner sep=0.75pt]  [font=\tiny,color={rgb, 255:red, 31; green, 47; blue, 255 }  ,opacity=1 ]  {$d_{0} -\alpha +1$};
\draw (319.33,116.07) node [anchor=north west][inner sep=0.75pt]  [font=\tiny,color={rgb, 255:red, 31; green, 47; blue, 255 }  ,opacity=1 ]  {$\alpha -1$};
\draw (265.33,233.73) node [anchor=north west][inner sep=0.75pt]  [font=\tiny,color={rgb, 255:red, 31; green, 47; blue, 255 }  ,opacity=1 ]  {$2$};
\draw (318.33,182.73) node [anchor=north west][inner sep=0.75pt]  [font=\tiny,color={rgb, 255:red, 31; green, 47; blue, 255 }  ,opacity=1 ]  {$( \alpha -1,1)$};
\draw (258.33,120.4) node [anchor=north west][inner sep=0.75pt]  [font=\tiny,color={rgb, 255:red, 31; green, 47; blue, 255 }  ,opacity=1 ]  {$\alpha -1$};
\draw (307,69.57) node [anchor=north west][inner sep=0.75pt]  [font=\tiny,color={rgb, 255:red, 31; green, 47; blue, 255 }  ,opacity=1 ]  {$1$};
\draw (195,106.4) node [anchor=north west][inner sep=0.75pt]  [font=\tiny,color={rgb, 255:red, 31; green, 47; blue, 255 }  ,opacity=1 ]  {$\alpha -1$};
\draw (191.16,123.53) node [anchor=north west][inner sep=0.75pt]  [font=\Large,rotate=-179.5]  {$\{$};
\draw (386,140.4) node [anchor=north west][inner sep=0.75pt]  [font=\tiny,color={rgb, 255:red, 31; green, 47; blue, 255 }  ,opacity=1 ]  {$\alpha -1$};
\draw (384.16,157.53) node [anchor=north west][inner sep=0.75pt]  [font=\LARGE,rotate=-179.5]  {$\{$};

\end{tikzpicture}}
    \caption{The cuts and the active edge case of transitioning between the ramification indices $(\alpha)$ and $(\alpha-1,1)$.}
    \label{fig:allcuts+activeedge}
\end{figure}

\subsubsection*{Multiplicity}
We now show that the bijection constructed above preserves multiplicity. In particular, each
pair of covers related by the local modification contributing with ramification profiles
$(\alpha)$ and $(\alpha-1,1)$ have the same local degree with respect to the morphism
$\Fs \times \Ft$. Consequently, the tropical generalized Tevelev degree is unchanged. Recall from Equation \eqref{eq:locdegforlater} that the local degree is given by a product of four factors: automorphism factor, product of local Hurwitz numbers, determinant of the matrix representing the map $\Fs\times\Ft$, and a product of edge lengths divided by their least common multiple. 

There are no new automorphism factors produced in any of changes discussed in the previous section. Therefore, the automorphism factor for each cover remains equal to $2^g$ coming from the pair of unmarked simple ramification points corresponding to each genus, as in \cite{troptev}. 

Previously, graphs constructed contain 2 types of local Hurwitz numbers: either \\$H_0(\alpha, (2,1^{d-2}), \beta)$, or $H_0((d),(d), 1)$, where the third point on the base is not a branch point, but only one of its inverse images is marked. Both of these types are always equal to 1, as shown in \cite{troptev}. In this section, different types of local Hurwitz numbers have been constructed. 
\begin{enumerate}
    \item All joins or all cuts cases: $H_0((\alpha),(\alpha), 1^\alpha)$ appears where none of the third point's inverse images are marked. This local Hurwitz number is equal $\frac{1}{\alpha}$.
    \item Active edge with joins or cuts cases: the new local Hurwitz number type is \\ $H_0((D),(d_0,(1)^{D-d_0}),(\alpha,(1)^{D-\alpha}))$. To calculate this we need $\sigma_1\sigma_2\sigma_3=id \in S_D$. Since $\sigma_1$ is a full $D$ cycle, it can be mapped to the cycle $(12\dots d)$. $\sigma_3$ is uniquely determined by $(\sigma_1\sigma_2)^{-1}$ when $(\sigma_1\sigma_2)$ is an $\alpha$-cycle. Due to $\sigma_1$ being a full cycle, any valid $\sigma_2$ can map to another valid $\sigma_2$, therefore there is 1 orbit of valid triples and thus the local Hurwitz number is 1.
    \item Active edge with joins and cuts cases: Let \(c\) and \(j\) denote the numbers of cuts and joins, respectively,
and let \(D\) denote the local degree at the relevant source vertex.
If \(d_0\) is the expansion factor of the incoming active edge from the left, then
\[
D=d_0+j
\qquad\text{and}\qquad
\alpha=c+j+1.
\]
The three local ramification profiles are
\[
(D-c,1^c),\qquad
(d_0,1^j),\qquad
(\alpha,1^{D-\alpha}).
\]
The nontrivial parts satisfy
\[
(D-c)+d_0+\alpha=2D+1.
\]
Therefore, this is a genus zero three-point pure-cycle Hurwitz number,
and
\[
H_0\left(
(D-c,1^c),
(d_0,1^j),
(\alpha,1^{D-\alpha})
\right)=1
\]
by \cite[Lemma~2.1]{LiuOsserman}. Note that the unique cover has no nontrivial automorphisms so the automorphism-weighted Hurwitz number is also $1$.
\end{enumerate}   
Overall, the product of local Hurwitz numbers only differs from $1$ when the local Hurwitz number type $H_0((\alpha),(\alpha), 1^\alpha)$ occurs, which introduces a factor of $\frac{1}{\alpha}$ to the product. 

The determinant factor is a dilation factor. Recall that $x_i, L_j$  are the lengths of the edges of the graphs $\overline{\Gamma}, \overline{T}$, and $y_k$ the lengths of $5g$ edges of $\Gamma$ then the rows of the matrix express the $y_k$'s as linear functions of $x_i, L_j$. Since all covers can be split into a genus part  and a marked tree part, the matrix used to calculate the dilation factor is block diagonal. It is shown in \cite{troptev} that the block corresponding to the genus part has determinant $2^g$. To calculate the determinant of the marked tree part, recall two facts: all marked points stabilize to the active path, so the rows corresponding to $x_i$'s in the marked tree part contain exactly one non-zero entry; and when writing the lengths $L_i$'s in terms of the lengths $y_j$'s, we observe that there is exactly one length of the cover that contributes to $L_i$ and does not lie on the active path. The authors in \cite{troptev} prove that these matrix entries are always one, we now examine what happens in the new cases. 
\begin{enumerate}
    \item All joins or all cuts cases: as shown in figures \ref{fig:alljoins} and \ref{fig:allcuts}, there exist 2 edge weights not equal to 1 covering the same edge. Let $b$ be the weight of the active edge and $a$ be the weight of the other edge. To obtain the lattice basis, the non-zero entry corresponding to the appropriate $x_i$ is $\frac{\lcm(a,b)}{b}$. Note that this not equal to $1$ when $a$ and $b$ are coprime.
    \item Active edge with joins or cuts cases: there exists a unique edge of weight not equal to 1, therefore, the nonzero entry in the dilation matrix is 1.
    \item Active edge with joins and cuts cases: same as previous case.
\end{enumerate}

The final factor of the local degree is \[\prod_{e\in CE(T)}\frac{\prod_{\phi(e') = e}m_{e'}}{M_e}\]
where $CE(T)$ denotes the set of compact edges of $T$ and for $e$ any compact edge of $T$, $M_e:=\lcm(\{m_{e'}| e'\in \Gamma, \phi(e')=e\})$, that was previously equal to $1$. This is not equal to $1$ only when there are multiple edges with degree not equal to one covering a compact edge. This only occurs in the case of all cuts or all joins. In this case we have degrees $a$ and $b$ as before and the factor becomes $\frac{a\cdot b}{\lcm(a,b)}$. 

In conclusion, we take the product of all four factors. For the cases involving the active edge, this is a product of $\frac{1}{2^g}\cdot2^g=1$. For the cases of all cuts or all joins let $a$ be the ramification of the marked preimage point of highest ramification in the section of the cover focused on: $\frac{1}{2^g}\cdot \frac{1}{a}\cdot \frac{\lcm(a,b)}{b}\cdot 2^g \cdot  \frac{a\cdot b}{\lcm(a,b)}=1$. Therefore, the local modification leaves all vertex multiplicities unchanged away from the
active edge and preserves the determinant contribution along the active edge, and thus the
bijection preserves multiplicities. This proves Proposition \ref{prop:splitting-part} and thus completes the proof of Theorem \ref{thm:higherram}.

\bibliographystyle{alpha}
\bibliography{lib}

\appendix
\section{Comparison with previously known algebraic formulas}\label{appendix}
\begin{proposition}\label{appendixprop1}
    The formula in Theorem \ref{thm:ellnegative} is equivalent to the $r=1$ case of \cite[Theorem 6]{CPS:Tevdeg}.
\end{proposition}
\begin{proof}
Theorem \ref{thm:ellnegative} states
        \[\Tev^\trop_{g,\ell}= 2^g - \sum_{i=0}^{-\ell-1}\bigg(g-2i+1\bigg)\bigg({g\choose i}-{g \choose i-1}\bigg).\]

        Summation by parts gives that \[\sum_{i=0}^{-\ell-1}\bigg(g-2i+1\bigg)\bigg({g\choose i}-{g \choose i-1}\bigg)=(g+2\ell+3){g\choose -\ell-1}+2\sum_{i=0}^{-\ell-2}{g\choose i}.\]
Next, we use the fact that \[m {g\choose m}= (g-m+1){g\choose m-1}\] and we rewrite the first term as follows \[(g+2\ell+3){g\choose -\ell-1}=-\ell {g\choose -\ell}-(-\ell-2){g\choose -\ell-1}.\]
All together, \[\Tev^\trop_{g,\ell}= 2^g -2\sum_{i=0}^{-\ell-2}{g\choose i}+\ell {g\choose -\ell}+(-\ell-2){g\choose -\ell-1}\]
which is exactly the $r=1$ case of \cite[Theorem 6]{CPS:Tevdeg}.
\end{proof}

\begin{proposition}
    The formula is Theorem \ref{thm:higherram} is equivalent to the formula of \cite[Theorem 9]{GenTev}.
\end{proposition}
\begin{proof} Theorem \ref{thm:higherram} states
 \begin{align*}
        \Tev^\trop_{g,\ell,\mu_1,\dots,\mu_k}=&2^g - \sum_{i=0}^{-\ell-1}\bigg(g-2i+1\bigg)\bigg({g\choose i}-{g \choose i-1}\bigg)-\sum_{i=-\ell}^{\lvert\mu_1\rvert-\ell-2}{g\choose i}-\sum_{i=-\ell}^{\lvert\mu_k\rvert-\ell-2}{g\choose i}\\&+\bigg(\lvert\mu_1\rvert+\lvert\mu_k\rvert-2\bigg){g\choose-\ell-1}-\sum_{h=2}^{k-1}\sum_{i=-\ell}^{\lvert\mu_h\rvert-\ell-2}\bigg(\lvert\mu_h\rvert-i-\ell-1\bigg)\bigg({g\choose i}-{g \choose i-1}\bigg).
    \end{align*}

We focus on \begin{align}\label{prop2eq1}
    & -\sum_{i=-\ell}^{\lvert\mu_1\rvert-\ell-2}{g\choose i}-\sum_{i=-\ell}^{\lvert\mu_k\rvert-\ell-2}{g\choose i}+\bigg(\lvert\mu_1\rvert+\lvert\mu_k\rvert-2\bigg){g\choose-\ell-1}\nonumber\\&-\sum_{h=2}^{k-1}\sum_{i=-\ell}^{\lvert\mu_h\rvert-\ell-2}\bigg(\lvert\mu_h\rvert-i-\ell-1\bigg)\bigg({g\choose i}-{g \choose i-1}\bigg)
\end{align}
    because Proposition \ref{appendixprop1} tells us what the first two terms are equivalent to. Summation by parts gives the following: \begin{equation*}
        \sum_{i=-\ell}^{\lvert\mu_h\rvert-\ell-2}\bigg(\lvert\mu_h\rvert-i-\ell-1\bigg)\bigg({g\choose i}-{g \choose i-1}\bigg)=\sum_{i=-\ell}^{\lvert\mu_h\rvert-\ell-2}{g\choose i}-(\lvert\mu_h\rvert-1){g \choose -\ell-1}.
    \end{equation*}
    After substituting into Equation \eqref{prop2eq1}, the expression becomes 
    \begin{align}\label{prop2eq2}
    -\sum_{i=-\ell}^{\lvert\mu_1\rvert-\ell-2}{g\choose i}-\sum_{i=-\ell}^{\lvert\mu_k\rvert-\ell-2}{g\choose i}+\bigg(\lvert\mu_1\rvert+\lvert\mu_k\rvert-2\bigg){g\choose-\ell-1}\nonumber\\+ \sum_{h=2}^{k-1}
    \bigg(\lvert\mu_h\rvert-1\bigg){g \choose -\ell-1}-\sum_{h=2}^{k-1}\sum_{i=-\ell}^{\lvert\mu_h\rvert-\ell-2}{g\choose i}
    \nonumber\\
    =\sum_{h=1}^{k}
    \bigg(\lvert\mu_h\rvert-1\bigg){g \choose -\ell-1}-\sum_{h=1}^{k}\sum_{i=-\ell}^{\lvert\mu_h\rvert-\ell-2}{g\choose i}
\end{align}
Looking at the last term, ${g\choose -\ell}$ occurs in the summation only when $\lvert\mu_h\rvert\geq2$ or equivalently, when $\lvert\mu_h\rvert\neq1$. This can be represented by \[k-\sum_{h=1}^k\delta_{\lvert\mu_h\rvert,1}\]
where $\delta_{\lvert\mu_h\rvert,1}$ is the Kronecker delta and equals 1 when $\lvert\mu_h\rvert=1$ and zero otherwise. We now pull the ${g\choose -\ell}$ term out of the double sum and rewrite Equation \eqref{prop2eq2}:
\begin{equation*}
    \sum_{h=1}^{k}
    \bigg(\lvert\mu_h\rvert-1\bigg){g \choose -\ell-1}-\bigg(k-\sum_{h=1}^k\delta_{\lvert\mu_h\rvert,1}\bigg){g\choose -\ell}-\sum_{h=1}^{k}\sum_{i=-\ell+1}^{\lvert\mu_h\rvert-\ell-2}{g\choose i}.
\end{equation*}
Putting this together with the result from Proposition \ref{appendixprop1} we obtain
\begin{align*}
    \Tev^\trop_{g,\ell,\mu_1,\dots,\mu_k}=&2^g-2\sum_{i=0}^{-\ell-2}{g\choose i}+
    \bigg(\sum_{h=1}^{k}\lvert\mu_h\rvert-k-\ell-2\bigg){g \choose -\ell-1}\\&+\bigg(\ell-k+\sum_{h=1}^k\delta_{\lvert\mu_h\rvert,1}\bigg){g\choose -\ell}-\sum_{h=1}^{k}\sum_{i=-\ell+1}^{\lvert\mu_h\rvert-\ell-2}{g\choose i}
\end{align*}
which is equivalent to \cite[Theorem 9]{GenTev}
\end{proof}

\end{document}